\documentclass[11pt]{amsart}

\usepackage{amssymb, latexsym}

\usepackage{graphicx, color, verbatim, ulem}

\usepackage{hyperref}

\begin{document}

\newcommand{\tv}{\tilde{V}}
\newcommand{\tu}{\tilde{U}}
\newcommand{\Mne}{M_{\perp, \e}}

\definecolor{darkgreen}{rgb}{0,0.55,0}
\newcommand{\blue}[1]{{\textcolor{blue}{#1}}} 
\newcommand{\red}[1]{{\textcolor{red}{#1}}} 
\newcommand{\green}[1]{{\bf\textcolor{darkgreen}{#1}}} 

\newcommand{\redi}{\color{red}}
\newcommand{\barint}{\overline{\hspace{.65em}}\!\!\!\!\!\!\int}
\newcommand{\e}{\varepsilon}
\newcommand{\eps}{\varepsilon}
\newcommand{\logeps}{|\log \eps|}
\newcommand{\ue}{u_\eps}
\newcommand{\ve}{v_\eps}
\newcommand{\jep}{j_\eps}
\newcommand{\loc}{ {\mbox{\scriptsize{loc}}} }
\newcommand{\R}{{\mathbb R}}
\newcommand{\RR}{{\mathbb R}}
\newcommand{\C}{{\mathbb C}}
\newcommand{\SSS}{{\mathbb S}}
\newcommand{\J}{{\mathbb J}}
\newcommand{\T}{{\mathbb T}}
\newcommand{\Z}{{\mathbb Z}}
\newcommand{\N}{{\mathbb N}}
\newcommand{\Q}{{\mathbb Q}}
\newcommand{\Hdf}{{\mathcal{H}}}

\newcommand{\h}{\mathcal{H}}
\newcommand{\calD}{{\mathcal{D}}}
\newcommand{\calE}{{\mathcal{E}}}
\newcommand{\calS}{{\mathcal{S}}}
\newcommand{\calL}{{\mathcal{L}}}
\newcommand{\calX}{{\mathcal{X}}}
\newcommand{\calF}{{\mathcal{F}}}
\newcommand{\calG}{{\mathcal{G}}}
\newcommand{\calH}{{\mathcal{H}}}
\newcommand{\calJ}{{\mathcal{J}}}
\newcommand{\calT}{{\mathcal{T}}}
\newcommand{\calO}{{\mathcal{O}}}
\newcommand{\calU}{{\mathcal{U}}}
\newcommand{\calM}{{\mathcal{M}}}
\newcommand{\dist}{\operatorname{dist}}
\newcommand{\sign}{\operatorname{sign}}
\newcommand{\spt}{\operatorname{spt}}
\def\rest{\hskip 1pt{\hbox to 10.8pt{\hfill
\vrule height 7pt width 0.4pt depth 0pt\hbox{\vrule height 0.4pt
width 7.6pt depth 0pt}\hfill}}}
\newcommand{\bd}{\partial}
\newcommand{\vp}{\varphi}

\newcommand{\dbar}{\bar{D}}
\newcommand{\dd}{D}

\newcommand{\vol}{\,\mbox{vol}}

\newcommand{\bn}[1]{ ({\texttt{#1}})}

\newcommand{\beq}{\begin{equation}}
\newcommand{\eeq}{\end{equation}}

\newcommand{\bea}{\begin{align*}}
\newcommand{\nd}{\noindent}

\newcommand{\Mn}{M_{\perp}}

\def\logeps{{|\!\log\eps|}}

\theoremstyle{plain}

\newtheorem{theorem}{Theorem}[section]
\newtheorem{lemma}[theorem]{Lemma}
\newtheorem{propo}[theorem]{Proposition}
\newtheorem{coro}[theorem]{Corollary}

\theoremstyle{definition}
\newtheorem{definition}{Definition}

\theoremstyle{remark}
\newtheorem{remark}[theorem]{Remark}
\newtheorem{example}[theorem]{Example}
\newtheorem{warning}{Warning}


\title[Renormalized energy between vortices on Riemannian surfaces]{Renormalized energy between vortices in some Ginzburg-Landau models on 2-dimensional Riemannian manifolds }

\maketitle

\begin{center}
{R. Ignat and R.L. Jerrard}
\end{center}

\begin{center}
\today
\end{center}

\bigskip

\begin{abstract}
We study a variational Ginzburg-Landau type model depending on a small parameter $\eps>0$ for (tangent) vector fields on a $2$-dimensional Riemannian manifold $S$. As $\eps\to 0$, these vector fields tend to have unit length so they generate singular points, called vortices, of a (non-zero) index  if the genus $\mathfrak{g}$ of $S$ is different than $1$. Our first main result concerns the characterization of canonical harmonic unit vector fields with prescribed singular points and indices. The novelty of this classification involves flux integrals constrained to a particular {\it vorticity-dependent} lattice in the $2\mathfrak{g}$-dimensional space of harmonic $1$-forms on $S$ if $\mathfrak{g}\geq 1$. Our second main result determines the interaction energy (called renormalized energy) between vortex points as a $\Gamma$-limit (at the second order) as $\eps\to 0$. The renormalized energy governing the optimal location of vortices depends on the Gauss curvature of $S$ as well as on the quantized flux. The coupling between flux quantization constraints and vorticity, and its impact on the renormalized energy, are new phenomena in the theory of Ginzburg-Landau type models. We also extend this study to two other (extrinsic) models for embedded hypersurfaces $S\subset \R^3$, in particular, to a physical model for non-tangent maps to $S$ coming from micromagnetics.
\end{abstract}

\tableofcontents

\section{Introduction}
\label{sec:intro}

 We consider three related  asymptotic variational problems similar to the Ginzburg-Landau model that are described by singularly perturbed functionals depending on a small parameter $\eps>0$. These functionals are defined for smooth vector fields on a $2$-dimensional compact Riemannian manifold $S$ (or otherwise, for embedded surfaces, 
	we consider smooth maps whose non-tangential component is strongly penalized). As $\eps\to 0$, we expect that these maps generate point singularities, called vortices, carrying a topological degree (or index).
In every case, our goal is to characterize the limit of minimizers of these functionals as $\e \to 0$,
or more generally, to prove a $\Gamma$-convergence result at second order that
 captures a ``renormalized energy" between the vortex singularities and identifies a ``canonical harmonic unit vector field" associated to these vortices. 
 
We classify all harmonic unit vector fields with singularities at prescribed vortex points with prescribed indices (satisfying a certain constraint coming from the topology of $S$). The subtlety  for surfaces of genus $\mathfrak{g}\geq 1$ is that a harmonic unit vector field depends not only on the prescribed vortex points with their topological degrees, but on some {\it flux integrals} constrained to belong to a particular vorticity-dependent lattice in the $2\mathfrak{g}$-dimensional space of harmonic $1$-forms on $S$. The renormalized energy associated to a configuration of vortices depends on vortex interaction (mediated by the Green's and Robin's functions for the Laplacian on $S$), a term arising from the Gaussian curvature of $S$, and the flux integrals. The dependence on vortex position and degree of the flux constraints, and through them the renormalized energy, constitutes a new phenomenon in the theory of Ginzburg-Landau type models.

\bigskip

\subsection{Three models}
We will always assume that the potential $F:\R_+\to \R_+$ is a continuous function such that there exists some $C>0$ with
\begin{equation}
F(1) = 0, \qquad F(s^2) \ge C(1- s)^2, \textrm{ for all } s>0.
\label{F.growth}\end{equation}

\bigskip

{\bf Problem 1}: Let $(S,g)$ be  a closed (i.e., compact, connected without boundary) oriented $2$-dimensional Riemannian manifold of genus $\mathfrak g$.
Consider (tangent) vector fields \footnote{{ In the sequel, a vector field on $S$ is always tangent at $S$ (the standard definition in differential geometry).}}
$$u:S\to TS, \quad \textrm{ i.e., } \,  u(x)\in T_x S \, \textrm{ for every } \, x\in S$$ where $TS=\cup_{x\in S} T_xS$ is the tangent bundle of $S$,
and minimize
the {\em intrinsic} energy 
\beq
\label{eepin}
E^{in}_\e(u) 
\  = \ 
\int_S e_\e^{in}(u) \vol_g, \quad e_\e^{in}(u):=\frac 12|\dd u|_g^2  +\frac 1{4\e^2} F(|u|^2_g).
\eeq
Here, $\vol_g$ is the volume $2$-form on $S$, $|v|_g$ is the length of a vector field $v$ { with respect to (w.r.t.)} the metric $g$ and
\[
|\dd u|_g^2(x) := |\dd_{\tau_1}u|_g^2(x) + |\dd_{\tau_2}u|_g^2(x)
\]
where $\dd_v u$ denotes covariant differentiation (with respect to the Levi-Civita connection) of
$u$  (in direction $v$) and $\{ \tau_1,\tau_2\}$ is any orthonormal basis for $T_xS$.

\bigskip

{\bf Problem 2}: Let $(S,g)$ be a closed oriented $2$-dimensional Riemannian manifold {\it isometrically embedded} in $(\R^3, \bar g)$.
To simplify the notation, we will still denote by $g$ the metric  $\bar g$ on $\R^3$,  which in
applications is typically the Euclidean metric. Consider sections $m$ of the tangent bundle $TS$ (i.e.,  $m(x) \in T_xS$ for a.e. $x\in S$), and minimize
the {\em extrinsic} energy
\[
E^{ex}_\e(m) \  = \ \int_S e_\e^{ex}(m) \vol_g, \quad e_\e^{ex}(m):=
\frac 12 |\dbar m|_g^2  +\frac 1{4\e^2} F(|m|_g^2). 
\]
That is, $|\cdot|_g$ denotes the length in the metric $g$ on $\R^3$
and $$|\dbar  m|_g^2 := |\dbar_{\tau_1}  \bar m|_g^2 + |\dbar_{\tau_2}  \bar m|_g^2,$$ where $\bar m$ is
an extension of $m$ to a neighborhood of $S$, $\{\tau_1(x), \tau_2(x)\}$ form a basis for $T_xS$,
and $\dbar_v$ denotes covariant derivative (with respect to the Levi-Civita connection) in $(\R^3, g)$ in the $v$ direction.
As is well known, $|\dbar m|_g^2$ is independent of the choice of extension $\bar m$. The difference between $|\dbar  m|_g^2$ in 
$e_\e^{ex}(m)$ and $|\dd m|_g^2$ in $e_\e^{in}(m)$ consists in the normal component $|\dbar  m\cdot N|_g^2$ of the full differential $\dbar m$ (the so called shape operator, see \eqref{shape} and Lemma \ref{lem:ext_Dir} below) where $N$ is the Gauss map at $S$. Problem 2 is relevant
to liquid crystals, as a relaxation of the model proposed in \cite{NapVerg1, NapVerg2}
and studied (for the torus) in \cite{SSV}.

\bigskip

{\bf Problem 3}: Let $(S,g)$ be a closed oriented $2$-dimensional Riemannian manifold {\it isometrically
embedded} in $\R^3$ (that is endowed with the Euclidean metric). Consider maps $M:S\to \R^3$ with $|M| =  1$ {\em a.e.} (standing for the magnetization), and minimize the micromagnetic energy on $S$:
\[
 E^{mm}_\e(M) \  = \ \ \int_S e_\e^{mm}(M) \vol_g, \quad e_\e^{mm}(M):=
\frac 12|\dbar M|^2  +\frac 1{4\e^2}F\big(1-(M\cdot N)^2\big).
\]
Here  
$|\dbar M|^2 := |\tau_1\cdot  \dbar \bar M |^2 + |\tau_2\cdot \dbar \bar M|^2$, where $|\cdot|$ denotes the euclidean length of a vector in $\R^3$, $\dbar$ is the differential operator in $\R^3$, $\bar M$ is
an extension of $M$ to a neighborhood of $S$ and $\{\tau_1(x), \tau_2(x)\}$ form an orthonormal  basis for $T_xS$ and $N(x)$ is the Gauss map at $S$.
As usual, $|\dbar M|^2$ is independent of the choice of extension $\bar M$. Note that if $M$ is decomposed as 
$$M=m+(M\cdot N) N,$$
where $m$ is the projection of $M$ on the tangent plane $TS$, then the energy $E^{mm}_\e(M)$ can be seen as a nonlinear perturbation of $E^{ex}_\e(m)$  in terms of the tangent component $m$ with  the potential $F(|m|^2)$ since $|m|^2=1-(M\cdot N)^2$ (see Section \ref{sec:11}). 
The above variational problem is a reduced model for thin ferromagnetic films  for the potential $F(s^2)=1-s^2$ for $s\in [0,1]$ (satisfying \eqref{F.growth} with $C=1$) (see Section \ref{sec:micromag}).

\medskip

\subsection{Vortices} Let $(S,g)$ be  a closed oriented $2$-dimensional Riemannian manifold of genus $\mathfrak g$ (not necessarily embedded in $\R^3$). We will identify vortices of a vector field $u$ with small geodesic balls centered at some points around which $u$ has a (non-zero) index. To be more precise, we introduce the Sobolev space (for $p \ge 1$)
\[
\calX^{1,p}(S) := \{ \mbox{vector fields $u:S\to TS$} : \ |u|_g, |Du|_g \in L^p(S)\}.
\]
We will also write $\calX(S)$ to denote the space of smooth vector fields on $S$.
Given $u\in \calX^{1,p}(S)\cap L^q(S)$ such that $\frac 1p + \frac 1q =1$, $p,q\in [1, \infty]$, we define the current $j(u)$ as the following $1$-form:
\beq
\label{j.def}
j(u)= (\dd u , iu)_g,
\eeq
where $(\cdot, \cdot)_g$ is the scalar product on $TS$ (more generally, the inner product associated to $k$-forms, $k=0,1,2$) and $i:TS\to TS$ is an isometry of $T_xS$ to itself for every $x\in S$ satisfying
\beq
\label{isom}
i^2w = -w, \qquad
(iw, v)_g = -(w,iv)_g \  = \ \vol_g(w,v). 
\eeq
In particular, $j(u)$ is a well-defined $1$-form in $L^1(S)$ if $u\in \calX^{1,1}(S)$ with $|u|_g=1$ almost everywhere in $S$.  
To introduce the notion of index, we assume that ${\mathcal O}$ is an open subset of $S$ of Lipschitz boundary and
$u\in \calX^{1,2}({\mathcal N})$ is a vector field in a neighborhood ${\mathcal N}$ of $\partial {\mathcal O}$ such that $|u|_g\ge \frac 14$ a.e. in ${\mathcal N}$;
then the {\em index} (or topological degree) 
of $u$ along $\partial {\mathcal O}$ is defined by
\beq
\label{deg.def}
\deg (u; \partial {\mathcal O}) := \frac 1{2\pi}\left( \int_{\partial {\mathcal O}}\frac{ j(u)}{|u|_g^2} + \int_{\mathcal O} \kappa \,\vol_g
\right),
\eeq
where $\kappa$ is the Gauss curvature on $S$ and the curve $\partial {\mathcal O}$ has the orientation inherited
in the usual way from ${\mathcal O}$ as oriented by the volume form,
so that Stokes' Theorem holds with the standard sign conventions
(see \cite{Car94} Chapter 6.1). 
In particular, if $u$ is smooth enough in ${\mathcal O}$ and has unit length on $\partial {\mathcal O}$, then one has 
$$\deg (u; \partial {\mathcal O})=\frac{1}{2\pi} \int_{{\mathcal O}} \omega(u)$$ where $\omega(u)$ is the {\it vorticity} (as a $2$-form) associated to the vector field $u$:
\begin{equation}
\omega(u) := d  j(u) + \kappa \vol_g,
\label{omega.def}\end{equation}
where $d j(u)$ is the exterior derivative of $j(u)$ (for more details, see Lemma \ref{L.degree} below).
Sometimes we will identify the index of $u$ at a point $P\in S$ with the index of $u$ along a sufficiently small curve around $P$.
Note that every smooth vector field $u\in \calX({\mathcal O})$ (or more generally, $u\in \calX^{1,2}({\mathcal O})$) of unit length in ${\mathcal O}$ has $\deg (u; \partial {\mathcal O})=0$; moreover, a vortex with non-zero index will carry infinite energy in Problems 1, 2 and 3 as $\eps\to 0$.\\

\subsection{Aim} We will prove a $\Gamma$-convergence result (at the second order) for the three energy functionals introduced above, as $\eps\to 0$. The genus $\mathfrak g$ and the Euler characteristic $$\chi(S)=2-2\mathfrak g$$ of $S$ will play an important role. In particular, at the level of minimizers $u_\eps$ of $E^{in}_\eps$, we show that  as $\e\to 0$, $u_\eps$ converges weakly in $\calX^{1,p}(S)$  for $p<2$,  see Theorem \ref{thm:min_in} (for a subsequence) to a canonical harmonic vector field $u^*$ of unit length that is smooth \footnote{In the case of a surface $(S,g)$ with genus $1$ (i.e., homeomorphic with the flat torus), then $n=0$ and $u^*$ is smooth in $S$. } away from $n=|\chi(S)|$ distinct singular points $a_1, \dots, a_n$, each 
singular point $a_k$ carrying the same index $d_k = \sign \chi(S)$ for $k=1, \dots, n$ so that \footnote{In fact, $\deg (u^*; \gamma)=d_k$ for every closed simple curve $\gamma$  around $a_k$ and lying near $a_k$.}  
\beq \label{necessary}
\sum_{k=1}^n d_k=\chi(S).
\eeq
Moreover, the vorticity $\omega(u^*)$ detects the singular points $\{a_k\}_{k=1}^n$ of $u^*$:  
\beq
\omega(u^*) =  
2\pi \sum_{k=1}^n d_k \delta_{a_k}  \qquad \textrm{ in }  S,
\label{ustar2}
\eeq
where $\delta_{a_k}$ is the Dirac measure (as a $2$-form) at $a_k$. 
The expansion of the minimal intrinsic energy $E^{in}_\eps$ at the second order is given by
$$E^{in}_\eps(u_\eps)=n\pi \log \frac 1 \eps+\lim_{r\to 0} \bigg( \int_{S\setminus \cup_{k=1}^n B_r(a_k)}\frac 12 |Du^*|_g^2\, \vol_g + n\pi \log r\bigg) +n\iota_F+o(1), \textrm{ as } \, \eps\to 0,$$
where $\iota_F>0$ is a constant depending only on the potential $F$ and $B_r(a_k)$ is the geodesic ball centered at $a_k$ of radius $r$, see again Theorem \ref{thm:min_in}. The second term in the above right-hand side (RHS) is called the {\it renormalized energy} between the vortices $a_1, \dots, a_n$ and governs the optimal location of these singular points; in the Euclidean case, this notion was introduced by Bethuel-Brezis-H\'elein in their seminal book \cite{BBH}. In particular, if $S$ is the unit sphere in $\R^3$ endowed with the standard metric $g$, then $n=2$ and $a_1$ and $a_2$ are two diametrically opposed points on $S$. 
Our results will give an explicit description of this renormalized energy, together with its
counterparts for the extrinsic Problems 2 and 3, see Section \ref{sec:renorm}.\\

\bigskip

\section{Main results}

\subsection{Canonical harmonic vector fields of unit length}
\label{sec:canon}

Let $(S,g)$ be  a closed oriented $2$-dimensional Riemannian manifold of genus $\mathfrak g$ (not necessarily embedded in $\R^3$).
We will say that a canonical harmonic vector field of unit length having distinct singular points $a_1,\ldots, a_n \in S$  of index
$d_1,\ldots, d_n\in \Z$ for some $n\geq 1$, is a vector field 
$u^* \in \calX^{1,1}(S)$ such that $|u^*|_g=1$ in $S$, \eqref{ustar2} holds, i.e., $$dj(u^*)=-\kappa \vol_g+2\pi \sum_{k=1}^n d_k \delta_{a_k}$$  and
\beq
d^* j(u^*) =0  \qquad \textrm{ in }  S.
\label{ustar1}
\eeq
Here, $d^*$ is the adjoint of the exterior derivative $d$, i.e., $d^* j(u^*)$ is the unique $0$-form on $S$ such that
$$ \int_S \big(d^* j(u^*), \zeta)_g \vol_g=\int_S \big(j(u^*), d\zeta)_g \vol_g \quad  \textrm{for every smooth $0$-form $\zeta$},$$
where $(\cdot, \cdot)_g$ is the inner product associated to $k$-forms, $k=0,1,2$. 
If $u^*$ satisfies \eqref{ustar2}, then \eqref{omega.def} combined with  
Gauss-Bonnet theorem
imply that necessarily \eqref{necessary} holds.

We will see that condition \eqref{necessary} is also sufficient. 
Indeed, if \eqref{necessary} holds, we will construct solutions of \eqref{ustar2} and \eqref{ustar1},
as follows: if $\psi=\psi(a;d)$ is the unique $2$-form on $S$ 
solving
\begin{equation}
-\Delta \psi = -\kappa\, \vol_g + 2\pi \sum_{k=1}^n d_k\delta_{a_k} \qquad \textrm{ in }  S, \qquad\qquad\int_S \psi = 0,
\label{psi.def}\end{equation}
with the sign convention that $-\Delta = d d^* + d^* d$, then the idea is to find $u^*$ such that $j(u^*)-d^*\psi$ 
belongs to the space of harmonic $1$-forms, i.e., 
\beq
\label{harmon}
Harm^1(S)
=\{ \mbox{integrable $1$-forms $\eta$ on $S$} \ : \ d\eta = d^*\eta = 0\mbox{ as distributions}\}.\eeq
The dimension of the space $Harm^1(S)$ is twice the genus (i.e., $2\mathfrak g$) of $(S, g)$ and {\bf we fix an orthonormal basis $\eta_1,\ldots, \eta_{2\mathfrak g}$ of $Harm^1(S)$} such that 
$$
\int_S (\eta_k , \eta_l)_g \ \vol_g= \delta_{kl} \quad \textrm{ for $k,l=1,\ldots, 2\mathfrak g$}.$$
Therefore, our ansatz for $j(u^*)$ may be written 
\begin{equation}
j(u^*) =   d^*\psi  + \sum_{k=1}^{2\mathfrak g} \Phi_k \eta_k \qquad \textrm{ in }  S
\label{form_jstar}\end{equation}
for some constant vector $\Phi = (\Phi_1,\ldots, \Phi_{2\mathfrak g}) \in \R^{2\mathfrak g}$.
We call these constants {\it flux integrals} as they can be recovered by 
\[
\Phi_k = \int_S (j(u^*), \eta_k)_g \vol_g, \qquad\textrm{ for } k=1,\ldots, 2\mathfrak g.
\]
These flux integrals play an essential role in our analysis. 
They depend nontrivially  on $(a,d)$; this phenomenon is new, 
as far as we know, in the study of  of Ginzburg-Landau models, see Section \ref{sec:challen} for more details.
Note that \eqref{form_jstar} combined with \eqref{psi.def} automatically
yield \eqref{ustar2} and \eqref{ustar1}.
One important point is to characterize
for which values of  $\Phi$  the RHS of \eqref{form_jstar} arises as $j(u^*)$ for
some vector field $u^*$ of unit length in $S$. For that condition, we need to recall the following theorem of Federer-Fleming \cite{FedFle}: there exist $2\mathfrak g$ simple closed geodesics $\gamma_\ell$ on $S$, $\ell=1,\ldots, 2\mathfrak g$, such that for any closed Lipschitz curve $\gamma$ on $S$, one can find integers $c_1\ldots, c_{2\mathfrak g}$
such that 
$$
\textrm{$\gamma$ is homologous to $\displaystyle \sum_{\ell=1}^{2\mathfrak g} c_\ell \gamma_\ell$}$$
i.e., there exists an integrable function $f:S\to \Z$ such that
$$\int_{\gamma}\zeta - \sum_{\ell=1}^{2\mathfrak g} c_\ell\int_{\gamma_\ell}\zeta \ = \ \int_S f \, d\zeta \quad \textrm{ for all
smooth $1$-forms $\zeta$}$$ (see more details in Section \ref{S:homology}).  
{\bf We fix a choice of such geodesic curves $\{\gamma_\ell\}_{\ell=1}^{2\mathfrak g}$}. With these chosen  geodesics $\{\gamma_\ell\}_{\ell=1}^{2\mathfrak g}$ and the harmonic $1$-forms $\{\eta_k\}_{k=1}^{2\mathfrak g}$,
we denote by
\begin{equation}
\alpha_{\ell k} := \int_{\gamma_\ell}\eta_k, \quad k, \ell=1,\ldots, 2\mathfrak g.
\label{akl.def}\end{equation}
The matrix $\alpha=(\alpha_{\ell k})_{1\leq k, \ell \leq 2\mathfrak g}$ is 
invertible\footnote{In fact, by changing the choice of  geodesics and the basis in $Harm^1(S)$, the matrix $\alpha$ is multiplied by an invertible matrix (similar to the standard change of coordinates in vector spaces) due to the above definition of homologous curves where $\int_{\gamma}\eta = \sum_{\ell=1}^{2\mathfrak g} c_\ell\int_{\gamma_\ell}\eta$ for every harmonic $1$-form $\eta$, see also Lemma~\ref{L.FedFlem}.}
(see Lemma \ref{L.FedFlem}).

\begin{theorem}\label{P1}
Let $n\geq 1$ and $d = (d_1,\ldots, d_n)\in \Z^n$ satisfy \eqref{necessary}. 
Then for every $a = (a_1,\ldots, a_n)\in S^n$, there exists
\[
\zeta_\ell = \zeta_\ell(a;d) \in  \R/2\pi\Z, \qquad \ell=1,\ldots, 2\mathfrak g
\]
such that if a vector field $u^*\in \calX^{1,1}(S)$ of unit length solves \eqref{ustar2} and \eqref{ustar1},
then 
$j(u^*)$ has the form 
\eqref{form_jstar}  for constants $\Phi_1,\ldots, \Phi_{2\mathfrak g}$ such that
\begin{equation}
\sum_{k=1}^{2\mathfrak g} \alpha_{\ell k} \Phi_k +\zeta_\ell(a;d) \in 2\pi \Z, \qquad \ell=1,\ldots, 2\mathfrak g, 
\label{lattice}\end{equation}
where $(\alpha_{\ell k})$ are defined in \eqref{akl.def}.
Conversely, given any $\Phi_1,\ldots, \Phi_{2 \mathfrak g}$ satisfying 
\eqref{lattice}, there exists a vector field $u^*\in \calX^{1,1}(S)$ of unit length solving \eqref{ustar2} and \eqref{ustar1} and
such that $j(u^*)$ satisfies \eqref{form_jstar}. In addition, the following hold:
\begin{itemize}
\item[1)]  $\zeta_\ell(\cdot ;  d)$ depends continuously on $a\in S^n$ for every $\ell=1,\ldots, 2\mathfrak g$.
More generally, if
\beq \label{mul.to.mu}
\mu_t := 2\pi\sum_{l=1}^{n_t} d_{l,t} \delta_{a_{l,t}}
 \to \mu_0 :=  2\pi\sum_{l=1}^{n_0} d_{l,0} \delta_{a_{l,0}}
\qquad\mbox{in $W^{-1,1}$} \qquad\textrm{ as } t\downarrow 0,
\eeq
$\{d_{l,t}\}_l$ are integers with \eqref{necessary} and $\sum_{l=1}^{n_t}|d_{l,t}|$ is uniformly bounded in $t$, 
 then $\zeta_\ell(a_t; d_t)\rightarrow \zeta_\ell(a_0; d_0)$ as $t\downarrow 0$.
(See Section \ref{sec:soboesp} for the definition of $W^{-1,1}$.)

\item[2)] any $u^*$ solving \eqref{ustar2} and \eqref{ustar1} belongs to $\calX^{1,p}(S)$ for all $1\leq p<2$,
and is smooth away from $\{a_k\}_{k=1}^n$.

\item[3)] If $u^*, \tilde u^*$ both satisfy \eqref{form_jstar} for the same $(a;d)$ and the same $\{\Phi_k\}_{k=1}^{2\mathfrak g}$, then $\tilde u^* = e^{i\beta}u^*$ for some
$\beta\in \R$ where $e^{i\beta}=\cos \beta+i\sin \beta$ for the isometry $i$ defined in \eqref{isom}.
\end{itemize}
\end{theorem}

\medskip

\begin{remark}
Throughout this paper, objects that we  write as functions of $(a;d)$,
such as $\psi(a;d), \zeta_\ell(a;d)$, and so on,
in fact depend only on
the measure $2\pi\sum_{l=1}^{n} d_{l} \delta_{a_l}$. 
As a result, one can always do the reduction of a set $(a;d)$ of points $a = (a_1,\ldots, a_n)\in S^n$ (not necessarily distinct) and integers $d = (d_1,\ldots, d_n)$ (that can be zero) satisfying \eqref{necessary} to a set $(\tilde a; \tilde d)$ where the points $a_k$ are distinct and $d_k\neq 0$; indeed, one can just put together all the identical $a_k$, sum their degrees $d_k$, relabel them and then cancel the $a_k$ with zero degree $d_k$ (of course, \eqref{necessary} is conserved). This is why we can always assume that the
points $(a_k)$ are distinct and that every $d_k$ is nonzero.
\end{remark}

The constants $\{\zeta_\ell(a;d)\}_{\ell=1}^{2\mathfrak g}$ are determined as follows. For every $\ell=1,\ldots, 2\mathfrak g$, we let 
$\lambda_\ell$ be some smooth simple closed
curve such that $\lambda_\ell$ is homologous to $\gamma_\ell$ (the geodesics
fixed in \eqref{akl.def}) and $\{a_k\}_{k=1}^n$ is disjoint from
$\lambda_\ell$; for example, $\lambda_\ell$ is either $\gamma_\ell$
or, if $\gamma_\ell$ intersects some $a_k$, a small perturbation thereof.
We now define  $\zeta_\ell(a;d)$ to be the element of  $\R/2\pi\Z$  such that 
\begin{equation}\label{zetak.def}
\zeta_\ell(a;d):=\int_{\lambda_\ell} (d^*\psi + A)  \,   \mod 2\pi, \quad \ell=1,\ldots, 2\mathfrak g,
\end{equation}
where $\psi = \psi(a;d)$ is the $2$-form given by \eqref{psi.def} and $A$ is the connection $1$-form associated to any moving frame defined in
a neighborhood of $\lambda_\ell$ (see Section \ref{subsec:connect}). The proof of
Theorem \ref{P1} will show that $\zeta_\ell(a;d)$ is well-defined as an element of $\R/2\pi \Z$.
In general, $\zeta_\ell(a;d)\ne 0\mod 2\pi$ for $\ell = 1,\ldots, 2\mathfrak g$ as
we will see in Example \ref{exam} in which it can be explicitly computed. \\

\nd {\bf The lattice $\calL(a;d)$}. Due to Theorem \ref{P1}, we introduce the following set corresponding to $n$ distinct points $a = (a_1,\ldots, a_n)\in S^n$ and nonzero integers $d = (d_1,\ldots, d_n)\in \Z^n$ satisfying \eqref{necessary}:  
\[
\mathcal L(a;d) := \{ \Phi = (\Phi_1,\ldots, \Phi_{2\mathfrak g})\in \R^{2\mathfrak g}\, : \, \sum_{k=1}^{2\mathfrak g} \alpha_{\ell k}\Phi_k +\zeta_\ell(a;d)\in 2\pi \Z, \, \ell=1, \dots, 2\mathfrak g \}.
\]
It is a lattice (up to a translation).  
Indeed, if $\alpha=(\alpha_{\ell k})_{1\leq \ell, k\leq 2\mathfrak g}$ is the matrix defined in \eqref{akl.def} with the inverse $\alpha^{-1}$, then
\beq
\label{setL}
\Phi\in \calL(a;d) \Longleftrightarrow \Phi \in 2\pi \alpha^{-1}\Z^{2\mathfrak g}-\alpha^{-1}\zeta,
\eeq
i.e., the lattice is determined by the columns of the matrix $\alpha^{-1}$ and it is shifted by the vector $\alpha^{-1}\zeta$ with $\zeta(a;d)=(\zeta_1, \dots, \zeta_{2\mathfrak g})$ defined by 
\eqref{zetak.def}. Due to the relation on $\Phi$, the above discussed change of geodesics $\{\gamma_k\}$ and basis of harmonics $\{\eta_k\}$ would be equivalent to a change of coordinates in the lattice $\calL(a;d)$.

The continuity of $\zeta$ stated at Theorem \ref{P1} point 1) can be quantified as follows:

\begin{lemma} 
\label{lem:latti} For every $K\in \Z_+$, there exists $C_K>0$ such that for every two measures $\mu=2\pi \sum_{k=1}^{n} d_{k} \delta_{a_{k}}$ and $\tilde \mu=2\pi \sum_{k=1}^{\tilde n} 
\tilde d_{k} \delta_{\tilde a_{k}}$ with the distinct points $a=(a_{k})_{k=1}^{n}$, $\tilde a=(\tilde a_{k})_{k=1}^{\tilde n} \subset S$ and the nonzero integers $d=\{d_{k}\}_{k=1}^{n}$ and $\tilde d=\{\tilde d_{k}\}_{k=1}^{\tilde n}$ satisfying \eqref{necessary} and $\sum_{k=1}^{n} |d_{k}|$, $\sum_{k=1}^{\tilde n} |\tilde d_{k}|\leq K$, then
\beq
\label{dist_lat}
\dist_{\R^{2\mathfrak g}} \big(\calL(a; d),\calL(\tilde a; \tilde d)\big)\leq C_K \|\mu-\tilde \mu\|_{W^{-1,1}(S)}.\eeq
\end{lemma} 

Here $\dist_{\R^{2\mathfrak g}}(\calL, \tilde \calL) = \inf_{\Phi\in \calL, \tilde \Phi\in \tilde \calL}|\Phi - \tilde \Phi|$, which  coincides with the Hausdorff distance, since  $\calL$ and $\tilde \calL$ are both translations of a fixed lattice $2\pi\alpha^{-1}\Z^{2\mathfrak g}$.

\bigskip

\subsection{Renormalized energy}
\label{sec:renorm}
$\quad$\\

\nd {\bf The intrinsic Dirichlet energy}. Let $(S,g)$ be  a closed oriented $2$-dimensional Riemannian manifold of genus $\mathfrak g$ (not necessarily embedded in $\R^3$).
For any $n\geq 1$, we consider $n$ {\bf distinct} points $a = (a_1,\ldots, a_n)\in S^n$. Let $d = (d_1,\ldots, d_n)\in \Z^n$ satisfying \eqref{necessary},  
$\{\zeta_\ell(a;d)\}_{\ell=1}^{2\mathfrak g}$ be given in Theorem \ref{P1} and $\Phi\in \R^{2\mathfrak g}$ be a constant vector inside the lattice $\calL(a;d)$ defined in \eqref{setL}.
We define the {\it renormalized energy} between the vortices $a$ of indices $d$ by
\beq
\label{defi_W}
W(a,d,\Phi):=\lim_{r\to 0} \bigg(\int_{S\setminus \cup_{k=1}^n B_r(a_k)}\frac12 |Du^*|^2_g\, \vol_g+\pi \log r \sum_{k=1}^n d_k^2\bigg),
\eeq
where $u^*=u^*(a,d,\Phi)$ is the unique (up to a multiplicative complex number) canonical harmonic vector field given in Theorem \ref{P1} and $B_r(a_k)$ is the geodesic ball centered at $a_k$ of radius $r$. (Our arguments will show that the above limit indeed exists, see \eqref{W.def}). As in the Euclidean case (see  the pioneering work of
Bethuel, Brezis and H\'elein \cite{BBH}), we can compute the renormalized energy by using the Green's function. For that, let
$G(x,y)$ be the unique function on $S\times S$ such that
$$
-\Delta_x ( G(\cdot, y)  \, \vol_g) = \delta_y \,  - \frac {\vol_g} {\mbox{Vol}_g(S)}\, \, \textrm{ distributionally in } S, \, \, \,  \int_S G(x,y) \vol_g(x) = 0,  \, \,  \forall y\in S,
$$
with $\mbox{Vol}_g(S):= \int_S \vol_g$. 
Then $G$ may be represented in the form 
(see  Chapter 4.2\footnote{More precisely, according to \cite{Aubin}, page 109, eqn (17),
one may define $G_0$ as above such that $H := G - G_0$ 
can be represented in the form
\[
H(x,y) =  \int_S \Delta_z G_0(x,z) G_0(z,y) \vol(z)  +\mbox{ smoother terms},
\]
(where here $\Delta_z$ denotes the pointwise Laplacian rather than the distributional Laplacian)
and in addition $\|\Delta_z G_0  \|_{L^\infty(S)}\le C$.})
 \cite{Aubin}$$
G(x,y) = G_0(x,y) + H(x,y),\qquad\mbox{ with }H\in C^1(S\times S),
$$
where $G_0$ is smooth away from the diagonal, 
 with
\begin{align*}
G_0(x,y) &= -\frac 1{2\pi}\log(\dist_S(x,y))\\ 
&\quad  \mbox{ if the geodesic distance } \dist_S(x,y) <\frac 12 (\mbox{injectivity radius of $S$}).
\end{align*}
The $2$-form $\psi=\psi(a;d)$ defined in \eqref{psi.def} can be written as: 
\beq
\label{newpsi}
\psi= 2\pi \sum_{k=1}^n d_k G(\cdot, a_k) \vol_g \ +  \psi_0  \vol_g \qquad\mbox{ in } S,
\eeq
where $\psi_0\in C^\infty(S)$ has zero average on $S$ and solves
\begin{equation}
-\Delta \psi_0  = -\kappa + \bar \kappa , \qquad
\mbox{ for }
\bar \kappa = \frac 1{\mbox{Vol}(S)} \int_S \kappa \vol_g  = \frac {2\pi \chi(S)}{\mbox{Vol}(S)} .
\label{psi0.def}\end{equation}
In other words, the $2$-form $x\mapsto \psi(x) + d_k \log \dist_S(x, a_k) \vol_g$ is $C^1$ in a neighborhood of $a_k$ for every $1\leq k\leq n$.
We have the following expression of the renormalized energy:

\medskip

\begin{propo}
Given $n\geq 1$ distinct points $a_1,\ldots, a_n \in S$, integers $d_1,\ldots, d_n$ with \eqref{necessary}
and $\Phi \in \mathcal L(a;d)$, then
\begin{align}
W(a,d,\Phi) &=  4\pi^2 \sum_{1\leq l< k\leq n} d_l d_k G(a_l,a_k) 
+2\pi \sum_{k=1}^n  \left[ \pi d_k^2 H(a_k,a_k) +d_k \psi_0(a_k) \right]  
\nonumber \\
&\quad \quad 
+ \frac 12|\Phi|^2 + \int_S \frac{|d\psi_0|_g^2}2 \vol_g \ ,
\label{W.formula}
\end{align}
where $\psi_0$ is defined in \eqref{psi0.def}. 
\label{prop.W}\end{propo}

\medskip

In the case of the unit sphere $S$ in $\R^3$ endowed with the standard metric (in particular, $\psi_0$ vanishes in $S$), 
if $n=2$ and $d_1=d_2=1$, then the second term in the RHS of \eqref{W.formula} is independent of $a_k$ (as $x\mapsto H(x,x)$ is constant, 
see \cite{steiner}); moreover, $\Phi=0$ and so, minimizing $W$ is equivalent by minimizing the Green's function $G(a_1, a_2)$ over the set of 
pairs $(a_1, a_2)$ in $S\times S$, namely, the minimizing pairs are diametrically opposed. 

More generally, if $S=\mathbb{S}^2$ is endowed with a non-standard metric $g$, then Steiner \cite{steiner} proves that $x\mapsto H(x,x)+\frac1{2\pi}\psi_0(x)$ is constant.\footnote{The function $x\mapsto H(x,x)$ is called the Robin's mass on $\mathbb{S}^2$, see e.g. \cite{steiner}. } Therefore, an optimal pair $(a_1, a_2)$ of vortices of degree $d_1=d_2=1$ minimizes the following energy
$$(a_1, a_2)\in S\times S\mapsto 4\pi G(a_1,a_2)+\psi_0(a_1)+\psi_0(a_2).$$
In general this is a complicated expression, but it should be possible to find minima in
special cases. For example, if $S$ is an ellipsoid, then we expect the vortices 
$a_1$ and $a_2$ will be placed at the two poles of the largest diameter as they have maximal Gauss curvature (the maximum principle suggests that this will minimize $\psi_0$), and they maximize the distance $\dist_S(a_1, a_2)$ (so minimize $G(a_1,a_2)$).
\\

\nd {\bf The extrinsic Dirichlet energy}. In the case of an embedded surface $S\subset \R^3$, when dealing with the extrinsic Dirichlet energy in Problems 2 and 3, a second interaction energy between vortices $a$ of degree $d$ is important next to $W(a,d,\Phi)$.  For that, we denote by $\calS:TS\to TS$ the shape operator on $S$, that is,
\beq
\label{shape}
\calS(v)=-\dbar_v N, \quad \textrm{for every } v\in TS,
\eeq
where $N$ is the Gauss map on $S$. Let $u^*=u^*(a,d,\Phi)$ be the unique (up to a multiplicative complex number) canonical harmonic vector field given in Theorem \ref{P1}. We consider
\beq
\label{renorm_ext}
\tilde{W}(a,d,\Phi)=\min_{\Theta:S\to \R} \frac12\int_{S} |d\Theta|^2_g+\big|\calS(e^{i\Theta}u^*)\big|^2_g\, \vol_g.
\eeq
(Existence of a minimizer is standard, as we discuss in more detail later.)
We will prove in Theorems \ref{ext.gamma} and \ref{mm.gamma} in Sections \ref{sec:ext_gamma} and \ref{sec:11} that the renormalized energy associated to the extrinsic energy $E^{ex}_\e$ (as well as the one associated to the energy $E^{mm}_\e$ in Problem~3) is given by
$${W}(a,d,\Phi)+\tilde{W}(a,d,\Phi).$$

Note that for the unit sphere $S$ in $\R^3$ endowed with the standard metric, the shape operator satisfies $|\calS(u)|_g = 1$  for any $x\in S$ and unit vector $u\in T_xS$, so that $\tilde{W}(a,d,\Phi) = 2\pi$ for all $(a,d, \Phi)$. Therefore, the total renormalized energy $W+\tilde W$ has the same minimizers as $W$.

\subsection{$\Gamma$-convergence}
\label{sec:gamma}

Given the potential $F$ in Section \ref{sec:intro}, we compute the intrinsic energy of the radial profile of a vortex of index $1$ inside a ball of radius $R>0$ with respect to the Euclidean structure on $\R^2$: 
\begin{align}
\label{ier.def}
I^{in}_F(R,\eps) &:= \min \left \{ \int_{B_R(0)} e_\e(v)\, dy \, : \, v:B_R(0)\to \C, v(y)=\frac{y}R \, \textrm{ for } |y|=R \right \}\\
\nonumber \textrm{with } \quad e_\e(v)&:=\frac12|\nabla v|^2 + \frac{1}{4\eps^2}F(|v|^2).
\end{align}
The above minimum is indeed achieved\footnote{In fact, the minimizer is unique and symmetric \cite{Pac-Riv,Miro}. For other uniqueness results, see \cite{INSZ1, INSZ2}. } and $I^{in}_F(R,\eps) = I^{in}_F(\lambda R, \lambda \eps) = I^{in}_F(1,\frac{\eps}{R})=:I^{in}_F(\frac{\eps}{R})$ for every
$\lambda>0$, and the following limit exists (see \cite[Lemma III.1]{BBH}): 
\beq
\label{gammaF.def}
\iota_F:=\lim_{t\downarrow 0}( I^{in}_F(t)  + \pi \log t).
\eeq
The extrinsic energy of the radial profile of a vortex of index $1$ in Problem 2 will also correspond to the one above. However, for Problem 3, due to the constraint of unit-length on the magnetization $M$, the following expression comes out: 
\begin{align}
\label{uuu}
I^{mm}_F(R,\eps) &:= \min \left \{ \int_{B_R(0)} \tilde{e}_\e(v)\, dy \, : \, v:B_R(0)\to \SSS^2, v(y)=\frac1R(y,0) \, \textrm{ for } |y|=R \right \}\\
\nonumber
 \textrm{with } \quad \tilde{e}_\e(v)&:=\frac12|\nabla v|^2 +  \frac{1}{4\eps^2} F(1-v_3^2) \quad \textrm{where } v=(v_1, v_2, v_3).
\end{align}
Again, the above minimum is indeed achieved for every fixed $R, \e>0$ and writing $I^{mm}_F(R,\eps)=:I^{mm}_F(\frac{\eps}{R})$, we obtain the following quantity  (see \eqref{aici})
\beq
\label{iota_til}
\tilde \iota_F:=\lim_{t\downarrow 0}( I^{mm}_F(t)  + \pi \log t).
\eeq

We state our main result for Problem 1 in a closed oriented $2$-dimensional Riemannian manifold of genus $\mathfrak g$ : \\

\begin{theorem}\label{intrinsic.gammalim}
The following $\Gamma$-convergence result holds.
\begin{itemize}
\item[1)] (Compactness) Let $(u_\eps)_{\eps\downarrow 0}$ be a family of vector fields in $\calX^{1,2}(S)$ satisfying
$
E^{in}_\eps(u_\eps)  \le T \pi |\log \eps| + C
$ 
for some integer
$ T\ge  0$
and a constant $C>0$. 
We denote by
$$\Phi(u_\eps) := \left(\int_S  (j(u_\eps) ,  \eta_1)_g \vol_g ,\ldots, \int_S  (j(u_\eps), \eta_{2\mathfrak g})_g \vol_g\right)\in \R^{2\mathfrak g},$$
where $\{\eta_k\}_{k=1}^{2\mathfrak g}$ are fixed in \eqref{form_jstar}. Then
there exists a sequence $\eps \downarrow 0$ such that 
\begin{equation}
\omega(u_\eps)  \longrightarrow 2\pi \sum_{k=1}^{n} d_k\delta_{a_k} \quad \textrm{in } \, W^{-1,1}, \, \,  \textrm{ as } \, \eps \to 0,
\label{convergence2}\end{equation}
where $\{a_k\}_{k=1}^n$ are distinct points in $S$ and $\{d_k\}_{k=1}^n$ are nonzero integers satisfying \eqref{necessary} and $\sum_{k=1}^n |d_k|\leq T$. Moreover, if $\sum_{k=1}^n |d_k|= T$, then
$n = T$ and 
$|d_k|=1$ for every $k=1, \dots, n$; in this case, for a further subsequence, there exists $\Phi\in \calL(a;d)$ such that $\Phi(u_\eps)\rightarrow \Phi$.\\

\item[2)] ($\Gamma$-liminf inequality) Assume that the vector fields $u_\eps\in \calX^{1,2}(S)$ satisfy \eqref{convergence2} for $n$ distinct points $\{a_k\}_{k=1}^n\in S^n$ and $|d_k|=1$, $k=1, \dots n$ that satisfy \eqref{necessary} and $\Phi\in \calL(a;d)$. Then
$$\liminf_{\eps\to 0}
\left[
E^{in}_\eps(u_\eps)- n \pi |\log \eps|
\right]  \ \ge \  
W(a,d,\Phi)+n \iota_F.$$

\item[3)] ($\Gamma$-limsup inequality) For every $n$ distinct points $a_1,\ldots, a_n\in S$ and $d_1,\ldots , d_n\in \{\pm 1\}$ satisfying
\eqref{necessary} and every $\Phi\in \calL(a;d)$
there exists a sequence of  
vector fields $u_\eps$ on $S$ such that $|u_\eps|_g\leq 1$ in $S$, \eqref{convergence2} holds and
$$
E^{in}_\eps(u_\eps)- n \pi |\log \eps| \longrightarrow  W(a,d, \Phi)+n \iota_F \quad \textrm{ as } \, \eps \to 0.
$$
\end{itemize}
\end{theorem}

In fact, in the case $|d_k|=1$, we will prove a sharper lower bound than the one stated  in point (2) above, see Proposition 
 \ref{intrinsic.gammalim2} below. In the general case of arbitrary degrees $d_k\in \Z\setminus \{0\}$ satisfying \eqref{necessary}, we only prove a lower bound at the first order, implicit in the fact that $\sum_{k=1}^n|d_k| \le T$; see also Corollary \ref{Cor.cgc1}.

If $T=0$, the theorem implies that $n=0$. In this case, then,
there are no limiting vortices,  so necessarily $\mathfrak g=1$ (i.e., $S$ is diffeomorphic to the $1$-torus). Also, $\calL(a,d)$ is a fixed lattice $\calL$. See also Remark \ref{rem:min_in} point 2) below. By \eqref{W.formula}, the renormalized energy in this case is exactly 
$\frac 12 |\Phi|^2+ \frac 12 \int_S |d\psi_0|^2 \, \vol_g$. It is not clear whether $\Phi = 0$ belongs to the lattice $\calL$ if the torus is
not flat.

The situation in points 2) and 3) above (i.e., all vortices have degree $\pm 1$) is typical when the vector fields $u_\eps$ are minimizers of $E^{in}_\eps$ (or energetically close to minimizing configurations). For more details, see Theorem \ref{thm:min_in}.

For Problem 2 where the surface $S$ is isometrically embedded in $\R^3$,
 one has the similar result by replacing the interaction energy between vortices with:
$$W(a,d, \Phi)+\tilde W(a,d, \Phi)$$
see Theorem \ref{ext.gamma}.
While for Problem 3, the difference with respect to the result of Problem~2 consists in replacing $\iota_F$ by $\tilde \iota_F$ (see Theorem \ref{mm.gamma});
so, up to this constant, there is no change of the vortex location when minimizing the interaction energy in Problem 3 w.r.t. Problem 2.

This theorem is the generalization of the $\Gamma$-convergence result for $E^{in}_\eps$ in the Euclidean case (see {\cite{CoJe,JeSo,SandSerfbook,AliPon}) and it is based on topological methods for energy concentration (vortex ball construction, vorticity estimates etc.) as introduced in \cite{Je,Sa}. A part of our results were announced in \cite{IgnJer}.

\bigskip

\nd {\it Outline of the article}.  In Section \ref{sec:motiv}, we give a motivation for our models coming from micromagnetics and geometry, while in Section \ref{sec:challen}, we present some challenges and novelties of our results with respect to other Ginzburg-Landau type models. Before giving the proofs of our results, we present in Section \ref{S:nb} some notation and background on differential forms, Sobolev spaces on manifolds and some useful computations involving the current. In Section \ref{S:chm}, we prove the characterization of canonical harmonic vector fields in Theorem~\ref{P1} as well as the stability estimate for the lattice $\calL(a;d)$ in Lemma \ref{lem:latti}; we also give Example~\ref{exam} for the non-triviality of the lattice $\calL(a;d)$ in the case of the flat torus
$\R^2/\Z^2$. In Section \ref{Sec:RE} we prove the formula of the renormalized energy in Proposition \ref{prop.W}. In Section \ref{sec:Comp}, we prove the compactness result for the vorticity measure in Theorem \ref{intrinsic.gammalim} point 1); as a consequence, we deduce the $\Gamma$-limit at the first order of the intrinsic energy $E^{in}_\eps$. The lower / upper bound in Theorem \ref{intrinsic.gammalim} are proved in Section \ref{sec:intrin}; in particular, we show an improved lower bound of the intrinsic energy $E^{in}_\eps$ in Proposition \ref{intrinsic.gammalim2}. In Sections~\ref{sec:ext_gamma} and \ref{sec:11}, we prove the $\Gamma$-convergence result at the second order for the extrinsic energy $E^{ex}_\eps$ and micromagnetic energy $E^{mm}_\eps$ (see Theorems \ref{ext.gamma} and \ref{mm.gamma}). Finally, in Section \ref{sec:min}, we characterize the asymptotic behavior of minimizers of our three energy functionals. In Appendix, we give the so-called ``ball construction" adapted to a surface $S$ which is a key tool in proving the lower bound of our functionals.  

\bigskip

\section{Motivation}
\label{sec:motiv}

\subsection{Micromagnetics} \label{sec:micromag}
One of the motivation of our study comes from micromagnetics. Micromagnetics is a variational principle
describing the behavior of small ferromagnetic bodies considered here of cylindrical shape
$\Omega=\Omega'\times (0,t)$ where $\Omega'$ is
the cross section of the sample of diameter $\ell$ and $t$ is the
thickness of the cylinder (see Figure \ref{cyli}).
\begin{figure}[htbp]
\center
\includegraphics[scale=0.4,
width=0.4\textwidth]{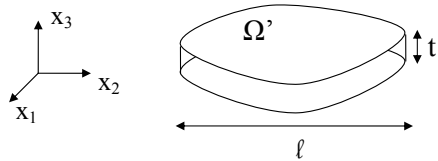} \caption{A ferromagnetic
sample.} \label{cyli}
\end{figure}
A ferromagnetic material is described by a $\SSS^2$-valued map
$$m:\Omega\to \SSS^2,$$ called magnetization, corresponding to the stable states of the energy functional (written here in the absence
of anisotropy and external magnetic field): 
\beq \label{en_global} E^{3D}(m)=\eta^2 \int_{\Omega}|\nabla
m|^2\, dx +\int_{\R^3}|\nabla
U|^2\, dx.\eeq

The first term, called exchange energy, penalizes the variations of $m$ according to the material constant $\eta>0$ (the exchange length) that is of the order
of nanometers. The second term of $E^{3D}$ is the stray field energy that favors flux closure; more precisely, the stray field potential
$U:\R^3\to \R$ is determined by the static Maxwell equation
\begin{align}
\label{stray}\Delta U&=\nabla\cdot \big(m {\bf 1}_\Omega\big)\quad \textrm{in} \quad \RR^3,\\
\nonumber \textrm{i.e.,}\quad \int_{\RR^3} \nabla U\cdot \nabla
\zeta \, dx&=\int_{\Omega} m\cdot \nabla \zeta\, dx,\quad \forall
\zeta\in C^\infty_c(\R^3).
\end{align}
In other words, the stray field $\nabla U$ is the Helmholtz projection of $m{\bf 1}_{\Omega}$ onto the $L^2$-gradient fields and
$$\int_{\R^3}|\nabla
U|^2\, dx=\|\nabla\cdot \big(m{\bf 1}_{\Omega}\big)\|^2_{\dot{W}^{-1,2}(\R^3)}.$$

\bigskip

\nd {\bf Thin film regime of very small ferromagnets}. Assume the following asymptotic regime\footnote{A thin film regime is characterized by a small aspect ratio $h$; the ferromagnetic samples considered here are very small because $\ell$ has the order of nanometers as $\eta$.}:
$$h:=\frac t \ell\to 0 \quad \textrm{and} \quad \eps:=\frac \eta \ell={\rm constant}$$
for some fixed parameter $\eps>0$. Set $x=(x', x_3)$, $x'=(x_1, x_2)$ where $\,' \,$  stands only in this section for an in-plane quantity. In order to study the asymptotic behavior as $h\to 0$, we rescale the variables:
${y}':={x'}/{\ell}$ (so, $\omega':=
{\Omega'}/{\ell}$ is of diameter $1$), ${y}_3:={x_3}/{t}$, $m_h(y):=m(x)$ and 
$$E_h(m_h):=\frac{1}{\eta^2t} E^{3D}(m), \quad m_h:\omega=\omega'\times (0,1)\to \SSS^2,$$
where the diameter of $\omega'$ equals $1$.
In this context, Gioia-James \cite{GioJam} proved the following $\Gamma$-convergence result in strong $L^2$-topology: 
$$E_h\stackrel{\Gamma}{\rightarrow} E_0$$
where the $\Gamma$-limit functional $E_0$ is given by
\begin{align*}
E_0(M)&=\int_{\omega} \bigg\{|\nabla M|^2+\frac 1{\eps^2} M_{3}^2 \bigg\}\, dy=\int_{\omega'} \bigg\{|\nabla' M|^2+\frac 1{\eps^2} M_{3}^2 \bigg\}\, dy'
\end{align*}
for a limit magnetization $M=(M', M_3):\omega\to \SSS^2$ that is invariant in $y_3$-direction, i.e., $\partial_{y_3}M=0$ in 
$\omega$, $\nabla'=(\partial_1, \partial_2)$, so that one can write $$M=M(y')\in W^{1,2}(\omega', \SSS^2), \quad y'=(y_1, y_2)\in \omega'.$$
The hint is the following: since the exchange energy term in $E_h(m_h)$ of $m_h$ is given by
$$m_h\mapsto \int_{\omega} \bigg(|\nabla' m_h|^2+\frac1{h^2}|\partial_{y_3} m_h|^2\bigg)\, dy,$$
it is clear that configurations $m_h$ of uniformly bounded energy (i.e., $E_h(m_h)\leq C$) tend to converge strongly in $L^2$ to
a limit $M$ depending only on $y'$-variables. The more delicate issue consists in understanding the scaling of the stray field energy term.
For that, we assume for simplicity that $m_h$ is invariant in $y_3$-direction (i.e., $m(x')=m_h(x'/\ell)$ for $x'\in\Omega'$). Then the Maxwell equation
\eqref{stray} turns into: $$ 
\Delta
U=\nabla'\cdot m' \, \, {\Hdf^3\llcorner \Omega}+m \cdot \nu \, {\Hdf^2\llcorner \partial
\Omega}\quad \textrm{in}\quad \RR^3,$$ where $\nu$ is the unit
outer normal vector on $\partial \Omega$ and $\Hdf^k$ is the Hausdorff measure of dimension $k$. This equation
is a transmission problem that can be solved
explicitly using the Fourier transform $\mathcal{F}(\cdot)$ in the
in-plane variables $x'$ and the computation yields (see e.g. \cite{Ignat}):
\begin{align*}
\int_{\RR^3}|\nabla U|^2\, dx=t\int_{\RR^2}
\tilde{f}(\frac{t}{2}|\xi'|)\big|{\mathcal{F}}{(m_3 {\bf 1}_{\Omega'})} \big|^2\,d\xi'+t\int_{\RR^2}
f(\frac{t}{2}|\xi'|)\big|\frac{\xi'}{|\xi'|}\cdot{\mathcal{F}}{(m' {\bf 1}_{\Omega'})}
\big|^2\,d\xi',\end{align*} where
$$\tilde{f}(s)=\frac{1-e^{-2s}}{2s}\quad \textrm{ and }\quad f(s)=1-\tilde{f}(s) \textrm{ if } s\geq0.$$
To conclude, one formally approximates $\tilde{f}(s)\approx 1$ and $f(s)\approx s$ if $s=o(1)$ so that\footnote{A different regime is studied in \cite{Ignat_Otto}.}
\begin{align*} \frac1{\eta^2 t}\int_{\RR^3} |\nabla U|^2\, dx\approx \frac1{\eta^2}\int_{\Omega'} m_3^2\, dx'+O(\frac{h}{\eps^2})\approx \frac1{\eps^2}\int_{\omega'} m_{h,3}^2\, dy',\end{align*}  
as $h\to 0$.
\bigskip

\nd {\bf Very small magnetic shells}. The situation of curved ferromagnetic samples was considered by Carbou \cite{Carbou}. The context is the following: let $S\subset \R^3$ be a surface isometrically embedded in $\R^3$ of diameter $\ell=1$ and $N$ be the Gauss map at $S$. A curved magnetic shell is considered occupying the domain
$$\Omega:=\bigg\{ x'+sN(x')\, :\, s\in (0, t), \, \, x'\in S \bigg\}.$$ 
Then Carbou \cite{Carbou} proved the corresponding $\Gamma$-convergence result as in Gioia-James \cite{GioJam} where the $\Gamma$-limit is given by
$$M\in H^1(S; \SSS^2) \mapsto \int_S |\dbar M|^2+\frac{1}{\eps^2}(M\cdot N)^2\, d\Hdf^2$$
where $\dbar M$ is the extrinsic differential of $M$ and $M\cdot N$ is the normal component of $M$ on the surface $S$. In the context of energy $E^{mm}_\e$, denoting $|m|^2=1-(M\cdot N)^2$ we have $F(|m|^2)=1-|m|^2=(M\cdot N)^2\geq (1-|m|)^2$, so \eqref{F.growth} is satisfied for $C=1$.

\subsection{Geometry and topology}
One of the first theorems one encounters in topology states that there does not exist any continuous nonvanishing vector field
on any closed oriented surface $S$ of genus $\mathfrak g \ne 1$. 
A unit vector field on such a surface must
therefore have singularities. If the surface has a Riemannian metric,
one might hope to use the metric structure to seek an energetically optimal unit vector field,
which presumably should have an energetically optimal placement of singularities.
This line of thought leads to the problem of minimizing the covariant Dirichlet 
energy
\begin{equation}
\int_S \frac 12 |Du|_g^2 \, \vol_g
\label{cde}\end{equation}
among all unit vector fields on $S$.
However, it follows from results in \cite{SSV} (an extension to the Sobolev space $W^{1,2}$ of the ``Hairy Ball Theorem", see also related results in \cite{CanSegVen})
that when $\mathfrak g\ne 1$, there does not exist any unit vector field on $S$
of finite energy. It is then reasonable (by analogy with standard
considerations in the analysis of the Ginzburg-Landau functional)
to seek energetically optimal 
vector fields by relaxing the constraint $|u|_g =1$ and replacing it with a
term that penalizes deviations of $u$ from unit length, then considering a 
suitable limit. This leads 
to Problem 1, or to Problem 2 if one is interested in the extrinsic Dirichlet 
energy on an embedded surface. One may thus interpret our results about these
problems as describing an optimal placement of singularities,  as sought above.

In the case of genus $1$, a number of results about minimization
of the extrinsic Dirichlet energy, in the space of unit tangent vector fields, are
proved in \cite{SSV}, motivated by models of liquid crystals \cite{NapVerg1, NapVerg2}.

\bigskip

\section{Challenges}
\label{sec:challen}

One first main result, Theorem \ref{P1}, contains a classification of all
harmonic unit vector fields in $\calX^{1,1}(S)$ with  singularities at
prescribed points. This classification is surprisingly subtle on manifolds
of genus $\ge 1$. Indeed, Theorem \ref{P1} shows that a 
harmonic vector field $u^*$ with singularities of degree $d_k$
at points $a_k\in S$ for $k=1,\ldots, n$ exists if and only if 
 the harmonic
part of the associated current $j(u^*)$ --- that is, the projection
of $j(u^*)$ onto the space of harmonic $1$-forms ---  belongs to a particular
lattice $\calL(a;d)$ in the $2\mathfrak g$-dimensional 
space of harmonic $1$-forms. (The degrees must also satisfy 
the natural topological constraint $\sum_{k=1}^n d_k=\chi(S)$; this
is clear and unsurprising.) 
We show that  $\calL(a;d)$ depends nontrivially on $(a;d)$
in a concrete example, and we believe this to be the case in general.
Although flux quantization constraints
appear in more or less all Ginzburg-Landau models on
non-simply connected domains, the dependence 
(encoded in $\calL(a;d)$)
of the constraints
on the vortex locations and the geometry of $S$ seems to be a new phenomenon.

The lattice $\calL(a;d)$ reappears and gives rise to novel issues 
in the proof of our main results. There we must
control energy coming from the harmonic part of
the current $j(u_\e)$
for a sequence $u_\e$ of vector fields; this requires a detailed understanding
of  the way in which the distribution of vorticity  in (approximately) unit vector fields
imposes vorticity-dependent (approximate) constraints on the harmonic part of the 
associated currents. 

These points do not appear in earlier
work on related problems. 
This includes papers of Orlandi \cite{Giando} and Qing \cite{JieQing}
that describe the asymptotic behaviour of minimizers
of a Ginzburg-Landau energy
for a section of a complex line bundle over a Riemannian manifold. 
This minimization problem involves finding not only an optimal unit-length section $u$  
of the bundle (corresponding 
in our setting to a tangent unit vector field), but also an optimal connection
on the bundle. By contrast, we insist on working with the Levi-Civita connection,
natural in our setting. A consequence of the freedom to choose an optimal connection
is that the vorticity-dependent constraints described by the lattice $\calL(a.d)$
do not arise in \cite{Giando, JieQing}, either in the description of optimal maps or the characterization of energy asymptotics.

A distinct and  important technical issue arises from the need to isolate
the energetic contribution of the vortex cores, reflected in the constants
$\iota_F$ and $\tilde \iota_F$ arising in Theorems \ref{intrinsic.gammalim}, 
\ref{ext.gamma}, and \ref{mm.gamma}. As usual, these terms are captured
by sharp energy estimates carried out near the vortex cores. The new feature
is that, in order to approximate the metric $g$ well by the Euclidean metric -- this is
necessary to correctly resolve $\iota_F$ and $\tilde \iota_F$ -- we must
carry out these estimates on geodesic balls that contain the vortices and {\it whose radii 
vanish as $\e$ tends to $0$.} This requirement forces us to rely on
refined quantitative control of the vorticity throughout our analysis.

Very closely related is the recent work of Canevari and Segatti \cite{CanSeg},
characterizing the asymptotics of a spatially-discretized 
covariant Dirichlet energy \eqref{cde} on a surface, 
in the limit as the discretization scale tends to zero. 
These authors prove results quite parallel to ours, but their
main focus is on the discrete-to-continuum limit, and the
renormalized energy that they find (see \cite{CanSeg}, equations (18), (20))
is described in a way that leaves the its dependence on $(a;d)$ very implicit and
does not resolve the issues appearing in our Theorem \ref{P1} 
and elsewhere in this paper.

\bigskip

\section{Notation and background}\label{S:nb}

Let $(S,g)$ be  a closed oriented $2$-dimensional Riemannian manifold of genus $\mathfrak g$, not necessarily embedded in $\R^3$.
We will write $\chi(S)$ and $\mathfrak g$ to denote the Euler characteristic and the genus of $S$ that are related by 
$\chi(S)=2-2\mathfrak g$. We write $\dd$ to denote the Levi-Civita connection on $(S,g)$. 
We will write $\dist_S(p,q)$ to denote the geodesic distance between $p\in S$ and $q\in S$:
\[
\dist_S(p,q) := \inf \left\{\int_0^1 |\gamma'(s)|_g ds \ : \ \gamma:[0,1]\to S \, \textrm{ Lipschitz, }\, \ \gamma(0)=p,\gamma(1)=q
\right\}.
\]
We will write $B_r(x)$ (or $B(x,r)$) to denote the open geodesic ball 
\[
B_r(p) := \{ q\in S : \dist_S(p,q)<r\}.
\]
and $\bar B_r(x)$ is the closure of this ball. Given points $a_1,\ldots, a_n\in S$ and $\sigma>0$, we also write $a=(a_1, \dots, a_n)$ and 
\[
S_\sigma(a) := S \setminus \cup_{k=1}^n \bar B_\sigma(a_k)
\]
and  
\[
\rho_a := \min_{k\ne \ell}\dist_S(a_k,a_\ell).
\]
We will also write simply $S_\sigma$, when it is clear which points $(a_1,\ldots , a_n)$ we
have in mind. We write ${\bf 1}_{S_\sigma}$ for the characteristic function of $S_\sigma$.

\subsection{Differential forms} If $\eta, \zeta$ are $k$-forms, $k=0,1,2$, we will write $(\eta, \zeta)_g$ to denote the 
inner product induced by the metric $g$, and the length $|\eta|_g := (\eta, \eta)_g^{1/2}$.
We will always fix a 
global volume $2$-form, denoted  $\vol_g$,
associated to the metric for which we define the isometry $$i:TS\to TS$$ by \eqref{isom}. 
The Hodge-star  operator, mapping $k$-forms to $2-k$ forms, is defined
by requiring that
\[
\eta \wedge \star \zeta =  (\eta, \zeta)_g \vol_g \qquad \mbox{ for all $k$-forms }\eta, \zeta.
\]
It is well-known, and straightforward to check, that $\star\star = (-1)^{k(2-k)}$
for a two-dimensional surface $S$. Also, for dimension $2$, we define the adjoint of the exterior derivative $d$ by $d^* :=- \star d \star$ on $S$. Then it follows that
\[
\int_S  (d\eta, \zeta)_g \vol_g = 
\int_S  (\eta,  d^*\zeta)_g \vol_g \qquad\mbox{ for a $k$-form $\zeta$ and a $k-1$-form $\eta$, $k=1,2$.}
\]
If we instead integrate over a subset of $S$ of the form
$S\setminus {\mathcal O}$, then this identity becomes
\begin{equation}\label{stokes.bdy}
\int_{S\setminus {\mathcal O}}  (d\eta, \zeta)_g \vol_g -
\int_{S\setminus {\mathcal O}}  (\eta, d^*\zeta)_g \vol_g   =  - \int_{\partial {\mathcal O}} \eta\wedge \star \zeta
\end{equation}
where we consider $\partial {\mathcal O}$ to have the orientation inherited from ${\mathcal O}$ (rather than $S\setminus {\mathcal O}$, hence the minus sign on the right-hand side). 
If $\eta$ is a $0$-form then  we will omit the wedge on the right-hand side of \eqref{stokes.bdy}.

For $p\in S$, we will write $\delta_p$ to denote the (measure-valued) $2$-form
 such that
\[
\int_S f \delta_p = f(p)\qquad\mbox{ for every continuous $f:S\to \R$}
\]
If $A$ is a $1$-form on $S$ and $v\in T_xS$, then
$A(v)$  denotes the number obtained via the action of the $1$-covector
$A_x\in T_x^*S$ on $v\in T_xS$. If $v$ is a vector field, then $A(v)$ denotes the function
whose value at $x$ is $A(v(x))$.

\subsection{The connection $1$-form} 
\label{subsec:connect} 
A {\em moving  frame}
on an open subset ${\mathcal O}\subset S$ will mean a pair of smooth, properly oriented, orthonormal vector fields $\tau_k\in \calX({\mathcal O})$ for
$k=1,2$, i.e., 
\[
(\tau_k,\tau_\ell)_g  = \delta_{k\ell} 
\qquad\qquad \vol_g(\tau_1,\tau_2) =1
\]
everywhere in ${\mathcal O}$.
Note that if $\tau$ is any smooth unit vector field on ${\mathcal O}$,
then  $\{\tau_1,\tau_2\} = \{ \tau, i\tau\}$ provides  a moving frame, 
and if $\{\tau_1,\tau_2\}$ is any moving frame, then $\tau_2=i\tau_1$.
In general a moving frame exists only locally on $S$.

On an open subset ${\mathcal O}\subset S$,
we will define the {\em connection 1-form} $A$ associated to a moving frame $\{ \tau_1,\tau_2\}$ by
$$
A(v) = (\dd_v\tau_2,\tau_1)_g = - (\dd_v \tau_1,\tau_2)_g, \qquad v\in \calX({\mathcal O}).
$$
Since 
$
0 = v(\tau_k, \tau_k)_g = 2 (\dd_v \tau_k, \tau_k)_g$ for $k=1,2$,
it follows that 
$\dd_v \tau_1 = -A(v)\tau_2$ and $\dd_v\tau_2 = A(v)\tau_1$.
Note that if $\{\tau_1,\tau_2\}$ is a moving frame on ${\mathcal O}\subset S$, then $A = - j(\tau_1)$
on ${\mathcal O}$ where $j(\tau_1)$ is the $1$-form defined in \eqref{j.def}. 
In complex notation, this fact and the Leibniz rule imply that
for any smooth complex-valued function $\phi$ on ${\mathcal O}$, 
\begin{equation}
\dd_v( \phi \tau_1 ) =   (d\phi(v) - iA(v) \phi) \tau_1.
\label{DandA}\end{equation}
The definition of $A$
is clearly independent of any coordinate system on ${\mathcal O}$ (since our definition does not
refer to any coordinates) but depends on the choice of a frame. 
However, it is a standard
fact that $dA$ is   independent of the frame. In particular, we have the identity
\beq
\label{rel_connection}
dA = \kappa \ \vol_g
\eeq
where $\kappa$ is the Gaussian curvature of $S$.
(See do Carmo \cite{Car94}, Proposition
2 on page 92; our $1$-form $A$ is written as $-\omega_{12}$ in
do Carmo's notation, see \cite{Car94} p. 94.)  In fact, this  may be
taken as the definition of Gaussian curvature. 
We recall several attributes of the Gaussian curvature. First,
the Gauss-Bonnet Theorem states that
$$
\int_S \kappa \vol_g = 2\pi \chi(S)
$$
where $\chi(S)$ is the Euler characteristic. (For a proof, with the definition
of the Euler characteristic, consult for example \cite{Car94}, section 6.1.)
Another classical fact that we will use is the  {\em Bertrand-Diguet-Puiseux Theorem},
which says that
$$
\kappa(P) = 
\lim_{r\searrow 0}  3\frac { 2\pi r - \calH^1(\partial B_r(P))}{ \pi r^3}
= 
\lim_{r\searrow 0} 12\frac { \pi r^2 - \mbox{Vol}_g(B_r(P))}{\pi r^4}.
$$

\subsection{Sobolev spaces} 
\label{sec:soboesp}
For $q\in [1, \infty]$, we define $L^{q}(S;\R)$ the space of $q$-integrable functions w.r.t. the volume form $\vol_g$ and the Sobolev spaces 
$$W^{1,q}(S;\R)=\{ f\in L^{q}(S;\R)\, : \,  
 \| f\|_{W^{1, q}}:= \max \{ \|f\|_{L^q}, \|df\|_{L^q}\}
<\infty\}.$$
If $\mu$ is a $2$-form (possibly measure-valued) then we write for $p, q\in [1,\infty]$ with 
$\frac 1p + \frac 1q =1$ that $W^{-1,p}$ is the dual of the Sobolev space $W^{1,q}$, i.e., 
\[
\| \mu\|_{W^{-1,p}} := \sup \left\{  \int_{S} f \mu  \ : \ f\in W^{1,q}(S;\R), \,  \| f\|_{W^{1, q}}\le 1 \right\}.
\]
We also recall the Hodge decomposition.
The following version will suffice for us:
if $\zeta$ is any square-integrable $1$-form on $S$, then there exist
a 
$0$-form $\xi$,  $2$-form $\beta$, and $\eta\in Harm^1(S)$ (see \eqref{harmon}) such that
\begin{equation}\label{Hodge}
\zeta = d\xi+ d^*\beta + \eta.
\end{equation}
Moreover, this decomposition is unique.
By integrating by parts one easily sees that for any
$0$-form $\xi$,  $2$-form $\beta$, and $\eta\in Harm^1(S)$,
one has
$$
\int_S (d\xi, d^*\beta)_g \vol_g \ = \ 
\int_S (d\xi, \eta )_g \vol_g \ = \ 
\int_S ( d^*\beta, \eta)_g \vol_g \ = \  0.
$$
and it follows that if \eqref{Hodge} holds, then
\[
\| d \xi \|_{L^2}^2 + \|d^*\beta\|_{L^2}^2 + \|\eta\|_{L^2}^2  = \|\zeta\|_{L^2}^2.
\]

We have the following density result in $\calX^{1,p}(S)$ (which is standard, see e.g. \cite{Kar}): 

 \begin{lemma} For any open ${\mathcal O}\subset S$, if $u\in \calX^{1,p}(\mathcal O)$, then for any open $\mathcal O'\subset\subset \mathcal O$ compactly supported in $\mathcal O$, there exists a  family of smooth vector fields  $(u_\eps)_{\eps\in (0,\eps_0)}\subset \calX({\mathcal O'})$ that converges	to  $u$ in $\calX^{1,p}(\mathcal O')$. If $|u|_g\le 1$ in $\mathcal O$, then one can arrange that $|u_\eps|_g\le 1$ in $\mathcal O'$ for $\eps<\eps_0$.  Moreover, if $p\ge 2$ and $|u|_g = 1$ in $\mathcal O$, then one can arrange that $|u_\eps|_g = 1$ everywhere in $\mathcal O'$.
\label{L.density}\end{lemma}

\begin{proof} Let $u\in \calX^{1,p}({\mathcal O})$. One considers a standard radial mollifier $\rho\in C^\infty(\R^2)$ such that $0\leq \rho\leq 1$, $\rho$ has support in the unit ball and $\int_{\R^2} \rho(z)\, dz=1$. For $x\in S$, we consider the exponential map $\exp_x:T_xS\to S$ and for 
$\eps\in (0, \eps_0)$ (with $\eps_0$ be the injectivity radius of $S$), let 
$$\rho_{\eps, x}(y)=\frac1{\eps^2} \rho\bigg(\frac{\exp_x^{-1}(y)}{\eps}\bigg)\quad \textrm{in a neighborhood of }\, x$$
 where we identified $T_xS$ with $\R^2$; we also consider the renormalized mollifiers
 $$\tilde \rho_{\eps, x}(y)=\frac{\rho_{\eps, x}(y)}{\int_S \rho_{\eps, x}\, \vol_g}.$$
 Now for $x\in \mathcal O$ such that $\dist_S(x, \partial \mathcal O)>\eps$, we define
 $$u_\eps(x)=\int_{\mathcal O} \tilde \rho_{\eps, x}(y) \tau_{y,x} u(y)\, \vol_g(y)\in T_xS, $$
 where $\tau_{y,x}:T_yS\to T_xS$ is the parallel transport along the shortest geodesic from $y$ to $x$.
 Then for any $\mathcal O'\subset\subset \mathcal O$, there exists $\eps_0$
 such that  
 $u_\eps\in
 \mathcal X({\mathcal O'})$ for  $0<\eps<\eps_0$,  and $u_\eps\to u$ in $\calX^{1,p}(\mathcal O')$ (see \cite{Kar} for more details). 
 Moreover, the following Poincar\'e-Wirtinger inequality holds:
 $$\int_{B_\eps(x)} |u_\eps(x)-\tau_{y,x} u(y)|_g\, \vol_g(y)\leq c \eps \int_{B_\eps(x)} |Du|_g\, \vol_g,$$
 for some universal constant $c>0$. 
 Also, note that $|u|_g\leq 1$ in $S$ implies that $|u_\eps|_g\leq 1$ in $\mathcal O'$.
  
 Assume now that $|u|_g=1$ in $\mathcal O$ and that  $p\geq 2$. As $|\tau_{y,x} u(y)|_g=|u(y)|_g=1$ a.e. in $\mathcal O$, we deduce:
 \begin{align*}
 \sup_{x\in \mathcal O'}\bigg|1-|u_\eps(x)|_g\bigg|&\leq C \sup_{x\in \mathcal O'} \frac{1}{\eps^2} \int_{B_\eps(x)} |u_\eps(x)-\tau_{y,x} u(y)|_g\, \vol_g(y)\\
 &\leq C  \sup_{x\in \mathcal O'} \frac1\eps \int_{B_\eps(x)} |Du|_g\, \vol_g \leq C  \sup_{x\in \mathcal O'} \|Du\|_{L^2(B_\eps(x))}\to 0 \, \, \textrm{as $\eps\to 0$},
 \end{align*}
 where we used the equiintegrability of $|Du|^2_g$ on $\mathcal O$. Therefore, $|u_\eps|_g\to 1$ uniformly in  ${\mathcal O'}$ as $\eps\to 0$ so that the smooth vector fields 
 $\tilde u_\eps=u_\eps/|u_\eps|_g$ are of unit length and converge to $u$ in $\calX^{1,p}(\mathcal O')$.
 \end{proof}

\subsection{A little homology}\label{S:homology}
Suppose that $\lambda_1,\ldots, \lambda_\ell$ are closed Lipschitz curves on $S$, by which 
we mean that $\lambda_k$ is a Lipschitz continuous function $[0,1]\to S$ such that 
$\lambda_k(0)=\lambda_k(1)$ for every $k$.
Given integers $c_1,\ldots, c_\ell$, we say that
\[
\sum_{k=1}^\ell c_k \lambda_k
\mbox{ is homologous to $0$}
\]
if there exists an integrable function $f:S\to \Z$ such that 
$$
\sum_k c_k\int_{\lambda_k}\phi \ = \ \int_S f \, d\phi\qquad\mbox{ for all
smooth $1$-forms }\phi.
$$
Here and below we use the notation
\[
\int_{\lambda}\phi := \int_0^1 (\lambda)^*\phi = \int_0^1 \phi_i(\lambda(s)) \, \lambda^i{}'(s)\, ds \ \mbox{ in local coordinates}
\]
We also say that $\lambda$ is homologous to $\sum_k c_k\lambda_k$ if
$\lambda - \sum c_k\lambda_k$ is homologous to $0$.

We will need a standard fact, which can be stated as follows:

\begin{lemma}
If $S$ is a compact Riemannian manifold of 
genus $\mathfrak g$,
then there exist simple\footnote{That is, non self-intersecting.} closed geodesics $\gamma_k$, for $k=1,\ldots, 2\mathfrak g$, such 
that if $\gamma$ is any closed Lipschitz curve, then there exist integers $c_1\ldots, c_{2\mathfrak g}$
such that 
$$
\gamma \mbox{ is homologous to }\sum_{k=1}^{2\mathfrak g} c_k \gamma_k.
$$
Moreover, these curves $\{\gamma_k\}_{k=1}^{2\mathfrak g}$ have the property that for $\eta\in Harm^1(S)$ defined in \eqref{harmon}, the following equivalences 
take place:
\begin{align*}
\eta = 0
&\qquad \ \ \Longleftrightarrow
\ \ 
\int_{\gamma} \eta = 0\ \ \ \ \mbox{ for every closed Lipschitz curve $\gamma$}\\
&\qquad \ \ \Longleftrightarrow
\ \ 
\int_{\gamma_k} \eta = 0\ \ \ \ \mbox{ for }k=1,\ldots, 2\mathfrak g.
\end{align*}
In particular, the matrix $\alpha=(\alpha_{\ell k})$ defined in \eqref{akl.def} is invertible.
 \label{L.FedFlem}\end{lemma}

\begin{proof} We sketch the proof for the reader's convenience.
First, it is a classical fact that the first singular homology group $H_1(S;\Z)$
of $S$ with integer coefficients is isomorphic to $\Z^{2\mathfrak g}$.
(In other words, the first Betti number of a surface of genus $\mathfrak g$ is
$2\mathfrak g$.)
 Second, the singular homology group $H_1(S;\Z)$
is isomorphic to the   homology group $H^{FF}_1(S;\Z)$ 
in the sense of Federer and Fleming (to whom this statement is
due, see \cite{FedFle} Theorem 5.11), consisting 
of integral $1$-cycles,
modulo boundaries of integral $2$-currents.
Thus $H^{FF}_1(S;\Z)$ is also isomorphic to $\Z^{2\mathfrak g}$.
Moreover, Federer and Fleming (see Corollary 9.6 in \cite{FedFle}) also show that every homology
class contains a mass-minimizing element.
Combining these facts, we may find mass-minimizing currents
$\Gamma_1,\ldots, \Gamma_{2\mathfrak g}$ 
in a collection of homology classes that generate $H^{FF}_1(S;\Z)$.
Each $\Gamma_j$ need not correspond to a simple closed
geodesic, but it follows from \cite{Federer} 4.2.25 that each $\Gamma_j$ can be 
written as a sum of ``indecomposable' currents, say $\gamma_{j,k}$ for $k=1, \ldots, \gamma_{j, \ell(j)}$,  which in the
present context (due to the minimality of $\Gamma_j$) correspond to  simple closed geodesics.
Now the collection of  homology classes corresponding to all $\{ \gamma_{j,k} \}$ generate
$H^{FF}_1(S;\Z) \cong \Z^{2\mathfrak g}$, so there must exist a subset containing exactly
$2\mathfrak g$ elements that also generates $H^{FF}_1(S;\Z)$. If we relabel the elements of
this subset as $\gamma_1,\ldots, \gamma_{2\mathfrak g}$, this 
says  that 
any integral $1$-cycle --- in particular, any 
closed Lipschitz curve --- is homologous to an integer linear
combination $\sum_{k=1}^{2\mathfrak g} c_k \gamma_k$, where ``homologous"
is understood in the sense of \cite{FedFle}, which is
our definition above.

Having found $\{ \gamma_k\}_{k=1}^{2\mathfrak g}$, it is immediate that
\[
\eta = 0 \quad \Rightarrow \quad \int_\gamma \eta = 0 \mbox{ for all closed Lipschitz $\gamma$}
\qquad \Rightarrow \quad \int_{\gamma_k}\eta = 0\mbox{ for }k=1,\ldots, 2\mathfrak g.
\]
So to establish the equivalences, we must only show that if $\eta\in Harm^1(S)$ and 
$\int_{\gamma_k}\eta = 0$ for all $k$, then $\eta =0$. To see this, note that
for each $k=1,\ldots, 2\mathfrak g$, we can identify $\gamma_k$ with the linear map $\eta \mapsto \int_{\gamma_k}\eta$ on $Harm^1(S)$. These maps are linearly independent, since they generate the $2\mathfrak g$-dimensional space $H^{FF}_1(S;\Z)$. Then the desired statement follows from the fact that $Harm^1(S)$ is $2\mathfrak g$-dimensional (noted in Section \ref{sec:canon}),  since any $\eta\in Harm^1(S)$ that is in the kernel of $2\mathfrak g$ independent linear maps
$S\to \R$ must therefore equal zero. 

Finally we prove that $\alpha$ is invertible. Assume by contradiction that there exists a vector $b\in \R^{2\mathfrak g}\setminus \{0\}$ such that $\alpha b=0$. By \eqref{akl.def}, it means that
$\int_{\gamma_\ell} \sum_{k=1}^{2\mathfrak g}  b_k \eta_k=0$ for every $\ell=1, \dots, 2\mathfrak g$. The above equivalences yields $ \sum_{k=1}^{2\mathfrak g}  b_k \eta_k=0$; as $\{\eta_k\}_k$ is a basis of $Harm^1(S)$, one has $b=0$ which is a contradiction.

\end{proof}

\subsection{Some useful calculations}
\label{sec:calc}
In this section we record some  straightforward facts that we will use repeatedly. Let $u$ be a smooth vector field
in an open set ${\mathcal O}\subset S$. 
First, note that wherever $u\ne 0$, for every smooth unit vector field $v$ we have
\[
\dd_v u = (\dd_v u, \frac{iu}{|u|_g)})_g \frac{iu}{|u|_g}  + (\dd_v u,\frac u{|u|_g})_g \frac{u}{|u|_g}  
= \frac{ j(u)(v) }{|u|_g} \frac{iu}{|u|_g} + v(|u|_g) \frac {u}{|u|_g},
\]
It follows that
\beq
\label{formula12}
|\dd u|_g^2 = \left| \frac { j(u)}{|u|_g} \right|_g^2 + | d|u|_g|^2.
\eeq
In particular, if $u$ is of unit length (i.e., $|u|_g=1$) and $\rho$ is a smooth scalar function, then 
$$|\dd u|_g^2=|j(u)|^2_g,\quad j(\rho u) = \rho^2 j(u),$$
and thus
$$
|\dd  (\rho u)|_g^2 =
\rho^2  \left|   j(u) \right|_g^2 + | d\rho|_g^2 = 
\rho^2  \left|  \dd u\right|_g^2 + | d\rho|_g^2 .
$$
Writing in complex variable $e^{i\Theta}=\cos \Theta+i\sin \Theta$ for a smooth scalar function $\Theta$ where $i$ is the isometry \eqref{isom},
then 
$$
j(e^{i\Theta}u)=j(u)+|u|^2_g \, d\Theta.
$$
The above properties generalize to suitable Sobolev spaces by a standard density argument (see Lemma \ref{L.density}).

\begin{lemma}
\label{L.last}
Let $\mathcal O$ be an open set in $S$. Then $j:\calX^{1,2}(\mathcal O)\to L^p({\mathcal O})$ is a continuous map for every $p\in [1,2)$ and $|dj(u)|_g \le |Du|_g^2$ a.e. in $\mathcal O$ for every $u\in \calX^{1,2}(\mathcal O)$. As a consequence, the map $u\in \calX^{1,2}(\mathcal O)\mapsto dj(u)$ is continuous as a map with values into 
the set of $2$-forms endowed with the $W^{-1,p}$-norm for every $p\in [1,2)$. Moreover, if $u\in \calX^{1,2}(\mathcal O)$, then $\frac{j(u)}{|u|_g}=(Du, \frac{iu}{|u|_g})_g$ is well defined and belongs to $L^2$.
\end{lemma}
\begin{proof}
If $u, v\in \calX^{1,2}(\mathcal O)$ and $p\in [1,2)$, then the H\"older inequality implies
\begin{align*}
\int_{\mathcal O} |j(u)-j(v)|_g^p\, \vol_g&\leq C\bigg( \int_{\mathcal O} |(D(u-v), iu)_g|^p\, \vol_g+\int_{\mathcal O} |\big(Dv, i(u-v)\big)_g|^p\, \vol_g\bigg)\\
&\leq C \bigg(\|D(u-v)\|^p_{L^2} \|u\|^p_{L^q}+\|Dv\|^p_{L^2} \|u-v\|^p_{L^q}\bigg)\\
&\leq 
C \|D(u-v)\|^p_{L^2}(\|Du\|^p_{L^2}+\|Dv\|^p_{L^2}), 
\end{align*}
where $q=p/(2-p)$ and we used the Sobolev embedding $\calX^{1,2}\subset L^q$.
Therefore, $j:\calX^{1,2}(\mathcal O)\to L^{p}({\mathcal O})$ is a continuous map.
As $d:L^p\to W^{-1,p}$ is continuous, we deduce that
$u\mapsto dj(u)$ is continuous as map with values into 
the set of $2$-forms endowed with the $W^{-1,p}$-norm for every $p\in [1,2)$.  

We now prove that $|dj(u)|_g \le |Du|_g^2$ a.e. in $\mathcal O$. Assume for the moment that $u$ is smooth in $\mathcal O$. Fix some $x\in \mathcal O$, and choose  (properly oriented) coordinates near $x$
such that the coordinate vector fields $\partial_{x_1}, \partial_{x_2}$ are orthonormal {\em at $x$.}
In these coordinates,
$j(u) = \sum_{k=1,2}(D_k u, iu)_g dx^k$
and thus, by the Schwartz lemma,
\beq
\label{egal_dj}
dj(u) =\sum_{k, \ell=1,2} (D_k u, iD_\ell u)_g dx^\ell \wedge dx^k =2 (i D_1 u, D_2 u)_g dx^1\wedge dx^2.
\eeq
Thus at $x$,
$$
|dj(u)|_g = 2 | (i D_1 u, D_2 u)_g | \le |D_1u|_g^2 + |D_2 u|_g^2 = |Du|_g^2,
$$
where we have used several times the choice of coordinates, which implies that
 $dx^1, dx^2$, are orthonormal at $x$, in particular that $dx^1\wedge dx^2 = \vol_g$. 
In the general case, by a standard density argument (via Lemma \ref{L.density}), one deduces that the above inequality holds a.e. in $\mathcal O$ for every $u\in \calX^{1,2}(\mathcal O)$.
The last part of the statement follows from \eqref{formula12}.
\end{proof}

As a consequence, we have the following:
\begin{lemma}\label{L.omega0}
Assume that ${\mathcal O}$ is an open subset of $S$ and that $u\in \calX^{1,2}({\mathcal O})$
satisfies $|u|_g=1$.
Then
\[
d j(u) = -\kappa \vol_g \qquad\mbox{in }{\mathcal O}.
\]
In particular, we have  $\omega(u)=0$ in ${\mathcal O}$.
\end{lemma}

\begin{proof}
If $u$ is smooth in ${\mathcal O}$, then we define $\tau_1 = u$ and $\tau_2 = i \tau_1$ in ${\mathcal O}$,
and the definitions imply that the connection $1$-form associated to this
choice of orthonormal frame is exactly $A = - j(u)$. So the conclusion
follows immediately as $dA=\kappa \vol_g$. For general $u\in \calX^{1,2}({\mathcal O})$ of unit length, one argues by density (see Lemma \ref{L.density}) and the continuity properties of $j(\cdot)$ in Lemma \ref{L.last}.
\end{proof}

\bigskip

\section{The canonical harmonic vector field. Proof of Theorem \ref{P1}}\label{S:chm}

In this section we consider $(S,g)$ to be a $2$-dimensional closed oriented Riemannian manifold
(not assumed to be embedded in any Euclidean space).
We will need the following

\begin{lemma}\label{L.frame}
If $B$ is any nonempty open subset of $S$, then there exists a moving frame $\{ \tau_1,\tau_2\}$
on $S\setminus B$.
\end{lemma}

\begin{proof}
A standard construction (see for example \cite{Car94} pages 103-4) yields a
smooth vector field 
that does not vanish outside some finite set (the vertices of a triangulation of $S$). After pushing forward via a 
diffeomorphism of $S$ that maps every point of this finite set into $B$,
we get a vector field $v$ such that $|v|_g>0$ outside $B$.
Then we obtain a moving frame on $S\setminus B$ by
setting $\tau = v/|v|_g$ and  $\{ \tau_1,\tau_2\}  = \{ \tau, i\tau\}$.
\end{proof}

\begin{lemma}\label{L.jsb}
Let $\gamma$ be any closed Lipschitz curve on $S$.
If $\{ \tau_1,\tau_2\}$ and $\{\tilde \tau_1,\tilde \tau_2\}$ are moving frames defined 
in a neighborhood of $\gamma$, and $A$ and $\tilde A$ are the associated connection $1$-forms, then
\[
\int_\gamma A = \int_\gamma \tilde A \quad \mod 2\pi.
\]
\end{lemma}

\begin{proof}
In the domain where they are both defined, there exists a smooth $\C$-valued function
$\phi$ such that $\tilde \tau_1 = \phi\tau_1$, since $\{\tau_1(x),\tau_2(x)\} = \{\tau_1(x),i\tau_1(x)\}$ form a basis
for $T_xS$. It then follows that $|\phi|=1$ everywhere and that $\tilde \tau_2 = \phi\tau_2$ as well.
If we write $\phi = \phi_1+i\phi_2$,  the definition of
the connection $1$-form together with Section \ref{sec:calc} 
imply that $A - \tilde A =  \phi_1 d\phi_2-\phi_2 d\phi_1 =: (d\phi, i\phi)$.

Next, it is convenient to abuse notation and write $\gamma$ to denote both the curve in
$S$ and a Lipschitz 
function $\gamma:[0,1]\to S$, with $\gamma(0)=\gamma(1)$, that parametrizes the
given curve, with the correct orientation. We will also write $\varphi = \phi\circ \gamma: [0,1]\to \SSS^1\subset \C$. Clearly $\varphi$ is Lipschitz, so we can find a Lipschitz function
$f: [0,1]\to \R$ such that $\varphi(s) = e^{if(s)}$ for $s\in [0,1]$. Then one readily checks that
\[
\int_\gamma(A-\tilde A) = \int_\gamma (d\phi, i\phi) = \int_0^1  (\varphi'(s), i\varphi(s)) ds 
= \int_0^1 f'(s)ds = f(1)- f(0)\in 2\pi \Z
\]
since $\varphi(0) = \varphi(1)$.
\end{proof}

As a consequence, we deduce that the index (or topological degree) defined in \eqref{deg.def} is an integer number:
\begin{lemma}
\label{L.degree}
Let ${\mathcal O}$ be a simply connected open subset of $S$ of nonempty Lipschitz boundary  and
$u\in \calX^{1,2}({\mathcal N})$ is a vector field in a neighborhood ${\mathcal N}$ of $\partial {\mathcal O}$ such that $|u|_g\ge \frac 12$ a.e. in ${\mathcal N}$;
then the {\em index} of $u$ along $\partial {\mathcal O}$ defined in \eqref{deg.def} is well defined and it is an integer.
\end{lemma}
\begin{proof}
We start by explaining why the definition \eqref{deg.def} makes sense for $u\in \calX^{1,2}({\mathcal N})$. In fact, if $\{\tau, i \tau\}$ is a moving frame in ${\mathcal N}\cup \mathcal O$ (which exists due to Lemma \ref{L.frame} as by our assumption $S\setminus \mathcal O$ has nonempty interior) and $\tilde u=u/|u|_g$, 
then $\tilde u=\phi \tau$ for some $\phi\in H^1(\mathcal N, \SSS^1)$. Denoting by $A$ the connection $1$-form associated to the frame, by \eqref{DandA}, we have that $D\tilde u=(d\phi-iA\phi)\tau$ so that $j(u)/|u|_g^2=j(\tilde u)=(d\phi, i\phi)-A$ where $(d\phi, i\phi)=\phi_1 d\phi_2-\phi_2 d\phi_1$ is the current associated to the unit-length  complex function $\phi$ belonging to $H^{1/2}(\partial \mathcal O, \SSS^1)$ by the trace theorem. Therefore, since $dA=\kappa \vol_g$, the 
Stokes theorem implies that \eqref{deg.def} writes
$$2\pi \deg(u; \partial {\mathcal O})=\int_{\partial {\mathcal O}}(d\phi, i\phi)-\int_{\partial {\mathcal O}} A+\int_{\mathcal O} \kappa \vol_g=\int_{\partial {\mathcal O}}(d\phi, i\phi)$$
where the meaning of the last term is given by the duality $(H^{-1/2}(\partial {\mathcal O}), H^{1/2}(\partial {\mathcal O}))$. Moreover, it is known (see \cite{BGP, BN}) that
this number is a multiple of $2\pi$ as long as $\phi \in H^{1/2}(\partial \mathcal O, \SSS^1)$.
\end{proof}

The following lemma is a main point in the proof of Theorem \ref{P1}.

\begin{lemma}\label{L.ns}
Let $u$ be a smooth unit vector field defined on an open set $\mathcal O\subset S$.
If $\gamma$ is any smooth closed curve in $\mathcal O$, and if $A$ is the connection $1$-form
associated to any moving frame defined in a neighborhood of $\gamma$, then
\begin{equation}\label{ns1}
\int_\gamma (j(u) + A) \in 2\pi \Z.
\end{equation}
Conversely, if $j$ is a smooth $1$-form in an open set $\mathcal O\subset S$ such
that 
\begin{equation}\label{ns2}
\int_\gamma (j+A) \in 2\pi \Z
\end{equation}
for any curve $\gamma$ and connection $1$-form $A$ as above, then there exists a smooth unit vector field $u$ in the open set $\mathcal O$,
such that $j(u)=j$.
\end{lemma}

\begin{proof} The first part is a direct consequence of Lemma \ref{L.jsb} as $\{u, iu\}$ is a moving frame around 
$\gamma$ to which the connection $1$-form $\tilde A$ is associated so that $j(u)=-\tilde A$. However, we give in the following a different proof that is needed for the last part of the statement.
Let $u$ be a smooth unit vector field on $\mathcal O\subset S$.
For simplicity we write $j_u := j(u)$.\\

\nd {\it Step 1. An ODE argument}. Fix some smooth curve $\gamma:[0,1]\to \mathcal O$ with $\gamma(0) = \gamma(1)$ and
for $s\in [0,1]$, let $U(s) := u(\gamma(s)) \in T_{\gamma(s)}S$. 
Then for $s\in [0,1]$,
we have
\begin{align}
D_{\gamma'(s)}U(s) 
&= 
(D_{\gamma'(s)}U(s)  , U(s))_g \, U(s) +  
(D_{\gamma'(s)}U(s)  , iU(s))_g \, iU(s) \nonumber \\
&= j_u(\gamma'(s)) \,iU(s)\, ,
\label{ode1}\end{align}
since $0 = \frac d{ds}|U(s)|_g^2 = 2(D_{\gamma'}U(s), U(s))_g$.
(We remind the reader of our convention that if $j_u$ is a $1$-form and  $v\in T_x S$, then $j_u(v)$ denotes
$j_u|_x(v)$.)
Now let $\{\tau_1,\tau_2\} = \{ \tau, i\tau\}$ be any moving frame defined in a neighborhood of $\gamma$, and let $A$
be the connection $1$-form associated to it. 
Writing $U(s)$ in terms of the frame, we have
$$
U(s) = \phi(s)\tau(s) = (\phi_1(s) + i \phi_2(s))\, \tau(s)
$$
where $\tau(s) := \tau(\gamma(s))$ and $\phi_j(s) = (U(s), \tau_j(s))_g$,
$j=1,2$. Using \eqref{DandA} to rewrite the ODE \eqref{ode1} in terms of $\phi$, 
we obtain
\[
\phi'(s) = (j_u + A)(\gamma'(s)) \ i \phi(s) .
\]
We solve to find that 
\begin{equation}
U(s) = \phi (s)\tau(s) = \phi(0) \exp\left[ i \int_0^s (j_u+A)(\gamma'(t)) \, dt \right]\tau(s),
\label{holonomy0}\end{equation}
for $0<s\le 1$.
Since $\gamma(0)=\gamma(1)$, however, it must be the case that $U(0)=U(1)$, 
and thus 
\[
 \int_0^1 (j_u+A)(\gamma'(t)) \, dt 
 =
 \int_\gamma (j_u+A) = 0\mod 2\pi.
\]
This proves \eqref{ns1}.\\

\nd {\it Step 2. Strategy}. To establish the converse, we now assume that $j$ satisfies \eqref{ns2} on an 
open set $\mathcal O$.
We may assume that $\mathcal O$ is connected, as otherwise we may follow the procedure described below on every connected component.
Now fix some $x\in \mathcal O$ and $v\in T_xS$ such that $|v|_g=1$.
Given any other $y\in \mathcal O$, we define $u(y)$ by the following procedure:
Fix a smooth curve $\gamma:[0,1]\to \mathcal O$ such that 
$\gamma(0)= x, \gamma(1) = y$. 
If $u$ exists, then $u(\gamma(s))$ must satisfy the ODE \eqref{ode1} found above. 
Motivated by this, 
we  let $U(s) \in T_{\gamma(s)}S$ be the solution of \eqref{ode1} with initial data as below:
\begin{equation}
D_{\gamma'(s)} U(s) =j( \gamma'(s)) \ iU(s),
\qquad U(0) = v.
\label{transport}\end{equation}
We  hope to define
\[
u(y) := U(1).
\]

\nd {\it Step 3. Independence of the connecting path}. We must verify that the above definition makes sense (in particular, is independent
of the choice of path connecting $x$ to $y$). 
For this, it suffices to show that 
for any piecewise smooth curve $\gamma:[0,1]\to \mathcal O$ 
such that $\gamma(0)=\gamma(1)$, 
$$
\mbox{ if $U$ solves \eqref{transport} along $\gamma$, \ \ \ then } 
U(0)=U(1).
$$
Indeed, if $\gamma_1$ and $\gamma_2$ are two such curves joining $x$ to $y$,
then
$$
\gamma(s) := \begin{cases}
\gamma_1(1-2s)&\mbox{ if }0\le s \le \frac 12\\
\gamma_2 (2s -1)&\mbox{ if }\frac 12 \le s \le 1
\end{cases}
$$
is a piecewise smooth curve beginning and ending at $y$ and passing through $x$ when $s=1/2$.
If we consider the solution of \eqref{transport} such that $U(\frac 12) = v\in T_xS$,
then $U(0)-U(1)$ characterizes the difference between the vectors obtained
by transporting $v$ from $x$ to $y$, using the ODE \eqref{transport},
along $\gamma_1$ and $\gamma_2$.

Now, exactly as above, by writing \eqref{transport} in terms of a moving
frame $\{ \tau_1,\tau_2\} = \{\tau, i\tau\}$
and solving the resulting equation, we find that \eqref{holonomy0} holds, and
thus that
$U(0)=U(1)$ if and only if  \eqref{ns2} is satisfied.
Thus the above procedure gives a well-defined vector
field $u$ on $\mathcal O$, which is clearly a unit vector field in view of \eqref{holonomy0}.
\\

\nd {\it Step 4. Smoothness of $u$ and $j_u=j$}. 

As $j$ is smooth and generating $u$ via \eqref{transport}, by regularity of ODEs w.r.t. change of parameters and initial data, we deduce that $u$ is smooth in $\mathcal O$. It remains to check that $j(u) = j$. Again we will write $j_u$ instead of $j(u)$.
Given any $y\in \mathcal O$ and $v\in T_yS$,  fix a smooth curve 
$\gamma:[0,1]\to  \mathcal O$ such that 
\[
\gamma(0) = x,
\qquad
\gamma(1)=y,\qquad
\gamma'(1)=v.
\]
Let $U(s)\in T_{\gamma(s)}S$ solve the ODE \eqref{transport}. By construction,
$U(s) = u(\gamma(s))$ for all $s$.
Then at the point $y$ (corresponding to $s=1$) we have
\[
j_u(v) \overset{\eqref{j.def} }= (\dd_v u, iu)_g =  (\dd_{\gamma'}U, iU)_g 
\overset{\eqref{transport}}
= j(\gamma')
= j(v).
\]
Since $v$ was arbitrary, it follows that $j(u) = j_u = j$, completing the proof.
\end{proof}

Before proving Theorem \ref{P1}, we need the following result:
\begin{lemma}\label{L.buxtehude}
Assume $a_1, \dots, a_n$ be $n$ distinct points in $S$, $d_1, \dots, d_n\in \Z$ such that \eqref{necessary} is satisfied and let $\psi$ be the zero average $2$-form solving \eqref{psi.def}.
Let $\lambda_\ell$, $\ell = 1,\ldots, J$ be
closed Lipschitz curves in $S$, all disjoint from the set $\cup_{k=1}^n \{a_k\}$ 
appearing in \eqref{psi.def}, and such that
$\sum_{\ell =1}^{J} \tilde d_\ell \lambda_\ell$ is homologous to $0$, 
for some integers $\tilde d_1,\ldots, \tilde d_{J}$.
Finally, let $\{\tau_1,\tau_2\}$ be a moving frame defined in a neighborhood of $\cup_{\ell=1}^J \lambda_\ell$,
and let $A$ be the connection $1$-form associated to it.
Then
\begin{equation}\label{jcb}
\sum_{\ell=1}^J \tilde d_\ell \int_{\lambda_\ell} (d^*\psi + A)  = 0 \mod 2\pi.
\end{equation}
\end{lemma}

\begin{remark}
The proof shows that the conclusion of the Lemma still holds if 
\[
-\Delta\psi +\kappa \vol_g = \omega
\]
where 
$\omega$ is a $2$-form supported in a union $\cup_{k=1}^n B_k$ of disjoint balls
such that $\int_{B_k}\omega=2\pi d_k$ for every $1\leq k\leq n$, \eqref{necessary} holds and
the curves $\{\lambda_\ell\}_{1\leq \ell\leq J}$ are disjoint from  $\cup_{k=1}^n \bar B_k$.
\label{R.buxtehude}\end{remark}

\begin{proof}
The assumption that $\sum \tilde d_\ell \lambda_\ell$ is homologous to $0$ means that
there exists an integrable function $f:S\to \Z$ such that
\begin{equation}
 \sum_{k=1}^{J} \tilde d_\ell\int_{\lambda_\ell}\phi =  \int_S f \, d\phi
\qquad \mbox{ for every smooth $1$-form $\phi$.}
\label{dcmpse}\end{equation}
Since we can add a constant to $f$ without changing the integral in \eqref{dcmpse},
we may also assume that $f = 0$ on an open set $B$. After shrinking $B$
if necessary, we may assume that its closure does not
intersect $\cup_{\ell=1}^J \lambda_\ell$. Then, according to Lemma \ref{L.frame},
there exists a moving frame $\{\tau_1,\tau_2\}$ defined on a neighborhood of the support of
$f$. Let $A$ denote the associated connection $1$-form. In view of  Lemma \ref{L.jsb},
it suffices to prove \eqref{jcb} for this choice of $A$.
We wish to substitute $\phi = d^*\psi+A$ in \eqref{dcmpse} (but $\psi$ is not smooth on $S$) to find that
\[
\sum_{\ell=1}^J \tilde d_\ell \int_{\lambda_\ell} (d^*\psi + A) =
 \int_S f \, d(d^*\psi+A) \in 2\pi \Z \, ,
\]
since $f$ is integer-valued and $d(d^*\psi+A) = -\Delta \psi + \kappa \vol_g  = 2\pi \sum_{k=1}^n d_k\delta_{a_k}$, according 
to \eqref{psi.def}.
To justify this, we approximate $\psi$ by smooth functions proceeding as follows.
First, it is a standard fact that if $\int f d\phi = 0$ for all smooth $1$-forms with support 
in an open set $U$, then $f$ is constant\footnote{Modifying $f$ on a null set, 
if necessary, we assume that
$f(x) = \lim_{r\to 0} \barint_{B(r,x)}f(y) \vol_g$ wherever this limit exists. 
}
in $U$. 
It thus follows from \eqref{dcmpse} that $f$ is locally constant away from $\cup \lambda_\ell$, 
and in particular in a neighborhood of each $a_k$.
For $0<\sigma < \frac 12 \min_{j\ne k} \dist_S(a_j,a_k)$, 
let $Q_\sigma$ be a smooth function supported in 
$\cup_{k=1}^n B(a_k, \sigma)$,  with $d_k Q_\sigma \ge 0$ inside $B(a_k, \sigma)$, 
and  such that $\int_{B(a_k, \sigma)} Q_\sigma \vol_g = d_k$ for every
$k$ and $\sigma$, 
and let  $\psi_\sigma$ solve
\[
-\Delta \psi_\sigma = -\kappa \vol_g + 2\pi Q_\sigma  \vol_g\ .
\]
Then \eqref{dcmpse} implies that for every $\sigma>0$,
\[
\sum_{\ell=1}^{J} \tilde d_\ell \int_{\lambda_\ell}(d^*\psi_\sigma+A) = 2\pi \int_S f  Q_\sigma \vol_g .
\]
The last integral belongs to $2\pi \Z$ for every  $\sigma< \min_k \dist_S(a_k, \cup_\ell \lambda_\ell)$, and 
standard theory (for example,  properties of the Green's function
recalled in Section \ref{sec:renorm}) implies that 
$d^*\psi_\sigma\to d^*\psi$ as $\sigma\searrow 0$, locally uniformly away from $\{a_k\}$.
Thus we deduce \eqref{jcb}
by taking the limit $\sigma\searrow 0$  .
\end{proof}

We can now give the main result of this section: 

\begin{proof}[Proof of Theorem \ref{P1}]

Let $\psi = \psi(a;d)$ solve \eqref{psi.def}, for fixed
$d_1,\ldots, d_n\in \Z$ and distinct $a_1,\ldots, a_n\in S$
such that  \eqref{necessary} holds. Let $j^*=j^*(a,d,\Phi)$ be defined by \eqref{form_jstar},
that is, $j^* = d^*\psi + \sum_{k=1}^{2\mathfrak g} \Phi_k \eta_k$. \\

\nd {\it Step 1.Definition of $\zeta_k(a;d)$ and its consequences}. 
We recall the definition of $\zeta_k(a;d)$.
For every $k=1,\ldots, 2\mathfrak g$, we let $\lambda_k$ be a smooth
curve that  is homologous to $\gamma_k$ (the geodesics
found in Lemma \ref{L.FedFlem}) and disjoint from $\{a_l\}_{l=1}^n$.
We now define  $\zeta_k(a;d)\in \R/2\pi\Z$  by
\eqref{zetak.def}, i.e.,
$$
\zeta_k(a;d):=\int_{\lambda_k} (d^*\psi + A)  \,   \mod 2\pi, \quad k=1,\ldots, 2\mathfrak g,
$$
where $A$ is the connection $1$-form associated to any moving frame defined in
a neighborhood of $\lambda_k$.
It follows from Lemmas \ref{L.jsb} and \ref{L.buxtehude} that
the above integral is independent, modulo $2\pi\Z$,
of the choice of moving frame and of the curve
$\lambda_k$ homologous to $\gamma_k$,
and hence that $\zeta_k$ is well-defined as an element of $\R/2\pi\Z$.

With this choice of $\zeta_k$, we deduce from \eqref{form_jstar} that
$$
\int_{\lambda_k} (j^*+A) = \int_{\lambda_k} \big(d^*\psi +A +\sum_{\ell=1}^{2\mathfrak g} \Phi_\ell \eta_\ell\big)=
\zeta_k+\sum_\ell \alpha_{k\ell}\Phi_\ell.
$$
where $\{\alpha_{k\ell}\}$ were defined in \eqref{akl.def}. 
Also, it follows from Lemma \ref{L.FedFlem} that
any $\gamma$ is homologous to a linear combination of $\gamma_1,\ldots, \gamma_{2\mathfrak g}$, and
hence to a linear combination of $\lambda_1,\ldots, \lambda_{2\mathfrak g}$,
say $\sum_{k=1}^{2\mathfrak g} \tilde d_k \lambda_k$. Then
Lemma \ref{L.buxtehude} implies that
\[
\int_\gamma (j^*+A) = \sum_{k=1}^{2\mathfrak g}\tilde d_k \int_{\lambda_k}(j^*+A)
=
\sum_{k=1}^{2\mathfrak g}\tilde d_k \left(\zeta_k+\sum_\ell \alpha_{k\ell}\Phi_\ell
\right).
\] 
It follows that for $\zeta_k(a;d)$ as defined above, $j^*(a,d, \Phi)$ satisfies 
\begin{equation}\label{lattice3}
\begin{aligned}
\int_\gamma(j^*+A)=0 &\mod 2\pi \quad \mbox{for every Lipschitz path $\gamma$ in } S\setminus \cup_{l=1}^n \{a_l\} \\
& \iff \qquad\qquad  \sum_{\ell=1}^{2\mathfrak g} \alpha_{k\ell}\Phi_\ell +\zeta_k = 0\mod 2\pi
\mbox{ for all }1\leq k\leq 2\mathfrak g,
\end{aligned}
\end{equation}

\smallskip
\nd {\it Step 2. First implication}. Assume that $u^*$ is a unit vector field satisfying  \eqref{ustar2} and \eqref{ustar1}.
These conditions and the equation \eqref{psi.def} for $\psi$ imply that $j(u^*) - d^*\psi$
is a harmonic $1$-form, and it follows that $j(u^*) =: j^*(a,d,\Phi)$
for certain constants $\Phi_k$. 
Then by combining \eqref{ns1}  and \eqref{lattice3}, we conclude that
$\sum_\ell \alpha_{k\ell}\Phi_\ell +\zeta_k = 0\mod 2\pi$ for every $k$,
which is \eqref{lattice}.

\smallskip
\nd {\it Step 3. Converse implication}.
Fix constants $(\Phi_k)$ satisfying \eqref{lattice}.  By combining 
\eqref{lattice3} and the sufficiency assertion from Lemma \ref{L.ns},
we conclude that there exists a smooth
unit vector field  $u^*$ in $\mathcal O := S\setminus \cup_{l=1}^n \{a_l\}$
satisfying $j(u^*)=j^*$ so that \eqref{form_jstar} is fulfilled. 

\smallskip
\nd {\it Step 4. Continuity of $\zeta_k$, $k=1, \dots, 2\mathfrak g$}.
To prove the continuity of $\zeta_k$,  consider a sequence $\mu_t$ as in \eqref{mul.to.mu},
and let $\nu_t := \mu_t-\mu_0$ with $t>0$ small. Then \eqref{mul.to.mu} and basic properties of the 
$W^{-1,1}$ norm imply that $\nu_t$ can be written in the form
\[
\nu_t = \sum_{l=1}^{K_t} 2\pi (\delta_{p_{l,t}} - \delta_{q_{l,t}}), \qquad 
\mbox{ with }
\sum_l \dist_S(p_{l,t}, q_{l,t}) \to 0 \quad \textrm{as } \, t\to 0 
\]
and $\{K_t\}_{t\to 0}$ is uniformly bounded.
For sufficiently small $r>0$, whenever $t>0$ is small enough,  we can find 
Lipschitz paths
$\lambda_{k,t}$, for $k=1,\ldots, 2\mathfrak g$, such that
\[
\dist_S(\lambda_{k,t} , \{ p_{l,t}, q_{l,t}\}_l ) \ge r ,\qquad
\lambda_{k,t} \mbox{ is homologous to }\gamma_k,
\qquad\mbox{ and }\quad \calH^1(\lambda_{k,t}) \le C 
\]
for all $k$ (where $\gamma_k$ are the geodesics fixed at Lemma \ref{L.FedFlem}). In fact, the curves $\lambda_{k,t}$ can be considered to be the geodesics $\gamma_k$, and whenever they pass through
$B(p_{l,t} , 2r)$ or $B(q_{l,t} , 2r)$, replace that portion of the path with an arc of the circle $\partial B(p_{l,t} , 2r)$ or 
$\partial B(q_{l,t} , 2r)$.
By \eqref{newpsi}, we write for $t>0$ small:
$$\psi_t=2\pi \sum_{l=1}^{K_t} \bigg[\big(G(\cdot, p_{l,t}) - G(\cdot, q_{l,t})\big)\vol_g\bigg],$$ so that
for every $k=1,\ldots, 2\mathfrak g$, the definition \eqref{zetak.def} of $\zeta_k$ 
implies that
\[
2\pi \sum_{l=1}^{K_t} \int_{\lambda_{l,t}} d^* \bigg[\big(G(\cdot, p_{l,t}) - G(\cdot, q_{l,t})\big)\vol_g\bigg]  = \zeta_k( a_t, d_t) - \zeta_k(a_0, d_0) \mod 2\pi.
\]
But facts about the Green's function summarized in Section \ref{sec:renorm} imply that 
$d^*\bigg[\big(G(\cdot, p) - G(\cdot, q)\big)\vol_g\bigg]\to 0$ as $\dist_S(p,q)\to 0$, uniformly in the set $\{(x,p,q) \, : \, 
\dist_S(x, \{p,q\}) \ge r\}$. 
Hence the sum of integrals on the left-hand side above tends to $0$ as $t\to 0$,
which is what we needed to prove.

\medskip
\nd {\it Step 5. Uniqueness (modulo a global rotation) of $u^*$}. Assume that 
$u^*$ and 
$\tilde u^*$ are two solutions of \eqref{ustar2} and \eqref{ustar1} such that $j(u^*) = j(\tilde u^*)$.
Fixing $x$ and $v = u^*(x)$ as at the start of the construction of $u^*$ (in \eqref{transport}).
Since both $v$ and $\tilde v := \tilde u^*(x)$
are unit vectors, there exists some $\alpha$ such that 
$\tilde v = e^{i\alpha} v$. Then by inspection we see that 
if $\gamma$ is any Lipschitz curve avoiding the points $\cup_{k=1}^n\{a_k\}$, then
$\tilde U(s) = e^{i\alpha} U(s)$ solves \eqref{transport} with initial data
$\tilde U(0)  = \tilde v$. It follows that $\tilde u^*(y) = e^{i\alpha}u(y)$ for
every $y\not\in \cup_{k=1}^n\{a_k\}$. Thus $\tilde u^* = e^{i\alpha}u^*$ a.e. in $S$.

\medskip
\medskip
\nd {\it Step 6. Regularity}. Standard estimates, such as those recalled in Section \ref{sec:renorm}
for example, imply that Green functions belong to $W^{1,p}$ for all $p<2$ and smooth away from $\cup_k \{a_k\}$ which by \eqref{newpsi} it leads to $\psi$ being in the same Sobolev space and smooth away from $\cup_k \{a_k\}$. Moreover, \eqref{form_jstar} in combination with $|Du^*|_g = |j(u^*)|_g$ (by \eqref{ode1}) yields $u\in \calX^{1,p}(S)$ for all $p<2$. As $j(u^*)$ is smooth away from 
$\cup_k\{a_k\}$, then Lemma \ref{L.ns} through the construction \eqref{ode1} yield
$u^*$ is smooth away from $\cup_k\{a_k\}$. 
\end{proof}

We also prove the estimate \eqref{dist_lat}:

\begin{proof}[Proof of Lemma \ref{lem:latti}] 
First note from \eqref{setL} that there exists some $C = C(\alpha)$ such that
\[
\dist_{\R^{2\mathfrak g}} \big(\calL(a; d),\calL(\tilde a; \tilde d)\big)\leq C
\qquad\mbox{ for all }(a;d), (\tilde a, \tilde d).
\]
It therefore suffices to prove \eqref{dist_lat} under the assumption that 
$\| \mu -\tilde \mu\|_{W^{-1,1}}\le 1$.
As in Step 4 in the proof of Theorem \ref{P1}, we can rewrite $\mu-\tilde \mu=2\pi  \sum_{l=1}^{\hat n} (\delta_{p_{l}} - \delta_{q_{l}})$ for the dipoles $\{p_l, q_l\}_{l=1}^{\hat n}\subset a\cup \tilde a$ with $\hat n\leq 2K$. It follows from our specific choice of the $W^{-1,1}$ norm (see Section \ref{sec:soboesp} and the fact that $\| \mu -\tilde \mu\|_{W^{-1,1}}\le 1$
(see \cite{BCL}) that  the $W^{-1,1}$ norm of  $\mu-\tilde \mu$ represents the minimal connection 
\[
\|\mu-\tilde \mu\|_{W^{-1,1}} = 2\pi\min_{\sigma\in \calE_{\hat n}}  \sum_{l=1}^{\hat n} \dist_{S}(p_{l}, q_{\sigma(l)}),\]
where $\calE_{\hat n}$ is the set of permutations of $\hat n$ elements. After relabelling, we can assume that an optimal permutation $\sigma$ is the identity. For sufficiently small $r>0$, we can find 
Lipschitz paths
$\lambda_{k}$ homologous to $\gamma_k$ (where $\gamma_k$ are the geodesics fixed at Lemma \ref{L.FedFlem}) and of uniformly bounded length, for $k=1,\ldots, 2\mathfrak g$, such that
$\dist_S(\lambda_{k} , \{ p_{l}, q_{l}\}_l ) \ge r$ for all $k$. If we denote by $\psi(a;d)$ and $\tilde \psi(\tilde a, \tilde d)$ the solutions defined in \eqref{psi.def} associated to $\mu$ and $\tilde \mu$, we have by \eqref{newpsi}:
$$\hat \psi:=\psi-\tilde \psi=2\pi \sum_{l=1}^{\hat n} \bigg[\big(G(\cdot, p_{l}) - G(\cdot, q_{l})\big)\vol_g\bigg].$$ 
As $\big|d^*\big[\big(G(x, p) - G(x, q)\big)\vol_g\big]\big|\leq C_r \dist_S(p,q)$ for $\dist_S(x, \{p,q\}) \ge r$, we deduce by \eqref{zetak.def} that
$$|\zeta_k(a;d)-\zeta_k(\tilde a, \tilde d)|\leq \int_{\gamma_k} \big|d^*\hat \psi|
\leq 2\pi \sum_{l=1}^{\hat n} \int_{\gamma_k} \big|d^*\big[\big(G(x, p_l) - G(x, q_l)\big)\vol_g\big]\big|\leq C_r \|\mu-\tilde \mu\|_{W^{-1,1}}.$$
The conclusion is now straightforward.
\end{proof}

\begin{example}
\label{exam}
Let $S$ be the flat torus  $\R^2/\Z^2$ with the standard $(x,y)$ coordinates
and the standard metric 
$ds^2 = dx^2+dy^2$. We will often identify $S$ with the unit square
with periodic boundary conditions.
Here the genus $\mathfrak g = 1$, and
the $1$-forms $\eta_k$ (fixed as an orthonormal basis in \eqref{harmon})
may be taken to be
\[
\eta_1 = dx, \qquad \eta_2 = dy.
\]
In addition, we may take the geodesics from Section \ref{S:homology} 
to be
\[
\gamma_1(s) = (s,0) ,\qquad \qquad \gamma_2(s) = (0,s),
\qquad \mbox{ for $0\le s \le 1$. }
\]
We let $\{\tau_1,\tau_2\}$ denote the standard coordinate vector fields,
yielding a global moving frame for which the connection $1$-form $A$ is identically $0$ and $\kappa=0$.

Fix some $(a;d)\in S^n\times \Z^n$ such that \eqref{necessary} holds
and let $\psi$ solve \eqref{psi.def}, i.e., $-\Delta \psi=2\pi \sum_{k=1}^n d_k\delta_{a_k}$ in $S$.
We will identify each $a_k$ with the point  $a_k := (a_k^x,a_k^y) \in [0,1)^2$
and we write
\[
I(y) := \int_{\lambda^y_1} d^*\psi,\qquad\mbox{ for }\lambda_1^y(s) = (s,y), \ \ 0\le s \le 1.
\]
For every $y\in [0,1)$, note that $\lambda_1^y$ is homologous to the geodesic $\gamma_1$.
According to the definition \eqref{zetak.def},
if  $y \not\in \{ a_k^y\}_{k=1}^n$, then
$\zeta_1(a;d)$ is the equivalence class in $\R/2\pi\Z$
containing $I(y)$.
We may assume by a translation that $0\not\in \{ a_k^y\}_{k=1}^n$. Then
by Stokes Theorem and the equation \eqref{psi.def} for $\psi$:
\begin{equation}\label{zeta.example1}
I(y) - I(0) =
\int_{ [0,1]\times (0, y)} d d^*\psi = 
2\pi
\int_{ [0,1]\times (0, y)} \sum_{k=1}^n d_k\delta_{a_k}
 \end{equation}
for {\em a.e.} $y\in (0,1)$.
On the other hand,  the $2$-form $\psi$ may be written $\psi = \tilde \psi(x,y) dx \wedge dy$
for some function $\tilde \psi:S\to \R$, which we may identify with a $\Z^2$ periodic
function on $\R^2$. Then $d^*\psi =-\partial_x \tilde\psi dy +\partial_y\tilde\psi dx$, so
that 
\[
\int_0^1 I(y) \, dy = \int_0^1 \left( \int_0^1 \partial_y\tilde \psi(s,y) ds\right) dy = 0
\]
by the periodicity of $\tilde \psi$. 
We can thus integrate \eqref{zeta.example1} and simplify (using the fact that $\sum d_i=0$) to find that 
\[
I(0) = \int_0^1 I(0) \,dy  
 = \ -2\pi \sum_{k=1}^n d_k (1-a_k^y)
 = \ 2\pi \sum_{k=1}^n d_k a_k^y.
\]
This determines $\zeta_1(a;d)$.
An identical computation shows that $\zeta_2(a;d)$ is the equivalence
class in $\R/2\pi\Z$ containing $2\pi \sum_{k=1}^n d_ka_k^x$.
\end{example}

\bigskip

\section{The intrinsic renormalized energy. Proof of Proposition \ref{prop.W}}
\label{Sec:RE}
In this section, we prove the characterization of the 
intrinsic renormalized energy in Proposition \ref{prop.W}. 

\begin{proof}[Proof of Proposition \ref{prop.W}] Let $r>0$ be small satisfying 
\[
\sqrt r \le \rho_a:=\min_{k\ne l}\dist_S(a_k, a_l)
\]
and recall the notation $S_r := S\setminus \cup_{k=1}^n B_r(a_k)$.
The fact that $u^*$ is a unit vector field implies that 
$|Du^*|_g^2 = |j(u^*)|_g^2$. 
Then the form \eqref{form_jstar} of $j(u^*)$
implies that 
$$
\frac 12\int_{S_r} |j(u^*)|_g^2 \vol_g  
=
\frac 12
\int_{S_r} \bigg(|d^*\psi|_g^2 + 2 \sum_{k=1}^{2\mathfrak g} \Phi_k(d^*\psi, \eta_k)_g + \sum_{l,k=1}^{2\mathfrak g} \Phi_l\Phi_k 
(\eta_l,\eta_k)_g\bigg) \ \vol_g.
$$

\nd {\it Step 1. Computing the integrals depending on $\Phi$}. As $\{\eta_k\}_{k=1}^{2\mathfrak g}$ are smooth forming an orthonormal basis of \eqref{harmon}, we compute
$$
\int_{S_r}  \sum_{l,k}\Phi_l\Phi_k (\eta_l,\eta_k)_g \ \vol_g
=
\int_{S}\sum_{l,k} \Phi_l\Phi_k (\eta_l,\eta_k)_g \ \vol_g  
+O(|\Phi|^2 r^2)=
|\Phi|^2 +O(|\Phi|^2 r^2).
$$
Similarly, integrating by parts, 
\begin{align*}
\int_{S_r} \sum_l \Phi_l (d^*\psi,\eta_l)_g \ \vol_g
&=
\int_{S}  \sum_l\Phi_l (\psi, \underbrace{d\eta_l}_{=0})_g \ \vol_g + O( |\Phi| \, r) \\
&  = O(|\Phi|^2  r^{3/2} + r^{1/2})
\end{align*}
where we used 
\eqref{newpsi} and the properties on Green's function, which imply
that
$$\int_{B_r} |d^* \psi|\, dvol\sim \int_0^r \frac 1s\, sds=O(r).$$
 \\

\nd {\it Step 2. Computing $\int_{S_r} |d^*\psi|_g^2\, \vol_g$.}
We rewrite \eqref{newpsi} as follows:
\[
\psi = (\psi_0 + \psi_1)\vol_g, \qquad \psi_1 := \sum_{k=1}^n 2\pi d_k G(\cdot, a_k).
\]
(Observe that we have taken $\psi_0, \psi_1$ to be functions, whereas $\psi$ is a $2$-form.)
Then $|d^*\psi|_g^2 = |\star d \star \psi|_g^2 = |d\star\psi|_g^2 = |d(\psi_0+\psi_1)|_g^2$.
Since $\psi_0$ is smooth and $\psi_1\in W^{1,p}$ for $p<2$, it follows that
$$
\int_{S_r} |d^*\psi|_g^2 \vol_g 
= 
\int_{S_r} |d\psi_1|_g^2 \vol_g 
+ \int_{S_r} \bigg(2(d \psi_1, d \psi_0 )_g +|d\psi_0|_g^2\bigg)\, \vol_g. 
$$

\medskip

\nd {\it Step 2a. Computing $\int_{S_r} |d\psi_1|_g^2\, \vol_g$.}
We use  Stokes Theorem (see \eqref{stokes.bdy})
to write
\[
\int_{S_r} |d\psi_1|_g^2 \vol_g
=
\int_{S_r} (d^*d\psi_1, \psi_1)_g \vol_g - \sum_{k=1}^n \int_{\partial B(a_k,r)} 
\psi_1 \star d\psi_1 \ .
\]
Since $\psi_1$ has mean $0$ and  $d^*d\psi_1$ is constant, equal with $-\bar \kappa$ (see \eqref{psi0.def}) away from $\{a_k\}$, it follows \footnote{Recall that $\Delta (\psi_1 \vol_g)=
(\Delta \psi_1) \vol_g$.}
\[
 \int_{S_r}  (d^* d\psi_1, \psi_1)_g \vol_g =
-\bar\kappa  \int_{S_r}\psi_1\vol_g 
= \bar\kappa \int_{\cup_k B_r(a_k)} \psi_1\, \vol_g
= O(r^2 (|\log r|+1)) 
\] where we used the Green functions properties recalled in Section \ref{sec:renorm} and the fact that the distance between the points $a_k$ is larger than $\sqrt r$, i.e.,
\begin{align*}
&\int_{B_r(a_k)} G(\cdot, a_k)\, \vol_g\leq C \int_0^r |\log s|\, sds=O(r^2 (|\log r|+1)),\\ 
&\int_{B_r(a_k)} G(\cdot, a_l)\, \vol_g\leq C |\log \dist_S(a_k, a_l)| \textrm{Vol}(B_r(a_k)),
\end{align*}
We now fix $k\in \{1\ldots, n\}$, and we 
write
\[
R_k(x) := \psi_1(x)  + d_k \log \dist_S(x, a_k) = 2\pi d_k H(x,a_k) +\sum_{l\ne k} 2\pi d_l G(x, a_l)
\]
to denote the regular part of 
$\psi_1$ near $a_k$. Since $H \in C^1(S\times S)$ and $\dist_S(a_l, a_k) \ge \sqrt r$ for every $l\ne k$, it is clear that $R_k$ is Lipschitz in $B_r(a_k)$, with Lipschitz constant
bounded by $Cr^{-1/2}$. In addition,
$|d\psi_1|_g\le C/r$ on $\partial B(a_k,r)$, so
\begin{align*}
\int_{\partial B(a_k,r)}\, \psi_1 (\star d\psi_1)
&= 
\int_{\partial B(a_k,r)} (R_k -  d_k \log r )(\star d \psi_1) \\
&= 
\big(R_k (a_k) -  d_k \log r + O(\sqrt r)\big)\int_{\partial B(a_k,r)} \star d \psi_1
\end{align*}
and (recalling that $\eta = \star \eta \,\vol_g$  for any $2$-form $\eta$)
\begin{align*}
\int_{\partial B(a_k,r)}\star d \psi_1 &=
\int_{B(a_k,r)} d\star d \psi_1 =
\int_{B(a_k,r)} \underbrace{\star d\star}_{=-d^*} d \psi_1 \vol_g  =
\int_{B(a_k,r)} \Delta  \psi_1 \vol_g \\
&= -2\pi d_k + \bar\kappa \mbox{Vol}(B(a_k,r)) = -2\pi d_k - O(r^2).
\end{align*}
Combining the above, we find that 
$$
 \int_{S_r} |d\psi_1|^2 \vol_g 
= -  \sum 2\pi d_k^2  \log r
+ \sum_k 4\pi^2 d_k^2 H(a_k,a_k) + 8\pi^2\sum_{1\leq l< k\leq n}d_kd_lG(a_k,a_l) +O(\sqrt r).
$$

\nd {\it Step 2b. Computing $\int_{S_r}(d\psi_1 ,d\psi_0)_g\, \vol_g$.}
Since  $\psi_0$ is smooth in $S$ and $\psi_1\in W^{1,p}(S)$ for $p<2$, H\"older's inequality leads to
\[
\int_{S_r} (d\psi_1 ,d\psi_0)_g \, \vol_g = 
\int_{S} (d\psi_1 ,d\psi_0)_g\, \vol_g +   \| d\psi_1\|_{L^{4/3}} O( r^{1/2})  = \int_{S} (d\psi_1 ,d\psi_0)_g\vol_g + O(\sqrt r).
\]
As $\psi_0$ has mean $0$, Stokes theorem and the equation satisfied by $\psi_1$ imply that
\[
\int_S (d\psi_0 ,d\psi_1)_g \vol_g = 
\int_S (\psi_0 , d^*d \psi_1)_g \vol_g =
\int_S (\psi_0 , -\Delta \psi_1)_g \vol_g =
 \sum_k 2\pi d_k \psi_0(a_k) .
\]

\nd {\it Step 3. Conclusion}. 
As a consequence of the above computation, we obtain the following: there exists $r_0(S)>0$ such that if 
$r\in (0,r_0)$ satisfies 
\[
\sqrt r \le \min_{k\ne l}\dist_S(a_k, a_l)
\]
then any $1$-form $j^*=j^*(a,d,\Phi)$ satisfying \eqref{form_jstar} (with $\psi=\psi(a;d)$ given by \eqref{psi.def} and $\{\Phi\}_{k=1}^{2\mathfrak g}$ not necessarily in $\calL(a;d)$)
we have that
\begin{align}
\label{W.def}
&\frac 12 \int_{S_r} |j^*(a,d, \Phi)|^2 \vol_g 
=- \pi \log r\sum_{k=1}^n d_k^2+4\pi^2 \sum_{1\leq l< k\leq n} d_l d_k G(a_l,a_k)   \\
&+2\pi \sum_{k=1}^n  \left[ \pi d_k^2 H(a_k,a_k) +d_k \psi_0(a_k) \right] 
+ \frac 12|\Phi|^2 + \int_S \frac{|d\psi_0|_g^2}2 \vol_g
 + O( \sqrt r)  + O(|\Phi|^2 r^{3/2}).
\nonumber \end{align}
Moreover, the constants above depend only on $S$ and $\sum_{k=1}^n  |d_k|$.
We conclude that the limit in the definition \eqref{defi_W} of $W(a,d, \Phi)$ exists and the desired formula \eqref{W.formula} holds true.

\end{proof}

\bigskip

\section{Compactness}
\label{sec:Comp}

The result of this section will be crucial in proving point 1 of our main result in 
Theorem~\ref{intrinsic.gammalim}. It is stated as precise estimates for the vorticity and the flux integrals in terms of  the intrinsic energy,
but  immediately
implies parallel results for the other energies (in view of \eqref{e_vs_e1} and
\eqref{e_vs_e2}, see below).

\begin{propo}
For every $p\in [1,2)$ and $T, C>0$, every integer $n>T-1$ and every $0<q< 1 - \frac T{n+1}$, then
there exist $\e_0 \in (0, \frac12)$, $C_p>0$ such that the following holds true: 
if $0<\e<\e_0$ and $u\in \calX^{1,2}(S)$ with  
\begin{equation}
\frac 12 \int_S |Du|_g^2 + \frac 1{2\e^2}F(|u|^2_g) \, \vol_g\le T \pi \logeps + C,
\label{Escaling}\end{equation}
then there exist $K$ distinct points $a_1,\ldots, a_K\in S$
and nonzero integers $d_1,\ldots, d_K\in \Z$ such that \eqref{necessary} holds, $\sum_{k=1}^K |d_k|\leq n$ (so, $K\leq n$) and 
\begin{equation}
\| \omega(u) - 2\pi \sum_{k=1}^K d_k \delta_{a_k} \|_{W^{-1,p}} \le 
C_p (n+1) T \logeps \e^{q(\frac 2p-1)}.
\label{Pcomp.c1}\end{equation}
Moreover, if we define
\begin{equation}\label{Phiu.def}
\Phi(u) = (\Phi_1(u),\ldots, \Phi_{2\mathfrak g}(u)) := \left(\int_S  (j(u) ,  \eta_1)_g \vol_g ,\ldots, \int_S  (j(u), \eta_{2\mathfrak g})_g \vol_g\right), 
\end{equation}
for the orthonormal basis $\{\eta_k\}_{k=1}^{2\mathfrak g}$ fixed in \eqref{harmon}, then
\begin{equation}
\dist_{\R^{2\mathfrak g}} (\Phi(u), \mathcal L (a;d)) \le C_q \e^q \ ,
\label{fluxes.converge}\end{equation}
where $\mathcal L (a;d)$ is the set defined in Section \ref{sec:renorm} for $a=(a_1,\ldots, a_K)$ and $d=(d_1,\ldots, d_K)$.
\label{P.comp}\end{propo}

In the above Proposition, $n$ can be $0$ (if $T\in (0,1)$), in which case, $K=0$.
Our proof will rely on the following result:

\begin{propo}\label{P.vballs}
For every $T, C>0$, every integer $n>T-1$ and every $0<q< 1 - \frac T{n+1}$, then
there exist  $\e_0, r_0, c>0$ such that the following holds true: 
if $\e\in (0,\e_0)$, $\sigma\in [\e^q, r_0]$ and $u\in \calX(S)$ be a smooth vector field with  \eqref{Escaling},
then there exists a collection of pairwise disjoint balls $\mathcal B^\sigma = \{ B_{l,\sigma}\}_{l=1}^{K_\sigma}$ of centers $a_{l, \sigma}\in S$ and radius $r_{l,\sigma}>0$
such that
\begin{align}\label{balls1}
\{ x\in S\, :\, |u(x)|_g &\le \frac 12 \} \subset \cup_{l=1}^{K_\sigma} B_{l,\sigma},
\\
\label{balls2}
\sum_{l=1}^{K_\sigma} | d_{l,\sigma} | \ &\le n \, ,
\qquad\qquad
\mbox{ where }d_{l,\sigma} := \deg(u; \partial B_{l, \sigma})\ .
\\ 
\label{balls3}
\sum_{l=1}^{K_\sigma} r_{l,\sigma}  \ &\le  (n+1)\sigma, \\
\label{balls4}
\int_{B_{l,\sigma}}e^{in}_\e(u) \vol_g &\ge |d_{l,\sigma}| (\pi  \log \frac \sigma \e - c),\quad l=1, \dots, {K_\sigma}.
\end{align}
\end{propo}

If $n=0$ above, then $K_\sigma$ is not necessarily $0$ (as balls of degree zero may appear). Proposition \ref{P.vballs} is proved by a rather standard vortex balls argument,
as introduced in \cite{Je, Sa} for the Ginzburg-Landau energy
in flat $2$-dimensional domains. 
We present some details in Appendix \ref{App.A}.
With Proposition \ref{P.vballs} available,
the proof of the basic compactness assertion \eqref{Pcomp.c1} follows
classical arguments, which we recall for the convenience of the
reader. The main new point is the estimate 
\eqref{fluxes.converge} of the flux integrals.

\begin{proof}[Proof of Proposition \ref{P.comp}] In what follows, $c>0$ 
is a constant that can change from line to line and that can depend on all parameters appearing the hypotheses of the proposition.
\\

\nd {\it Step 1. Reduction to smooth bounded vector fields}. We consider $h:\R^+\to \R^+$ to be the Lipschitz cut-off function $h(s)=1$ if $s\leq 1$ and $h(s)=1/s$ if $s>1$ and 
$$\hat u:=h(|u|_g) u.$$ First, we want to show that we can replace $u$ by $\hat u$ in the statement of Proposition \ref{P.comp}. Indeed, $\hat u\in \calX^{1,2}(S)$ and since $|D\hat u|_g\leq |Du|_g$ (see Section \ref{sec:calc}) and $F(|\hat u|_g)\leq F(|u|_g)$ (because $F(1)=0$), we get that 
$$\int_{\calO} e^{in}_\e(\hat u) \vol_g \leq \int_{\calO} e^{in}_\e(u) \vol_g, \textrm{ for every $\calO\subset S$},$$ so the bound \eqref{Escaling} is conserved for $\hat u$. Moreover, by Section \ref{sec:calc},
\[
|j(\hat u)-j(u)|_g =  |h^2(|u|_g)-1| \, |j(u)|_g \ \le |u|_g \, |h^2(|u|_g)-1| \, |Du|_g .
\]
Moreover, the definition of $h$ and \eqref{F.growth} imply that 
\[
|u|_g\, |h^2(|u|_g)-1| \ \le 2 \left| 1 - |u|_g \right| \le c \sqrt { F(|u|_g^2)}.
\]
It follows that 
$$
\| j(\hat  u) - j(u)\|_{L^1(S)} \le c \e E^{in}_\e(u) \ .
$$
In particular, by \eqref{Phiu.def}, we have for any $k \in \{1,\ldots, 2\mathfrak g\}$ that 
$$\Phi_k(u) =\int_S ( j(\hat u), \eta_k)_g \vol_g \ + O (\e E^{in}_\e(u)) \ =  \
\Phi_k(\hat u) + O(\e \logeps).
$$
Moreover, for any $\varphi \in W^{1, \infty}(S)$, we have
$$
\bigg|\int_S \varphi  \left[ \omega(u) -  \omega(\hat u) \right ]\bigg|  = 
\bigg|\int_S \varphi  \, d \left[j(u) -  j(\hat u) \right ] \bigg|  \le c \| d \varphi \|_{L^\infty} \e E^{in}_\e(u);
$$
this yields $ \| \omega(\hat u) - \omega(u)\|_{W^{-1,1}(S)} \le c \e \logeps$. To estimate $ \| \omega(\hat u) - \omega(u)\|_{W^{-1,p}(S)}$
for $1<p<2$, we use Lemma \ref{L.last}:
\[
\| \omega(u) - \omega(\hat u)\|_{L^1} \le \int_S |d j(u)|_g+|d j(\hat u)|_g \vol_g \leq \int_S |Du|^2_g+|D\hat u|^2_g \vol_g \le c \logeps
\]
and
then the interpolation inequality: 
$$\| \omega(\hat u) - \omega(u)\|_{W^{-1,p}} \le  C\|\omega(\hat u) - \omega(u)\|_{W^{-1,1}}^{ \frac 2p-1} \|\omega(\hat u) - \omega(u)\|_{L^1}^{2-\frac 2p}
=O(\eps^{\frac 2p-1} \logeps).$$
 Therefore, it is enough to prove the statement for $\hat u$ instead of $u$. 
Furthermore, due to the density result in Lemma \ref{L.density} and the continuity results in Theorem \ref{P1} point 1) and Lemma \ref{L.last}, {\bf we can assume that $u$ is a smooth vector field in $S$ with $|u|_g\leq 1$.} (The cutting-off procedure
$|\hat u|_g\leq 1$ is needed in order that the potential term in the energy $E^{in}_\eps$ passes to the limit, as $F$ could increase very fast at infinity.) \\

\nd {\it Step 2. An approximation $\tilde u$ of $u$}. 
Let $\tilde h:\R^+\to \R^+$ be a smooth function such that
\[
\tilde h(s) = 1 \mbox{ for }0\le s \le \frac 14, 
\qquad
\tilde h(s) = 1/s \mbox{ for }s \ge \frac 12, 
\qquad
s\mapsto s \tilde h(s)\mbox{ is nondecreasing}
\]
and define the smooth vector field
\begin{equation}\label{tildeu.def}
\tilde u = \tilde h(|u|_g) u.
\end{equation}
The advantage of working with $\tilde u$ is that $|\tilde u|_g=1$ if $|u|_g\geq \frac12$. 
Then Section \ref{sec:calc} implies $j(\tilde u) = \tilde h^2(|u|_g) j(u)$ and $|D\tilde u|_g\leq c |Du|_g$ in $S$ since by \eqref{formula12}, we have
$$|D\tilde u|^2_g\leq \big[\tilde h^2(|u|_g)+\big(\frac{d}{ds}(s\tilde h(s))\big)^2\big|_{s=|u|_g}\big] |Du|^2_g.$$
By the computations in Step 1, we deduce
$$
\| j(\tilde u) - j(u)\|_{L^1(S)} \le c \e E^{in}_\e(u) \ .
$$

\nd {\it Step 3. Proof of \eqref{Pcomp.c1}}. 
For any $\varphi \in W^{1,\infty}(S)$, it follows from Step 2 that
$$
\bigg|\int_S \varphi  \left[ \omega(u) -  \omega(\tilde u) \right ]\bigg|  = 
\bigg|\int_S \varphi  \, d \left[j(u) -  j(\tilde u) \right ] \bigg|  \le c \|d \varphi \|_{L^\infty} \e E^{in}_\e(u).
$$
Wit the notations of Proposition \ref{P.vballs} applied for the smooth vector field $u$, we claim that for $\varphi$ as above and $\e^q \le \sigma \le r_0$,
$$
\bigg|\int_S \varphi  \big[  \omega(\tilde u) - 2\pi\sum_{l=1}^{K_\sigma} d_{l, \sigma} \delta_{a_{l, \sigma}} \big]\bigg|  \le c
\|d\varphi\|_{L^{\infty}}
\bigg(\sum_{l=1}^{K_\sigma} r_{l,\sigma} \bigg)E^{in}_\e(u).
$$
Indeed, it follows from Proposition \ref{P.vballs} (see \eqref{balls1}) and \eqref{tildeu.def} that $|\tilde u|_g=1$ outside the balls 
$\cup_{l=1}^{K_\sigma} B_{l,\sigma}$ so that Lemma \ref{L.omega0} implies $\omega(\tilde u) = 0$
outside $\cup_{l=1}^{K_\sigma} B_{l,\sigma}$. So
\[ 
\int_S \varphi \omega(\tilde u) 
= 
\sum_{l=1}^{K_\sigma} \int_{B_{l,\sigma}} \varphi\,\omega(\tilde u).
\] 
For each $1\leq l\leq K_\sigma$, we have that 
\begin{align*}
\int_{B_{l,\sigma}}\varphi\,\omega(\tilde u) & =  \varphi(a_{l, \sigma})\int_{B_{l,\sigma}}\omega(\tilde u)
 + \int_{B_{l,\sigma}}\big(\varphi(x) - \varphi(a_{l,\sigma})\big)\omega(\tilde u)\\
 &=\varphi(a_{l,\sigma})\underbrace{\bigg(\int_{\partial B_{l,\sigma}} j(\tilde u) +\int_{B_{l,\sigma}} \kappa \vol_g\bigg)}_{= 2\pi d_{l,\sigma}}+
 \int_{B_{l,\sigma}}(\varphi(x) - \varphi(a_{l,\sigma}))\omega(\tilde u),
\end{align*}
where we used \eqref{deg.def}, \eqref{omega.def} and the fact that $|\tilde u|_g=1$
on $\partial B_{l,\sigma}$ by \eqref{balls1}. In particular, for $\varphi=1$ in $S$, one has that 
$$2\pi \sum_{l=1}^{K_\sigma} d_{l,\sigma}
=\int_S \omega(\tilde u)=\int_S \kappa \vol_g=2\pi \chi(S),$$ i.e., \eqref{necessary} holds for the integers $\{d_{l,\sigma}\}_l$.
To estimate the last term in the above RHS, note that clearly
\[
|\varphi(x) - \varphi(a_{l, \sigma})| \le \|d \varphi\|_{L^\infty}r_{l,\sigma} \ \ \ \mbox{ for }x\in B_{l,\sigma}.
\]
Moreover, Lemma \ref{L.last} and the definition of $\omega$ imply that
$|\omega(\tilde u)|_g \le |D\tilde u|_g^2 + |\kappa|$ in $S$, and as a consequence, 
\[
\int_{B_{l,\sigma}} |\omega(\tilde u)|_g \vol_g \le  \int_S |D\tilde u|_g^2 \vol_g + c
\le c\int_S |Du|_g^2 \vol_g + c \le c(E^{in}_\eps(u) + 1).
\]
We may assume that $\e_0<\frac 12$, and
then we can absorb the additive constant in the multiplicative constant.
By combining these estimates with \eqref{balls3}, we see that for any smooth $\varphi$,
\[
\left|\int_S \varphi \bigg[ \omega(u) - 2\pi \sum_{l=1}^{K_\sigma} d_{l,\sigma}\delta_{a_{l,\sigma}}\bigg]  \right| \le
c (n+1) T \logeps \sigma \|d\varphi \|_{L^\infty}.
\]
Setting $\sigma=\eps^q$, this is the case $p=1$ of \eqref{Pcomp.c1}, noting that all the points $\{a_{l,\sigma}\}_{l=1}^{K_\sigma}$ are disjoint (as they belong to pairwise disjoint balls), \eqref{necessary} holds and $\sum_{l=1}^{K_\sigma} |d_{l, \sigma}|\leq n$ by \eqref{balls2}. 
For $1<p<2$, we complete the proof of \eqref{Pcomp.c1} using (again) the interpolation inequality 
$\| \mu\|_{W^{-1,p}} \le  C\| \mu\|_{W^{-1,1}}^{ \frac 2p-1} \|\mu\|_{L^1}^{2-\frac 2p}$, where  $L^1$ norm is understood to mean the total variation if $\mu$ is a measure,
together with the fact that
\[
\| \omega(u) - 2\pi \sum_{l=1}^{K_\sigma} d_{l,\sigma}\delta_{a_{l,\sigma}}\|_{L^1} \le \int_S (|Du|_g^2 +|\kappa|)\vol_g + 2\pi n \le c n T\pi \logeps,
\]
provided that $\e< 1/2$. This  follows easily from \eqref{omega.def}, \eqref{balls2} and Lemma \ref{L.last}. Also, \eqref{Pcomp.c1} holds for $\omega(\tilde u)$ (as the interpolation argument works for $\omega(\tilde u)$ exactly as for $\omega(u)$). Discarding the points $a_{l, \sigma}$ with zero degree $d_{l, \sigma}=0$, one may assume that
in \eqref{Pcomp.c1} all the integers $d_k$ are nonzero.\\

\nd {\it Step 4. Proof of \eqref{fluxes.converge}}.
For any $k \in \{1,\ldots, 2\mathfrak g\}$, it follows from Step 2 that 
\begin{equation}\label{Phi.dist1}
\Phi_k(u) = 
 \int_S ( j(\tilde u), \eta_k)_g \vol_g \ + O (\e E^{in}_\e(u)) \ =  \
\Phi_k(\tilde u) + O(\e E^{in}_\e(u)).
\end{equation}
Since $h^2(|u|_g)|u|_g\leq c$, we have $|j(\tilde u)|_g\le c|Du|_g\in L^2$.
As $\tilde u$ is smooth in $S$, the Hodge decomposition \eqref{Hodge} implies
\begin{equation}\label{juHodge}
j(\tilde u) = d\xi + d^*\tilde \psi +\sum_{k=1}^{2\mathfrak g} \tilde \Phi_k \eta_k, \qquad\qquad
\tilde \Phi_k := \Phi_k(\tilde u)
\end{equation}
for some smooth function $\xi$ and $2$-form $\tilde \psi$. Taking the exterior derivative of this, we find that
\[
-\Delta \tilde \psi + \kappa \vol_g = \omega(\tilde u) \quad \textrm{in}\quad S,
\]
and the RHS is supported in $\cup_{l=1}^{K_\sigma} B_{l, \sigma}$ (see Step 2).
As in Step 4 of the proof of Theorem \ref{P1}, for some   $r>0$ (small but fixed, independent of $\e$) and  every small enough $\e>0$ (and 
hence also $\sigma=\e^q$),  we fix
Lipschitz paths
$\lambda_k$, for $k=1,\ldots, 2\mathfrak g$  such that
\[
\lambda_k\cap \left( \cup_{l=1}^{K_\sigma} B_{l,\sigma}
\right)= \emptyset, \qquad
\dist_S(\lambda_k ,\cup_{l=1}^{K_\sigma} B_{l, \sigma} ) \ge r \mbox{ if }d_{l,\sigma}\ne 0,
\]
and
\[
\lambda_k \mbox{ is homologous to }\gamma_k,
\qquad\mbox{ and }\quad \calH^1(\lambda_k) \le c 
\]
for all $k\in \{1,\ldots, 2\mathfrak g\}$ (recall that $\{\gamma_k\}$ are the curves fixed in Lemma \ref{L.FedFlem}). The point is that, as in  Theorem \ref{P1}, we obtain $\lambda_k$
by  starting with
 $\gamma_k$ and modifying it as necessary, first to make it disjoint from
all $B_{l,\sigma}$, increasing the arclength by at most $2\pi(n+1)\sigma$ due to \eqref{balls3}; and next to arrange that it is always a distance at least $r$ from every ball with nonzero degree $d_{l,\sigma}$.  Since the number of such balls is at most $n$, due to \eqref{balls2} this can be done in
such a way that the arclength increases by a controlled amount, for example $2\pi nr$.
If $r$ and $\sigma$ are small enough, these modifications preserve the homology class.
We next define
\[
\tilde \zeta_k := \int_{\lambda_k} (d^*\tilde \psi + A) \ \in \R/2\pi\Z
\]
where $A$ is the connection $1$-form associated to any moving frame defined in
a neighborhood of $\lambda_k$.
It follows from \eqref{juHodge} 
and Lemma \ref{L.ns} (as $|\tilde u|_g=1$ outside $\cup_{l=1}^{K_\sigma} B_{l, \sigma}$) that for $k=1,\ldots, 2\mathfrak g$
\[
\tilde \zeta_k + \sum_{\ell=1}^{2\mathfrak g} \alpha_{k \ell}\tilde \Phi_\ell 
\ = \ \int_{\lambda_k} ( d^*\psi+A + \sum_{\ell=1}^{2\mathfrak g} \tilde \Phi_\ell   \eta_\ell ) \ = 
\int_{\lambda_k}j(\tilde u)+A -\underbrace{\int_{\lambda_k} d\xi}_{=0}= \ 0 \mod 2\pi,
\]
where $(\alpha_{k\ell})$ were defined in \eqref{akl.def}. Let us write $(\alpha^{k\ell})$ to denote the inverse of $(\alpha_{k\ell})$.
Denoting $d^\sigma=(d_{1, \sigma}, \dots, d_{K_\sigma, \sigma})$, as $d^\sigma=\{d_{l, \sigma}\}_{l=1}^{K_\sigma}$ satisfy \eqref{necessary}, we may consider the unique solution $\psi=\psi(a^\sigma; d^\sigma)$ of \eqref{psi.def} of zero mean on $S$, i.e., 
\[
-\Delta  \psi = - \kappa \vol_g + 2\pi \sum_{l=1}^{K_\sigma} d_{l, \sigma}\delta_{a_{l, \sigma}} \quad \textrm{in}\quad S.
\]
Considering $\zeta_\ell(a;d)$ given by \eqref{zetak.def} with $a=(a_{1, \sigma}, \dots, a_{K_\sigma, \sigma})$ and
$d=(d_{1, \sigma}, \dots, d_{K_\sigma, \sigma})$, we deduce that 
\[
\zeta_\ell(a;d)  +  \sum_{k=1}^{2\mathfrak g} \alpha_{\ell k} \bigg[
\underbrace{ \tilde \Phi_k + \sum_{m=1}^{2\mathfrak g} \alpha^{k m}(\tilde \zeta_m - \zeta_m(a;d))}_{=:\hat{\Phi}_k}\bigg]
= 0 \mod 2\pi \quad\mbox{ for }\ell=1,\ldots, 2\mathfrak g , 
\]
which implies that the vector in square brackets $\{\hat{\Phi}_k\}_{k=1}^{2\mathfrak g}$ belongs to $\calL(a;d)$.
Hence, in view of \eqref{Phi.dist1}, we have that
\begin{equation}\label{Phi.dist2}
\dist_{\R^{2\mathfrak g}}(\Phi(u), \calL(a;d)) \le c \e E^{in}_\e(u) + c\sup_\ell |\tilde \zeta_\ell - \zeta_\ell(a;d)|.
\end{equation}
To estimate $\tilde \zeta_\ell - \zeta_\ell(a;d)$,  we investigate the equation
\[
-\Delta(\tilde \psi - \psi) = \omega(\tilde u) - 2\pi\sum_{l=1}^{K_\sigma} d_{l,\sigma} \delta_{a_{l, \sigma}}.
\]
Thus, for any $p\in(1,2)$, we see from Step 2 and elliptic regularity that
\[
\| \tilde \psi - \psi\|_{W^{1,p}} \le C_p(n+1)T|\log\e| \e^{q(\frac 2p -1)}.
\]
Also, $\tilde \psi -\psi$ is harmonic away from $\cup_{k=1}^{K_\sigma} B_{k, \sigma}$, so we further deduce from standard elliptic theory that for $r>0$ fixed above,
\[
\| \tilde \psi - \psi\|_{C^1(\{ x\in S \, : \, \dist_S(x, \cup B_{k, \sigma}) > r \}) } \le C_{p,r}(n+1)T|\log\e| \e^{q(\frac 2p -1)}.
\]
In particular this estimate holds on $\lambda_\ell$ for every $\ell=1, \dots, 2\mathfrak g$. Thus, as a direct
consequence of the definitions of $\tilde \zeta_\ell$ and $\zeta_\ell(a;d)$, we obtain for a fixed small $r>0$:
\[
\left|\tilde \zeta_\ell - \zeta_\ell(a;d)\right| = \left|\int_{\lambda_\ell} d^*(\tilde \psi  - \psi) \right| \ \le C_p(n+1)T|\log\e| \e^{q(\frac 2p -1)} \ .
\]
For any $\tilde q\in (0, 1-\frac{T}{n+1})$, one chooses some $q\in (\tilde q, 1-\frac{T}{n+1})$ and $p\in (1,2)$ close to $1$ so that 
$(n+1)T|\log\e| \e^{q(\frac 2p -1)}\leq \eps^{\tilde q}$ for some $\eps\leq \eps_{\tilde q}$ and the above inequality and \eqref{Phi.dist2} together yield
\eqref{fluxes.converge} for $\tilde q$.
\end{proof}

As a direct consequence, we have partially the point 1 in Theorem \ref{intrinsic.gammalim}, together with a lower bound (at the first order) of the intrinsic energy:

\begin{coro}\label{Cor.cgc1}
Let $(u_\e)_{\e\downarrow 0}$ be a sequence of vector fields in $\calX^{1,2}(S)$ satisfying \eqref{Escaling} for some fixed $T,C>0$.
Then there exists a subsequence for which the vorticities $\omega(u_\e)$ converge in $W^{-1,p}(S)$,
for all $1\le p<2$, to a limit of the form $2\pi \sum_{k=1}^K d_k\delta_{a_k}$ for $K$ distinct points $a_1, \dots a_K\in S$ and nonzero $d_1, \dots, d_K\in \Z$ with \eqref{necessary} and $\sum_{k=1}^K |d_k|\leq T$ (so, $K\leq T$). Moreover, 
\[
\liminf_{\eps\to 0} \frac 1 {\pi \logeps} E^{in}_\e(u_\e) \ge \sum_{k=1}^{K} |d_k| .
\]
\end{coro}

\begin{proof}
Fix the integer $n$ satisfying $n+1>T\geq n$ and $q\in (0, 1-\frac{T}{n+1})$. By Step 1 in the proof of Proposition \ref{P.comp}, we may assume that $u_\eps$ are smooth vector fields with $|u_\eps|_g\leq 1$ in $S$. Furthermore, for each $\eps>0$, as in the proof of Proposition \ref{P.comp}, we consider $\sigma=\eps^q$ and the set of pairwise disjoint balls $\cup_{l=1}^{K_\eps} B_{l, \eps}$ of center $\{a_{l, \eps}\}_l$ associated to $u_\eps$ such that 
$d_{l, \eps}$ is the degree of $u_\eps$ on $\partial B_{l, \eps}$ satisfying \eqref{necessary}. Moreover, $\sum_l |d_{l, \eps}|\leq n$ which entails that for a subsequence $\eps\downarrow 0$, there exist points $a_1, \dots, a_K\in S$ (not necessarily distinct) and $d_1, \dots, d_K\in \Z$ such that the measures $\mu_\eps:=2\pi\sum_{l=1}^{K_\eps} d_{l, \eps} \delta_{a_{l, \eps}}$ converge to 
$$\mu:=2\pi \sum_{l=1}^K d_{l} \delta_{a_l}$$ as measures, and thus, in $W^{-1,p}$ for any $p\in [1, 2)$ (as $W^{1,\tilde p}(S)$ embeds in the space of continuous functions, for the conjugate real $\tilde p=\frac{p}{p-1}>2$). 
Relabeling the indices, we may assume that $a=(a_1, \dots, a_K)$ are distinct and that $d_k\ne 0$ for $k=1,\dots, K$.
Obviously, \eqref{necessary} holds (as  $\mu_\eps(S)$ is preserved by the convergence, i.e, equal to $2\pi \chi(S)$), as well as the upper bound of the total variation of those measures is conserved leading to 
\beq
\label{equ123}
\sum_{l=1}^K |d_{l}|=\frac{|\mu|(S)}{2\pi}\leq \liminf_{\eps\to 0}\frac{|\mu_\eps|(S)}{2\pi}=\liminf_{\eps\to 0}\sum_{l=1}^{K_\eps} |d_{l,\eps}|\leq n\leq T.
\eeq 
By \eqref{Pcomp.c1}, we conclude that $\omega(u_\eps)\to \mu$ in any $W^{-1,p}$ for $p\in [1, 2)$ as $\eps\to 0$.  
Finally, the lower bound of the energy is obtain by \eqref{balls4} for $\sigma=\eps^q$:
\begin{align*}
\liminf_{\eps\to 0} \frac 1 {\pi \logeps} E^{in}_\e(u_\e) &\ge\liminf_{\eps\to 0} \frac 1 {\pi \logeps} \sum_{k=1}^{K_\eps} \int_{B_{l, \eps}} e^{in}_\eps(u_\eps) \vol_g\\
&\geq \liminf_{\eps\to 0} \sum_{1\leq l\leq  K_\eps, \, d_{l, \eps}\neq 0} (1-q)|d_{l, \eps}|\geq (1-q)\sum_{l=1}^K |d_{l}|=(1-q)\frac{|\mu|(S)}{2\pi},
\end{align*}
where we used \eqref{equ123}. As $\mu$ is the limit of $\omega(u_\eps)$ (so independent of $q$), passing to the limit $q\to 0$, the conclusion is straightforward.
\end{proof}

\begin{remark}
At this stage, we cannot conclude that the sequence $\{\Phi(u_\eps)\}_{\eps\downarrow 0}$ is bounded as large oscillations might arise a-priori in the current $j(u_\eps)$. To handle this difficulty, we need to insure that the excess of energy away from vortices is of order $O(1)$ (see Proposition~\ref{intrinsic.gammalim2}).
\end{remark}

\bigskip

\section{Renormalized energy as a $\Gamma$-limit in the intrinsic case. Proof of Theorem \ref{intrinsic.gammalim} }
\label{sec:intrin}

In this section, we focus on the situation where all vortices have degree $\pm 1$ and the excess of energy away from vortices is of order $O(1)$. We will prove that the flux integrals converge 
and that we have a stronger lower bound (than the one stated in Theorem \ref{intrinsic.gammalim}, point 2). This is typically the situation when the vector fields $u_\eps$ are minimizers of $E^{in}_\eps$ (or energetically close to minimizing configurations). The following Proposition together with Corollary \ref{Cor.cgc1} lead to the final conclusion of Theorem \ref{intrinsic.gammalim}. 

\begin{propo}\label{intrinsic.gammalim2}

\nd 1) Let $(u_\e)_{\e\in (0,1)}$ be a family of vector fields in $\calX^{1,2}(S)$ satisfying
\begin{equation}
E_\e^{in}(u_\e)  \le n \pi \logeps + C
\qquad\mbox{ for every $\e$}
\label{Escaling2}\end{equation}
for some integer $n>0$, and assume that
there exist  $n_0(\le n)$ distinct points $a_1,\ldots, a_{n_0}\in S$,
and nonzero integers $d_1,\ldots, d_{n_0}$ satisfying \eqref{necessary} such that 
\begin{equation}
\omega(u_\e) \overset{W^{-1,1}}\longrightarrow 2\pi \sum_{k=1}^{n_0} d_k\delta_{a_k}
\qquad\mbox{ with }\quad \sum_{k=1}^{n_0} |d_k| = n.
\label{convergence}\end{equation}
Then $n_0 = n$ and $|d_k|=1$ for every $k$,
and there exists 
$\Phi\in \calL(a;d)$ such that, after passing to a further subsequence if necessary,
\begin{equation}
\Phi(u_\e)\rightarrow \Phi ,\qquad\Phi(u_\e)\mbox{ defined in \eqref{Phiu.def}}.
\label{Phi.conv}\end{equation}
Moreover, for every $\sigma>0$, 
\begin{multline}\label{lb1}
\liminf_{\e\to 0}
\left[
E_\e^{in}(u_\e)- n (\pi \logeps  + \iota_F) 
\right]  \ \ge \  
W(a,d,\Phi) 
\\
+  \liminf_{\e\to 0}\int_{S \setminus \cup_{k=1}^n B_\sigma(a_k)} \bigg[\frac 12 \big|\frac{j(u_\e)}{|u_\e|_g}-j(u^*)\big|_g^2 + 
e_\e^{in}(|u_\e|_g)\bigg] \vol_g 
\end{multline}
for $u^* = u^*(a,d,\Phi)$, $a=(a_1, \dots, a_n)$ and $d=(d_1, \dots, d_n)$. 

\nd 2) Conversely, for every 
distinct $a_1,\ldots, a_n$ and $d_1,\ldots , d_n\in \{\pm 1\}$ satisfying \eqref{necessary}, and every $\Phi\in \calL(a;d)$
there exist sequences of smooth vector fields $u_\eps$ such that $|u_\eps|_g\leq 1$ in $S$,  \eqref{convergence} and \eqref{Phi.conv} hold and
\begin{equation}\label{intrinsic.ubd}
E_\e^{in}(u_\e)- n (\pi \logeps  + \iota_F) \to  W(a,d, \Phi) \quad \textrm{as $\e\to 0$}.
\end{equation}
\end{propo}

\bigskip

\subsection{Useful coordinates}\label{sect.coords}

It will be useful to carry out certain computations in exponential normal coordinates
near  certain points (typically, one of the points $P\in S$ about which $\omega(u_\e)$ concentrates). 
These are defined by the map $y\in \R^2\mapsto \exp_P(y_1\tau_{1,P} + y_2\tau_{2,P}) =: \Psi(y) $, where $\{\tau_{1,P},\tau_{2,P}\}$ is an orthonormal basis for $T_PS$.
This map is a diffeomorphism when restricted to a suitable neighborhood of the origin in $\R^2$.
In this neighborhood,
\[
x = \Psi(y) \  \Rightarrow\  \dist_S(P,x) = |y|, \qquad\mbox{ so }B_r(P) \cong \{ y \in \R^2 : |y|<r\}.
\]
Here $|y|$ denotes the Euclidean norm of $y\in \R^2$.
We will write $g_{lk}(y) := (\partial_l \Psi(y), \partial_k\Psi(y))_g$ to denote the components of the metric 
tensor in this coordinate system (where we identify $\partial_l\Psi(y)$ with an element of $T_{\Psi(y)}S$ in the natural way).
It is then a standard fact that 
\begin{equation}\label{normalcs}
g_{lk}(y) := \delta_{lk} + O(|y|^2), \qquad\ \ \mbox{ and hence } \ \ g(y) := \det( g_{lk}(y)) = 1 + O(|y|^2).
\end{equation}
Furthermore, we can also find a moving frame $\{\tau_1,\tau_2\}$ near $P$ such that
the connection $1$-form $A$ satisfies
\begin{equation}\label{A0}
|A(\Psi(y))|_g= O(|y|).
\end{equation}
Indeed, \eqref{A0} can be achieved by starting with an arbitrary
moving frame $\{\tau_1,\tau_2\}$  near $P$, and replacing it by
$\{ e^{i\phi}\tau_1,e^{i\phi}\tau_2\}$ for a suitable function $\phi$.

For any point $P\in S$ there is a $\sigma>0$ such that both  normal
coordinates and the above moving frame are defined in $B_\sigma(P)\subset S$.
Thus, given a vector field $u$, in this neighborhood we can
define $v = v_1+iv_2:B_\sigma(0)\subset \R^2\to \C$ by requiring that
\begin{equation}\label{uv}
u(\Psi(y)) = v(y) \tilde \tau_1(y) = v_1(y)\tilde \tau_1(y) + v_2(y)\tilde \tau_2(y), \qquad\quad
\tilde \tau_k(y) = \tau_k(\Psi(y)), \, k=1,2.
 \end{equation}
We will write $|v|=(v,v)^{1/2} := (v_1^2+v_2^2)^{1/2}$, so that $|v(y)| = |u(\Psi(y))|_g$.
We will also write the energy density $e_\e(v)$ and the current $j(v)$ to denote the Euclidean quantities
\[
e_\e(v) := \frac 1 2 |\nabla v|^2+ \frac 1 {4\e^2} F(|v|^2),\qquad
j(v) := \sum_{k=1}^2 (iv, \partial_{y_k}v) dy_k,
\]
where here all norms and inner products $(\cdot, \cdot)$ are understood with respect to the
Euclidean structure on $\R^2$, with respect to which the $1$-forms $\{dy_1, dy_2\}$  are orthonormal.
It is then routine to check that
\begin{align}
e_\e^{in}(u)(\Psi(y)) &= [1+ O(|y|^2)] e_\e(v)(y) + O(|v|^2(y))\nonumber \\
\Psi^* j(u) &= j(v) + O(|y|)|v|^2 \label{uvsv}\\
\Psi^* \vol_g &= [1+ O(|y|^2)] dy\nonumber 
\end{align}
where $dy = dy_1\wedge dy_2$ denotes the Euclidean area element.
Thus for example
\begin{align}
\int_{B_\sigma(P)} e_\e^{in}(u) \, \vol_g
&= (1+ O(\sigma^2) )
\int_{\{y\in \R^2: |y|<\sigma\}}  e_\e(v)  
+O(|v|^2) \ dy  \ .
\label{reduce}
\end{align}

\subsection{Upper bound} 

Given $F$ satisfying \eqref{F.growth}, we recall the notations
\eqref{ier.def} for the intrinsic energy of the radial vortex profile $I^{in}_F(R,\eps)$ as well as the limit $\iota_F$ defined in \eqref{gammaF.def}.
The coordinate system described above 
will allow us to reduce energy estimates on small balls to classical facts about the
Ginzburg-Landau energy in the Euclidean setting. We first use this reduction to
prove the upper bound part of Proposition \ref{intrinsic.gammalim2}.

\begin{proof}[Proof of Proposition \ref{intrinsic.gammalim2}, point 2)]
Recall that we constructed a canonical harmonic unit vector field $u^*=u^*(x;a,d,\Phi)$ in Theorem \ref{P1}. We will construct an appropriate vector field $U_\e=U_\e(a,d, \Phi)$ for the upper bound in Proposition \ref{intrinsic.gammalim2}, point 2) as follows:
first, we choose 
\[
U_\e := u^* \mbox{ in }S_{\sqrt \e} := S \setminus \cup_{k=1}^n B(a_k , \sqrt \e).
\]
In order to define $U_\e$ inside the balls $B(a_k , \sqrt \e)$, we need to prove that $u^*$ has the appropriate behavior at the boundary  $\partial B(a_k , \sqrt \e)$ which is done in the next step.

\medskip

\nd {\it Step 1. Estimating $u^*$ on $\partial B(a_k , \sqrt \e)$}.
Writing $j^* := j(u^*)$, by  \eqref{form_jstar}, \eqref{newpsi}, and properties of the Green's function
$G$ (see Section \ref{sec:renorm}), we have in a neighborhood of the vortices $a_k$:
\begin{align}
j^* (x) 
&= d^*[ 2\pi d_k G(x, a_k)\vol_g  + \mbox{smooth terms}] \nonumber \\
&
= d^*[- d_k \log (\dist_S(x, a_k)) \vol_g + \mbox{$C^1$ terms}]  \nonumber\\
&\  \ = \star d[d_k \log (\dist_S(x, a_k))]  + \mbox{$C^0$ terms} . \label{jstar.log}
\end{align}
Let $v^*:B(0, \sqrt \e)\to \SSS^1$ be the representation of $u^*$ in exponential normal 
coordinates near $a_k$ given by \eqref{uv}.
Since within these coordinates 
$
\Psi^*j(u^*)=j(v^*)-\Psi^*A, 
$ near $a_k$, we deduce that
$$j(v^*)=d_k d\theta+\mbox{$C^0$ terms} \quad \textrm{in } B(0, \sqrt \e),$$
where $d\theta$ is the angular $1$-form $d\theta := \frac 1{|y|^2}( y_1dy_2 - y_2dy_1)$. In particular, we have that 
\beq
\label{def_v_bdry}
v^*=e^{i(d_k\theta+\eta)} \textrm{ on } \partial B(0, \sqrt \e),
\eeq
for a $C^1$ function $\eta:\partial B(0, \sqrt \e)\to \R$ that we write $\eta=\eta(\theta)$ with the angular derivative $|\partial_\theta \eta|\leq C\sqrt \e$. Moreover, as $|d_k|=1$, it follows that
\[
\quad  |j(v^*)|^2(y) = \frac 1{|y|^2} + O(\frac 1{|y|}) \quad \textrm{ in } B(0, \sqrt \e).
\]

\medskip

\nd {\it Step 2. Defining $U_\e$ inside the ball $B(a_k , \sqrt \e)$}. We define $V_\e:=v^*=e^{i(d_k\theta+\eta)}$ on $\partial B(0, \sqrt \e)$. Setting $\bar \eta$ to be the mean of $\eta$ over $\partial B(0, \sqrt \e)$, we define $V_\e$ inside the annulus 
$B(0, \sqrt \e)\setminus B(0, \frac{\sqrt \e}2)$ by linear interpolation in the lifting as follows: 
$$V_\e(re^{i\theta})=e^{i[d_k\theta +\bar \eta  + 2(\frac{r}{\sqrt \e}-\frac12)(\eta-\bar \eta)]} \quad \textrm{ for } r\in (\frac{\sqrt \e}2, \sqrt \e).$$
Finally, as $|d_k|=1$, we define $V_\e$ inside the ball $B(0, \frac{\sqrt \e}2)$ as being a minimizer of $I^{in}_F(\frac{\sqrt \e}2, \e)$ if $d_k=1$ (or its complex conjugate if $d_k=-1$) up to a rotation of angle $\bar \eta$. The minimizing property of $V_\e$ implies that $|V_\e|\leq 1$ everywhere (by cutting off at $1$). Through the normal coordinates \eqref{uv}, we define $U_\e$ to be the corresponding vector field to $V_\e$ inside the ball $B(0, \sqrt \e)$. Note that by construction $U_\e\in \calX^{1,2}(S)$ (in fact, it is Lipschitz since every minimizer in $I^{in}_F(\frac{\sqrt \e}2, \e)$ is Lipschitz) and $|U_\e|_g\leq 1$ in $S$.

\medskip

\nd {\it Step 3. Estimating the energy of $U_\e$ and $j(U_\e)$ inside the ball $B(a_k , \sqrt \e)$}.
First, by definition of $V_\e$ inside the ball $B(0, \frac{\sqrt \e}2)$, we obtain via \eqref{gammaF.def}
$$ \int_{B(0,\frac{\sqrt \e}2)}e_\e(V_\e) \, dy=\pi \log \frac {\sqrt \e}{2\e} + \iota_F + o(1).
$$
Second, inside the annulus 
$B(0, \sqrt \e)\setminus B(0, \frac{\sqrt \e}2)$, since $|d_k|=1$, we have
\begin{align}\nonumber
\int_{B(0, \sqrt \e)\setminus B(0, \frac{\sqrt \e}2)} \frac12|\nabla V_\e|^2\, dy&=
\int_{ \frac{\sqrt \e}2}^{\sqrt \e}\int_0^{2\pi}
\frac1{2r}\big|d_k+2(\frac{r}{\sqrt \e}-\frac12)\partial_\theta \eta\big|^2+\frac{2r}\e|\eta-\bar\eta|^2\, d\theta dr
\\\label{123}
&\leq \pi \int_{ \frac{\sqrt \e}2}^{\sqrt \e} \frac1r(1+O(\sqrt\e))\, dr+\int_0^{2\pi} |\partial_\theta \eta|^2\, d\theta=\pi \log 2+o(1)
\end{align}
where we used the Poincar\'e inequality and $|\partial_\theta \eta|\leq C \sqrt \e$. Finally, by \eqref{normalcs} and \eqref{uvsv}, we compute
\begin{align*}
 \int_{B(a_k,\sqrt \e)} e^{in}_\e(U_\e) \, \vol_g
 &= 
\int_{\{ y\in \R^2 : |y|<\sqrt \e\} } \left[(1+O(\e))e_\e(V_\e)+ O(1)\right] \sqrt{g(y)}dy 
\\
 &\leq \pi \log \frac {\sqrt \e}{\e} + \iota_F + o(1) \quad \textrm{as } \e\to 0.
\end{align*}
To estimate the current
$j(U_\e)$, note that since $|V_\e|\le 1$ everywhere,
$$|j(V_\eps)|^2 \le |V_\e|^2 |\nabla V_\eps|^2 \le  |\nabla V_\eps|^2 \le 2 e_\e(V_\e).$$
 Thus for every $p\in [1,2)$,
we have by H\"older's inequality 
$$
\int_{B(0, \sqrt \e)} |j(V_\eps)|^p \, dy\leq \int_{B(0, \sqrt \e)} 2^{p/2} e_\e(V_\e)^{p/2} \, dy 
\leq |B(0, \sqrt \e )|^{1-\frac p 2}\bigg(\int_{B(0,\sqrt \e)} 2e_\e(V_\e) \, dy \bigg)^{p/2}\to 0.
$$ 
Combined with the equality $\Psi^*j(U_\e)=j(V_\e)-|V_\e|^2\Psi^*A
$ near $a_k$, as $|V_\e|\leq 1$ everywhere, we conclude that 
$\int_{B(a_k, \sqrt \e)} |j(U_\eps)|^p\to 0$ for every $p\in [1,2)$ as $\eps\to 0$.

\medskip

\nd {\it Step 4. Conclusion.}
Using the definition of $W(a,d,\Phi)$ in Section \ref{sec:renorm} and Step 3, we compute
\begin{align*}
E^{in}_\e(U_\e) 
&=
\int_{S_{\sqrt \e} } \frac 12 |D u^*|_g^2 \, \vol_g
+ \sum_{k=1}^n \int_{B(a_k,\sqrt \e)} e^{in}_\e(U_\e)\, \vol_g
\\
&=
W(a,d,\Phi) + n \pi \log \frac 1{\sqrt \e} + o(1)+ \pi n\log \frac {\sqrt \e}{\e} + n\iota_F + o(1)\\
&=W(a,d,\Phi) + n (\pi \log \frac 1{\e} +\iota_F) + o(1)
 \quad \textrm{as } \eps\to 0.
\end{align*}
As $U_\e := u^*$ in $S_{\sqrt \e}$, by Steps 1 and 3, we deduce that
\beq
\label{eqj}
j(U_\eps)-j(u^*)\to 0 \quad \textrm{ strongly in } \quad L^p(S)
\eeq
 for $p\in [1,2)$ which entails 
$dj(U_\eps)\to dj(u^*)$ strongly in $W^{-1,p}(S)$ for $p\in [1,2)$, in particular \eqref{convergence} and \eqref{Phi.conv} hold where $\Phi\in \calL(a;d)$ was given in the hypothesis as the flux integrals associated to 
$u^*$ by \eqref{Phiu.def}. 
As $U_\e$ is not smooth, using the smoothness argument in Step 1 of the proof of Proposition \ref{P.comp}, $U_\e$ can be replaced by a smooth vector field $u_\e$ with the desired properties. 

\end{proof}

\subsection{Lower bound}
{\bf Throughout most of this section, we assume that $(u_\e)_{\e\in (0,1)}$ is a 
sequence of {\it smooth} vector fields with $|u_\e|_g\leq 1$ in $S$ satisfying the hypotheses 
\eqref{Escaling2} and \eqref{convergence} of Proposition \ref{intrinsic.gammalim2} point 1) } (the smoothness assumption
follows by the argument in Step 1 of the proof of Proposition \ref{P.comp}). We will drop the
assumed bound on $|u_\e|_g$ only in  Step 6 in the proof of Proposition \ref{intrinsic.gammalim2} point 1), where we explain  how to get the lower bound in the general case.
All constants appearing in our estimates may depend on $S, n$, and the constant 
$C$ in \eqref{Escaling2}. Our first lemma allows us to  approximate the
vorticity $\omega(u_\e)$  by a sum of point masses that are 
well-separated, relative to the scale of the approximation.

\begin{lemma}
There exists $\e_0>0$ such that for $\e\in (0, \e_0)$, we can find $r_\e \in (\e^{\frac 1{2(n+1)}},\e^\beta)$ for some $\beta = \beta(n)>0$ and
$K=K(\eps)\in \Z_+$ distinct points $a^\eps=(a_{1, \eps}, \ldots, a_{K, \e})$ in $S$ and
nonzero integers $d^\eps=(d_{1, \eps},\ldots, d_{K, \eps})$ with \eqref{necessary} such that $\sum_{k=1}^K |d_{k, \e}| \le n$ (so, $K\le n$) and
\begin{equation}\label{farrr}
\|\omega(u_\e) - 2\pi \sum_{k=1}^K  d_{k,\e} \delta_{a_{k, \eps}} \|_{W^{-1,1}} \le r_\e^2,
\qquad
\dist_S(a_{k, \eps},a_{l, \eps})\ge \sqrt {r_\e} \mbox{ for all }1\leq k < l\leq K.
\end{equation}
In addition, there exists $\Phi^\e\in \calL(a^\e,d^\e)$ such that $|\Phi(u_\e) - \Phi^\e|\leq C\sqrt{r_\eps}$.
\label{L.sep}\end{lemma}

\begin{proof}
Let $0<q< \frac 1{n+1}$ and $\sigma_1 = \e^{q/2}$
Apply now Proposition \ref{P.comp} for $T=n$, $p=1$ and $q$,
and consider the $K(\leq n)$ distinct points $a_{k, \e}$ and nonzero integers $d_{k, \e}$ provided by it ($1\leq k\leq K$). We know by \eqref{fluxes.converge} that
$\dist_{\R^{2\mathfrak g}}(\Phi(u_\eps), \calL(a^\eps ,d^\eps))\leq \sigma_1^2$. 
Set the associated measure $\mu_1(a^\e,d^\e)=2\pi \sum_k d_{k, \e} \delta_{a_{k, \e}}$.
If $\dist_S(a_{k, \e}, a_{l, \e})\ge \sqrt{\sigma_1}$ for all  $k\ne \ell$, then this collection
satisfies \eqref{farrr} with $r_\e = \sigma_1$,  for small enough $\e$.
If not,  define a new collection of points
as follows: consider some pair $a_{l, \e}\neq a_{\ell, \e}$ such that $\dist_S(a_{l, \e}, a_{\ell, \e}) < \sqrt{\sigma_1}$. 
Remove this pair from $\{ a_{k, \e}\}$ and replace them by a point $P$ with the associated degree $d=d_{l, \e}+d_{\ell, \e}$ such that
$\dist_S(P, a_{l, \e}) < \frac 12 \sqrt{\sigma_1}$ and $\dist_S(P, a_{\ell, \e}) < \frac 12\sqrt{\sigma_1}$. The total sum of absolute values of the new degrees
could decrease, so it stays $\leq n$.
Note that
\begin{equation}
\| (d_{l, \e} \delta_{a_{l, \e}} + d_{\ell, \e} \delta_{a_{\ell, \e}}) - (d_{l, \e}+d_{\ell, \e}) \delta_P\|_{W^{-1,1}} \le
(|d_{l, \e}| + |d_{\ell, \e}| )\frac{\sqrt{\sigma_1}}2 \le n \frac{\sqrt{\sigma_1}}2. 
\label{induction}\end{equation}
Continue in this fashion until a new collection is reached (still denoted $\{a_{k, \e}\}$ and where the points of zero degree are suppressed)
such that $\dist_S(a_{l, \e}, a_{\ell, \e}) \ge \sqrt{\sigma_1}$ for all distinct $a_{l, \e}\neq a_{\ell, \e}$. This takes at most
$K-1 \leq n-1$ of the above steps because at each step the number of points decreases. It follows from \eqref{induction} that
\[
\| \omega(u_\e) - 2\pi \sum_{k=1}^K d_{k, \e} \delta_{a_{k, \e}}\|_{W^{-1,1}} \le \sigma_1^2 + n(n-1)\frac{\sqrt{\sigma_1}}2
\le n^2\sqrt{\sigma_1} =: \sigma_2^2 .
\]
Denoting $\mu_2$ the measure associated to this new collection of points $a_{k, \e}$  and degrees 
$d_{k, \e}$, we note that $\|\mu_1-\mu_2\|_{W^{-1,1}}\leq 
\|\mu_1-\omega(u_\e)\|_{W^{-1,1}}+\|\omega(u_\e)-\mu_2\|_{W^{-1,1}}  \le (n^2+1) \sqrt{\sigma_1}$.
If, for this collection,  $\dist_S(a_{\ell, \e}, a_{l, \e})\ge \sqrt{\sigma_2}$ for all  $k\ne \ell$, then again we are finished. If not, we continue in the same fashion.
Within (at most) $n-1$ iterations of this procedure, we obtain a collection of
points satisfying \eqref{farrr} for some $r_\e \le C(n)\e^\beta$ for some (large) $C(n)$ and
(small) positive $\beta$. By decreasing $\beta>0$ we may suppose that $C(n) = 1$.
Moreover, if $\tilde a^\e$ is the final collection of points with the nonzero degrees $\tilde d^\e$, denoting $\tilde \mu(\tilde a^\e, \tilde d^\e)$ the associated measure,
we have that  $\|\mu_1-\tilde \mu\|_{W^{-1,1}}\leq C\sqrt{r_\eps}$. Now we use \eqref{dist_lat} and \eqref{fluxes.converge} to conclude that
$$\dist_{\R^{2\mathfrak g}}(\Phi(u_\eps), \calL(\tilde a^\e, \tilde d^\e))\leq \dist_{\R^{2\mathfrak g}}(\Phi(u_\eps), \calL(a^\e, d^\e))+\dist_{\R^{2\mathfrak g}}(\calL(a^\e, d^\e), \calL(\tilde a^\e, \tilde d^\e))\leq C\sqrt{r_\eps}.$$
\end{proof}

Our next lemma provides a good lower energy bound away from the
vortices. This will be used several times in the proof of the compactness
and lower bound assertions of Proposition \ref{intrinsic.gammalim2}.

\begin{lemma}\label{L.qgamma}
Using the notations in Lemma \ref{L.sep}, let $a^\e=(a_{1, \eps}, \ldots, a_{K, \e})$, $d^\e=(d_{1, \eps}, \ldots, d_{K, \e})$ satisfy \eqref{farrr} for some 
$r_\e\in ( \e^{\frac 1{2(n+1)}}, \e^\beta)$ for some $\beta = \beta(n)>0$ and $\Phi^\e=(\Phi_{k,\eps})_{k=1}^{2\mathfrak g} \in \calL(a^\e, d^\e)$ such that $| \Phi(u_\e)  - \Phi^\e| \leq C\sqrt{r_\eps}$.
Let $u^*(a^\e, d^\e, \Phi^\e)$ be a  canonical harmonic vector field given in Theorem \ref{P1} with the associated current $j^*_\e := j(u^*(a^\e, d^\e, \Phi^\e))$. 
Then for all sufficiently small $\e>0$, 
\begin{align}
\int_{S_{r_\e}} e_\e^{in}(u_\e) \vol_g 
&\ge
 \pi(\sum_{k=1}^K {d^2_{k,\e}}) \log \frac 1 {r_\e} + W(a^\e,d^\e,\Phi^\e) \nonumber \\
&\,
 + \int_{S_{r_\e}}
\Big(\frac 12 
\left| \frac { j(u_\e)}{|u_\e|_g } - j^*_\e \right|_g^2 + e_\e^{in}(|u_\e|_g) \Big) \vol_g 
-O(r_\eps^{1/3}) - O(r_\eps^{1/2}|\Phi^\e|^2)
\label{qgamma}
\end{align}
for $S_{r_\e} := S\setminus \cup_{k=1}^K B_{r_\e}(a_{k, \e})$.
\end{lemma}

\begin{proof}
The proof uses some arguments
from \cite{JeSp}, Theorem 2.
First, we use Section \ref{sec:calc} and  elementary algebra to find that
\begin{equation}
 e_\e^{in}(u_\e) = \frac 12|j_\e^*|_g^2 +
\frac 12\left| \frac { j(u_\e)}{|u_\e|_g } - j^*_\e \right|_g^2 +
 (j^*_\e, \frac{j(u_\e)}{|u_\e|_g} - j^*_\e)_g
+ e_\e^{in}(|u_\e|_g) \quad \textrm{in} \, \,  S_{r_\e}. 
\label{ilb.1}\end{equation}
In addition,
by \eqref{W.def}, we have that
\[ 
\frac 12
\int_{S_{{r_\e}}}|j^*_\e|_g^2\vol_g  = \pi(\sum_k d^2_{k, \eps}) \log \frac 1{{r_\e}} + W(a^\e,d^\e,\Phi^\e) + O(\sqrt {r_\e}) + O(r_\eps^{3/2}|\Phi^\e|^2) \quad\mbox{ as }\e\to 0 \ .
\] 
After combining these,  we find that to prove \eqref{qgamma}, 
it suffices to prove that
\begin{equation}
\left| \int_{S_{r_\e}}
 (j^*_\e, \frac{j(u_\e)}{|u_\e|_g} - j^*_\e)_g
\vol_g   \right|  =O({r_\eps}^{1/3}) + O({r_\eps^{1/2}}|\Phi^\e|^2) \quad\mbox{ as }\e\to 0.
\label{ilb.2}\end{equation}
Toward this end, we let $\psi_\eps:=\psi(a^\eps, d^\eps)$ be the solution of \eqref{psi.def} and we start by using  \eqref{form_jstar} to  write 
\begin{align*}
\int_{S_{r_\e}}
 (j^*_\e, \frac{j(u_\e)}{|u_\e|_g} - j^*_\e)_g
 \vol_g   
 &= 
\int_{S_{r_\e}}
 (d^*\psi_\eps , \frac{j(u_\e)}{|u_\e|_g} - j^*_\e)_g \vol_g
 +\sum_{k=1}^{2\mathfrak g} \Phi_{k, \eps} 
\int_{S_{r_\e}}
 (\eta_k , \frac{j(u_\e)}{|u_\e|_g} - j^*_\e)_g \vol_g\\
&=: L_0 +\sum_{k=1}^{2\mathfrak g} \Phi_{k, \eps}  L_k.
\end{align*}

\nd {\it Step 1. Estimate of $L_k$ for $k=1, \dots, 2\mathfrak g$}. We decompose
\[
L_k = \int_S (\eta_k, j(u_\e)- j^*_\e)_g \, \vol_g+ \int_S (\eta_k, \frac{j(u_\e)}{|u_\e|_g})_g(1-|u_\e|_g) \vol_g - \int_{S\setminus {S_{r_\e}}}
(\eta_k , \frac{j(u_\e)}{|u_\e|_g} -  j^*_\e)_g \vol_g .
\]
We estimate the terms on the right-hand side.
First, as $\Phi^\eps$ are the flux integrals associated to $u^*(a^\e, d^\e, \Phi^\e)$, we deduce
\[
|\int_S (\eta_k, j(u_\e)- j^*_\e)_g \vol_g | = |\Phi_k(u_\e) - \Phi_{k, \eps} | =O(\sqrt{r_\eps}).
\]
Next, since by \eqref{formula12}, $\left|\frac{ |j(u_\e)|_g}{|u_\e|_g}(1-|u_\e|_g) \right| \le \big|1-|u_\e|_g\big| |Du_\eps|_g \leq C \e e_\e^{in}(u_\e)$, it is clear that
\[
\left| \int_S (\eta_k, \frac {j(u_\e)}{|u_\e|})_g (1-|u_\e|_g)\, \vol_g \right|
\le C\|\eta_k\|_{L^\infty} \e \logeps.
\]
We split the remaining term into two pieces. By Cauchy-Schwarz,
\begin{align*}
\left|\int_{\cup_{\ell=1}^{K_\eps} B_{r_\e}(a_{\ell, \eps})}  (\eta_k, \frac {j(u_\e)}{|u_\e|_g})_g \vol_g\right|
&\le \left( \int_{ \cup_{\ell=1}^{K_\eps} B_{r_\e}(a_{\ell, \eps})} |\eta_k|^2 \vol_g \int_{\cup_{\ell=1}^{K_\eps} B_{r_\e}(a_{\ell, \eps})} |Du_\e|_g^2  \vol_g\right)^{1/2}\\
 &= O({r_\e}\logeps^{1/2}) = O({r_\e}^{1/2}).
 \end{align*}
Next, from \eqref{form_jstar}, \eqref{newpsi} and properties of the Green's function
in Section \ref{sec:renorm} (in particular that $\|d^*G(\cdot, a_{k, \eps}) \|_{L^1(B_{r_\e}(a_{\ell, \eps}))}=O(r_\eps)$ for every $k$ and $\ell$), one readily checks that 
\[
\left|\int_{ \cup_{\ell=1}^{K_\eps} B_{r_\e}(a_{\ell, \eps})} (\eta_k, j^*_\e)_g\vol_g\right|
\le C\|\eta_k\|_{L^\infty}(r_\eps+|\Phi^\eps|r_\eps^2\|\eta_k\|_{L^\infty}).
\]
By combining the above, we conclude that
\[
|\Phi_{k, \eps}  L_k |  =O\bigg((\sqrt{r_\eps}+|\Phi^\eps|r_\eps^2)|\Phi^\e|\bigg)=O(\sqrt{r_\e})+O(\sqrt{r_\e}|\Phi^\e|^2) \qquad\mbox{ for every $k=1, \dots, 2\mathfrak g$}.
\]

\nd {\it Step 2. Estimate of  $L_0$}.
Next, with $\psi_\eps:=\psi(a^\eps, d^\eps)$ the $2$-form solving \eqref{psi.def}, we define
\[
\tilde \psi_\e(x) :=
\begin{cases}\psi_\e(x) &\mbox{ in }S_{r_\e}\\
\psi_\e(x) + d_k( \log \dist_S(x, a_{\ell, \eps}) - \log {r_\e} ) \vol_g &\mbox{ in } B_{r_\e}(a_{\ell, \eps}), \, \ell=1,\dots,  K_\eps.
\end{cases}
\]
Since $\dist_S(a_{l, \eps}, a_{\ell, \eps})\geq \sqrt{r_\eps}$, it follows from \eqref{newpsi} 
and properties of the Green's function from  Section \ref{sec:renorm}
that $\tilde \psi_\e$ is Lipschitz continuous in $S$ and $C^1$ in
$\cup_{\ell=1}^{K_\eps} B_{r_\e}(a_{\ell, \eps})$,
with Lipschitz constant bounded by $ C/{\sqrt{r_\e}}$
in $S$. 
Thus we can write 
\[
d^* \tilde \psi_\e = {\mathbf 1}_{S_{r_\e}} d^*\psi_\e+  \xi_\e,  
\qquad\mbox{ with } 
\quad \| d^*\tilde \psi_\e\|_{L^\infty(S)} ,  \| \xi_\e \|_{L^\infty(S)}  \le C / \sqrt{r_\e}
\]
where $\xi_\e$ is a $1$-form supported in $\cup_{\ell=1}^{K_\eps} B_{r_\e}(a_{\ell, \eps})$.
With this notation we have  
\[
L_0 = \int_S (d^*\tilde \psi_\e, j(u_\e)- j^*_\e)_g \vol_g + \int_S (d^*\tilde \psi_\e, \frac{j(u_\e)}{|u_\e|_g})_g(1-|u_\e|_g) \vol_g- 
\int_{S\setminus S_{r_\e}}
(\xi_\e, \frac{j(u_\e)}{|u_\e|_g} -  j^*_\e)_g \vol_g.
\]
We consider the terms on the right-hand side.
First, writing $u^*_\e:=u^*(a^\e, d^\e, \Phi^\e)$, by the Stokes theorem and the definition of the Hodge star operator, we have 
\[
\int_S (d^*\tilde \psi_\e, j(u_\e)- j^*_\e)_g \vol_g
 = 
\int_S \star\tilde \psi_\e\,  (\omega(u_\e) - \omega(u^*_\e)) 
=
\int_S \star \tilde \psi_\e\,  \big(\omega(u_\e)-2\pi \sum_{\ell=1}^{K_\eps} d_{\ell, \eps} \delta_{a_{\ell, \eps}}\big)
\]
and from this, together with \eqref{farrr}, we conclude that
\[
\left|\int_S (d^*\tilde \psi_\e, j(u_\e)- j^*_\e)_g \vol_g \right| \le
\| \star \tilde \psi_\e\|_{W^{1,\infty}}\| \omega(u_\e) - 2\pi \sum_{\ell=1}^{K_\eps} d_{\ell, \eps} \delta_{a_{\ell, \eps}}\|_{W^{-1,1}}\le C {r_\e}^{3/2}.
\]
The other terms in the decomposition of $L_0$ are estimated exactly like their 
counterparts in Step 1 above, using the estimates
$\|d^*\tilde \psi_\e\|_\infty$, $\| \xi_\e \|_\infty \le {C} {r_\e^{-1/2}}$.
This leads to
\[
|L_0| \le C(\ {r_\e}^{3/2} + \frac  \e {r_\e}  \logeps  +  r_\e^{1/2} \logeps^{1/2} + {r_\e}^{1/2}+{r_\e}^{3/2}|\Phi^\eps|) 
=O({r_\e}^{1/3}+r_\e^{2}|\Phi^\eps|^2)
\]
as $\e\to 0$.
\end{proof}

Our next lemma provides a rather crude estimate of the energy near the ``vortex cores".

\begin{lemma}\label{L.near}
Using the notations in Lemma \ref{L.sep}, let $a^\e=(a_{1, \eps}, \ldots, a_{K, \e})$, $d^\e=(d_{1, \eps}, \ldots, d_{K, \e})$ satisfy \eqref{farrr} for some 
$r_\e\in (\e^{\frac 1{2(n+1)}}, \e^\beta)$ for some $\beta = \beta(n)>0$. Then for all sufficiently small $\e>0$, 
\begin{equation}\label{lbd.near}
\int_{B_{{r_\e}}(a_{k, \e})} e_\e^{in}(u_\e)
 \vol_g \ge
 |d_{k, \e}| (\pi \log \frac{{r_\e}}{\e}- C ), \quad k=1, \dots, K.
\end{equation}
\end{lemma}

\begin{proof}
Consider the collection of balls $\{B_{\ell, \sigma}\}_\ell$ provided by applying Proposition \ref{P.vballs}
to $u_\e$,
with $T=n$ and $\sigma = \frac 1{4n+4} {r_\e}$. In view of \eqref{balls4}, it suffices to show that for every $1\leq k\leq K$, 
\begin{equation}\label{deg.bigenough}
\sum_{ \ell : B_{\ell, \sigma} \subset B_{r_\e}(a_{k,\e}) }| d_{\ell, \sigma} | \ge |d_{k, \e}|.
\end{equation}
To do this, we fix some $k\in \{1, \dots, K\}$, and we define the set
\[
\calT_k := \{ r \in (0,r_\e) : \partial B_r(a_{k,\e}) \cap (\cup_\ell B_{\ell, \sigma}) = \emptyset \}.
\]
It follows from \eqref{balls3} that $|\calT_k|\ge \frac 12 r_\e$.
Now define a Lipschitz function $\varphi:S\to \R$ by
\[
\varphi(x) := f(R_k(x)) , \qquad\mbox{ for }\quad f(r) :=|\calT_k| - \int_0^r {\mathbf 1}_{\calT_k}(s) ds,\qquad R_k(x) := \dist_S(x,a_{k, \e}).
\]
It is clear that $\varphi(a_{k,\e})=|\calT_k|$, $\| \varphi\|_{W^{1,\infty}} \le 2$ and the support of $\varphi$ is inside $B_{r_\e}(a_{k,\e})$, so 
\[
|\int \varphi \, \omega(\ue)  - 2\pi d_{k, \e} \varphi(a_{k,\e}) | \le \| \varphi\|_{W^{1,\infty}} \| \omega(u_\e) - 2\pi \sum_{l=1}^K d_{l, \e} \delta_{a_{l,\e}} \|_{W^{-1,1}} \le 2 r_\e^2,
\]
because by \eqref{farrr}, $a_{l, \e}\notin B_{r_\e}(a_{k,\e})$ if $l\neq k$.
Next, we define $\tilde u_\e$ as in \eqref{tildeu.def}, so that by Step 3 in the proof of Proposition \ref{P.comp}:
\[
\int \varphi \, \omega(u_\e) 
= 
\int \varphi \, \omega(\tilde u_\e)  + O(\e\logeps) .
\]
We  fix a moving frame $\{\tau_1,\tau_2\}$ defined in $B_{r_\e}(a_{k,\e})$ and let $A$ be the connection $1$-form associated to it.
Then by \eqref{rel_connection} we may write $\omega(\tilde u_\e) = d( j(\tilde u_\e) + A)$ in $B_{r_\e}(a_{k,\e})$. It follows
\[
\int \varphi \, \omega(\tilde u_\e) =\underbrace{\int d\bigg(\varphi ( j(\tilde u_\e) + A)\bigg)}_{=0}-\int d\varphi \wedge (j(\tilde u_\e) + A)
=
- \int f'(R_k) dR_k \wedge (j(\tilde u_\e) + A) .
 \]
But the coarea formula and the definition of $f$ imply that
\begin{align*}
- \int_{B_{r_\e}(a_{k,\e})} f'(R_k) dR_k \wedge (j(\tilde u_\e) + A)
&=
\int_{r\in \calT_k} \int_{\partial B_r(a_{k,\e})} (j(\tilde u_\e)+A)\, dr\\
&=
2\pi \int_{r\in \calT_k} \deg(\tilde u_\e, \partial B_r(a_{k,\e}))\, dr,
\end{align*}
where we used \eqref{deg.def} and $|\tilde u_\e|_g=1$ on $\partial B_r(a_{k,\e})$ for every $r\in \calT_k$. Combining these, we find that 
\[
\left| 2\pi d_{k, \e} |\calT_k| - 2\pi \int_{r\in \calT_k} \deg (\tilde u_\e, \partial B_r(a_{k,\e})) \, dr \right|\le Cr_\e^2
\]
for small $\eps>0$. As $|\calT_k|\geq \frac{r_\eps}{2}$, it follows that if $\e$ is small enough, then $|\deg(\tilde u_\e, \partial B_r(a_{k,\e}))| \ge |d_{k, \e}|$
for a large set of $r\in \calT_k$. Choose one of these $r\in \calT_k$. Since 
\[
\deg(\tilde u_\e,\partial B_r(a_{k, \e})) = \sum_{\ell\, :\, B_{\ell, \sigma} \subset B_r(a_{k,\e})} d_{\ell, \sigma},
\]
this implies \eqref{deg.bigenough}.
\end{proof}

We now present the proof of Proposition \ref{intrinsic.gammalim2}. The bulk of the proof is 
devoted to a sharp lower bound near the vortices, which uses
preliminary estimates provided by Lemmas \ref{L.qgamma} and \ref{L.near}
to  refine the conclusion of Lemma \ref{L.near}.

\begin{proof}[Proof of Proposition \ref{intrinsic.gammalim2} point 1)]
By Step 1 in the proof of Proposition \ref{P.comp}, we may assume that $u_\eps$ are smooth vector fields on $S$ with $|u_\e|_g\le 1$ everywhere (as the cutting $u_\eps$ by 
$\hat u_\e$ and then regularizing as in Lemma \ref{L.density}, the new vector field satisfies the hypotheses of the Proposition but
has less energy). We will explain in Step 6 below how to get the result for general vector fields without the constraint on the length of $u_\e$.
Next to the distinct points $a=(a_1,\ldots, a_{n_0})\in S$
and nonzero integers $d=(d_1,\ldots, d_{n_0})$ (given in the hypothesis of Proposition \ref{intrinsic.gammalim2}), using Lemma  \ref{L.sep}, we find $K=K_\eps$ distinct points $a^\e=(a_{1, \eps}, \ldots, a_{K, \e})$ and  nonzero integers $d^\e=(d_{1, \eps}, \ldots, d_{K, \e})$ satisfying \eqref{farrr} for some 
$r_\e\in (\e^{\frac 1{2(n+1)}}, \e^\beta)$ and for some $\beta = \beta(n)>0$, $\sum_k |d_{k, \e}|\leq n$ and $\Phi^\e=(\Phi_{k,\eps})_{k=1}^{2\mathfrak g} \in \calL(a^\e, d^\e)$ such that $| \Phi(u_\e)  - \Phi^\e| \leq C\sqrt{r_\eps}$.
Let $j^*_\e := j(u^*(a^\e,d^\e,\Phi^\e))$ as defined in \eqref{form_jstar}.\\

\nd {\it Step 1. We prove that $\sum_{k=1}^K |d_{k, \e}| = n$, and we control $\dist_S(a_{k,\eps}, a_\ell)$ and the
signs of $d_{k,\e}$ for small $\e>0$, i.e., every $a_{k, \eps}$ is close to some $a_\ell$ with $\sign d_{k,\e}=\sign d_\ell$. Moreover, $W(a^\eps, d^\eps, \cdot)$ is coercive in $\Phi^\eps$ and $W(a^\eps, d^\eps, \Phi^\eps)\to \infty$ if a limit degree satisfies $|d_\ell|>1$.} 
First, Lemma \ref{L.sep} and \eqref{convergence} imply that
\beq
\label{to_be_erased}
\|\sum_{k=1}^{K_\eps} d_{k,\e} \delta_{a_{k,\e}} - \sum_{\ell = 1}^{n_0}d_\ell \delta_{a_\ell}\|_{W^{-1,1}} := s_\e  \to 0\quad\mbox{ as $\e\to 0$.}
\eeq
For small $\eps$ and for $\ell = 1,\ldots, n_0$, we consider the Lipschitz function 
$$
f_{\e, \ell}(x) := 
\sign(d_\ell)[ 2s_\e - \dist_S(x, a_\ell)]^+,\quad x\in S,
$$
where $[\cdots ]^+ = \max\{[\cdots], 0\}$. We also define, for $\ell=1, \dots, n_0$,
\begin{align*}
I_\ell^\e &:= \{ k\in \{1, \dots, K\} : \dist_S(a_{k,\e}, a_\ell)  \leq  2 s_\e\}, \\
I_\ell^{\e,+} &:= \{ k\in I_\ell^\e : \sign(d_{k,\e}) = \sign(d_\ell) \}.
\end{align*}
We henceforth assume that $\e$ is small enough that $s_\e <1$  and 
the closed
balls  $\{ \bar B_{2s_\e}(a_\ell)\}$ are disjoint, and hence
$\{ I^\e_\ell\}_\ell$ are pairwise disjoint.
It follows from our convention for 
defining Sobolev norms
(see Section \ref{sec:soboesp}) that
$\| f_{\e,\ell}\|_{W^{1,\infty}} = 1$.
Thus
\[
\int f_{\e,\ell} \, \big(\sum_{l = 1}^{n_0}d_l \delta_{a_l}-\sum_{k=1}^{K_\e} d_{k,\e} \delta_{a_{k,\e}}\big) \le  \| f_{\e, \ell}\|_{W^{1,\infty}} 
\|\sum_{k=1}^{K_\eps} d_{k,\e} \delta_{a_{k,\e}} - \sum_{l = 1}^{n_0}d_l \delta_{a_l}\|_{W^{-1,1}} = s_\e.
\]
However, the definition and the smallness condition on $s_\e$ imply that
\begin{align*}
\int f_{\e,\ell} \, \big(\sum_{l = 1}^{n_0}d_l \delta_{a_l}-\sum_{k=1}^{K_\e} d_{k,\e} \delta_{a_{k,\e}}\big) 
&= 
2 s_\e |d_\ell| - \sign(d_\ell)\sum_{k\in I^\e_\ell}d_{k,\e}(2s_\e - \dist_S(a_{k,\e}, a_{\ell})) \\
&\ge 
2 s_\e |d_\ell| -  2s_\e \sum_{k\in I^{\e,+}_\ell} |d_{k,\e}|.
\end{align*}
We combine these facts and divide  by $2s_\e$ 
to find that
$
|d_\ell| - \sum_{k\in I_\ell^{\e,+}}|d_{k,\e}| \le  \frac 12$.
Since both terms on the left are integers, it follows that 
\begin{equation}
|d_\ell| \le \sum_{k\in I_\ell^{\e,+}}| d_{k,\e}|
\quad\mbox{ for $\ell = 1,\ldots, n_0$.}
\label{dldk}\end{equation}
Summing over $\ell$ and using the disjointness of $\{ I^\e_\ell\}_\ell$,
we obtain
\[
n = \sum_{\ell = 1}^{n_0} |d_\ell|
 \le \sum_{\ell}\sum_{k\in I_\ell^{\e,+}} |d_{k,\e}|
 \le \sum_{\ell}\sum_{k\in I_\ell^{\e}} |d_{k,\e}|
\le \sum_{k=1}^{K_\e}|d_{k,\e}| \le n.
\]
It follows that in fact for small $\eps$,
\begin{equation}
\sum_{k=1}^K |d_{k, \e}| = n, \qquad \sign d_{k,\e} = \sign d_\ell\ \mbox{ for all }k\in I^\e_\ell, \,
\qquad
\sum_{k\in I^\e_\ell} d_{k, \e} = d_\ell
\label{dkep}\end{equation}
where the last equality holds due to \eqref{dldk} and the first two
equalities in \eqref{dkep}.
In particular, one has that $a_{k, \eps}\to a_\ell$ for every $k\in I^\e_\ell$; as for all $1\leq \ell\leq n_0$ the indices $\{ d_{k, \e}\, :\, k\in I^\e_\ell\}$ have the same sign, the explicit formula \eqref{W.formula} for $W$ implies that $W(a^\e, d^\e, \cdot)$ is coercive in $\Phi_\e$:
\[
\begin{aligned}
W(a^\e, d^\e, \Phi^\e) &\ge -C(a,d) + |\Phi_\e|^2 \qquad &\mbox{ for all }\e \in (0,\e_0),\\
W(a^\e,d^\e, \Phi^\e) &\to +\infty \qquad &\mbox{ as $\e\to 0$,  if $|d_\ell |>1$ for any $\ell$.} 
\end{aligned}
\]
The point is that in the sum $\sum_{1\le l<k\le n} d_{l,\e} d_{k,\e} G(a_{k,\e}, a_{l,\e})$
in formula \eqref{W.formula} for $W$,
Step 1 implies that $\dist_S(a_{k,\e},a_{l,\e})$ is bounded away from $0$, for small
$\e$, for all pairs $k,l$ such that $d_{k,\e} d_{l,\e}<0$.
Thus contributions from this sum are bounded below, and all other terms are manifestly also bounded from below as $\dist_S(a_{k,\e},a_{l,\e})$ is bounded. Moreover, if $|d_\ell|>1$ for any $\ell$, then multiple points $a_{k,\e}$
with the same sign $d_{k,\e}$
must converge to the same $a_\ell$, causing the  sum 
$\sum_{1\le l< k\le n} d_{l,\e} d_{k,\e} G(a_{k,\e}, a_{l,\e})$ 
to diverge.

\medskip

\nd {\it Step 2. We prove that $|d_\ell|=|d_{k, \e}| = 1$ for all $1\leq k\leq K$ and $1\leq \ell\leq n_0$ (so, $K=n=n_0$), $\dist_S(a_{k, \e}, a_{\ell, \e})\geq C>0$ for every $k\neq \ell$ and $\{\Phi_\eps\}$ converge (for a subsequence) as $\eps\to 0$}. Indeed, by combining the energy estimates away from the vortex cores and inside the vortex cores as shown in Lemmas \ref{L.qgamma} and \ref{L.near},
we find for all sufficiently small $\e>0$:
\begin{align}
\int_{S} e_\e^{in}(u_\e) \vol_g 
&\ge
\pi \sum_{k=1}^K |d_{k, \e}| \log \frac 1 \e + \pi  \sum_{k=1}^K (d_{k, \e}^2 - |d_{k, \e}|) \log \frac  1{r_\e} + W(a^\e,d^\e,\Phi^\e)
\nonumber\\
&\qquad+
\int_{S_{r_\e}}
\frac 12\big| \frac { j(u_\e)}{|u_\e|_g } - j^*_\e \big|_g^2 + e_\e^{in}(|u_\e|_g)\vol_g  \ - C \,,
\label{qg2}\end{align}
where $S_{r_\e} = S\setminus \cup_{k=1}^K B_{r_\e}(a_{k, \e})$.
Then the upper bound \eqref{Escaling2} and \eqref{dkep} imply that
\[
 \pi  \sum_{k=1}^K (d_{k, \e}^2 - |d_{k, \e}|) \log \frac  1{r_\e} + W(a^\e,d^\e,\Phi^\e) \le C.
\]
Combined with the coercivity of $W$ proved in Step 1, it follows that 
\[
|d_{k, \e} | = 1\mbox{ for all }k,\, (\textrm{so, } K=n),\qquad |d_\ell| = 1 \,\mbox{ for all } \ell, \qquad
|\Phi^\e|^2 \le C
\]
for all sufficiently small $\e$. Also, Step 1 implies $I_\ell^{\e, +}=I_\ell^\e$ containing only one point $a_{\ell, \e}$ that converges to $a_\ell$ as $\eps\to 0$, thus, yielding $\dist_S(a_{k, \e}, a_{\ell, \e})\geq C>0$ for every $k\neq \ell$ for small $\e>0$.
In particular, \eqref{Phi.conv} holds (that is, $\Phi^\e$ converges) after 
 possibly passing to a subsequence, and \eqref{qg2} implies that
\[
\int_{S_{r_\e}}
\frac 12\left| \frac { j(u_\e)}{|u_\e|_g } - j^*_\e \right|_g^2 + e_\e^{in}(|u_\e|_g)\vol_g  \
\le C.
\]
It follows from the coarea formula that
\[
\int_ {r_\e}^{\sqrt{r_\e}} \sum_{k=1}^n
\int_{\partial B_t(a_{k, \e})} 
\left(\frac 12\left| \frac { j(u_\e)}{|u_\e|_g } - j^*_\e \right|_g^2 + e_\e^{in}(|u_\e|_g)\right) d\calH^1
dt \le C
\]
and hence, since $r_\e \leq \e^\beta$ for some positive $\beta$, that there exists $t_\e\in  (r_\e, \sqrt {r_\e})$ such that 
\begin{equation}
\sum_{k=1}^n
t_\e  \int_{\partial B_{t_\e}(a_{k, \e})} 
\left(\frac 12\left| \frac { j(u_\e)}{|u_\e|_g } - j^*_\e \right|_g^2 + e_\e^{in}(|u_\e|_g)\right) d\calH^1
\le C\logeps^{-1} . 
\label{really}\end{equation}

\nd {\it Step 3. Passage in normal coordinates}. 
We now fix some $k\in \{1,\ldots, n\}$. We assume for concreteness, and to simplify the notation, that $d_{k, \e}= +1$.
We aim to rewrite the integral around $\partial B(t_\e, a_{k,\e})$
in exponential normal coordinates near $a_{k, \e}$, using a moving frame such that
\eqref{A0} holds 
(see \eqref{uv} and the discussion in Section \ref{sect.coords} for notation).
Since we have arranged that $|u_\e|_g\le 1$ everywhere,
it follows from \eqref{normalcs}, \eqref{uvsv} and \eqref{jstar.log} that 
on $\partial B_{t_\e}(a_{k, \e})$, 
\begin{align}
\left| \frac { j(u_\e)}{|u_\e|_g } - j^*_\e \right|_g^2(\Psi(y))
&= [1+ O(t_\e^2) ] \left| \frac { j(v_\e)}{|v_\e| }(y) -   d\theta +O(1)\right|^2
\nonumber\\
&= [1+ O(t_\e^{1/2})]
\left| \frac { j(v_\e)}{|v_\e| }(y) -   d\theta\right|^2 + O(t_\e^{-1/2}).
\label{seriously}
\end{align}
where $d\theta := \frac 1{|y|^2}( y_1dy_2 - y_2dy_1)$ and we used the Young inequality $|z_1+z_2|^2\leq (1+t_\e^{1/2})|z_1|^2+(1+t_\e^{-1/2})|z_2|^2$.
Combining this with \eqref{really} and again using \eqref{uvsv}, 
we obtain
\begin{equation}\label{teest}
t_\e  \int_{\{y\in \R^2 : |y|=t_\e\} } 
\frac 12\left| \frac { j(v_\e)}{|v_\e| }(y) - d\theta \right|^2 + e_\e(|v_\e|)(y) \ d\calH^1(y)
\le C \logeps^{-1}.
\end{equation}
\\

\nd {\it Step 4. We will show that 
$$
\int_{\{y\in \R^2 : |y|< t_{\e} \}} e_\e(v_\e) \ dy \ge \pi \log \frac {t_\e}\e + \iota_F  + o(1) \quad \textrm{as } \eps\to 0.
$$
}
To do this, it is convenient to define
\begin{equation}
\beta_\e(v;r) := r\int_{\partial B_r}\left( \left | \frac {j(v)}{|v|} - d\theta\right|^2 + e_\e(|v|) \right)d\calH^1(y),
\label{betaep.def}\end{equation}
for $v\in H^1(O;\C)$, where $O$ is a neighborhood of the origin in $\R^2$ containing the disk $B_r=\{y\in \R^2 : |y|< r \}$.
We further define for small $\delta>0$:
\[
I^\delta(\e,r) := \inf \left\{ \int_{B_r} e_\e(v) \ :  \ v\in H^1(B_r;\C),  \
\beta_\e(v, r)\le \delta \right\} \ .
\]
By a change of variables one finds that for $v^r(x) := v(rx)$,
\[
\beta_\e(v,r) =  \beta_{\e/r}(v^r,1), \qquad\mbox{ and thus } \ \ \ \ 
I^\delta(\e, r) = I^{\delta}(\e/r, 1) .
\]
Note that $\beta_\e(v,1)=0$ implies that $|v|=1$ and $j(v)=d\theta$ on $\partial B_1$ (i.e., $v=e^{i(\theta+\bar \eta)}$ for some constant $\bar \eta$) so that 
by \eqref{gammaF.def}:
$$\lim_{\e\to 0} \big( I^0(\e,1) - \pi \log \frac 1\e\big)=\iota_F$$
(see \cite[Lemma III.1]{BBH}).
We also claim that for small $\delta>0$:
\begin{equation}\label{Idelta}
I^\delta(\e,1) \ge \pi \log \frac 1\e + \iota_F - C \delta+o(1) \quad \textrm{as } \eps\to 0.
\end{equation}
This follows from the fact that if $v\in H^1(B_{1};\C)$ and $\beta_\e(v;1)\le \delta$,
then $v$ admits an extension to a function $\tilde v\in H^1(B_{2};\C)$ such that
$\beta_\e(\tilde v,2)=0$
and
\[
\int_{B_{2}\setminus B_{1}} e_\e(\tilde v)
\le  \pi \log 2 + C \delta . 
\]
Indeed, this may be done by writing $v(y) = \rho(y) e^{i(\theta+\eta(y))}$
for $|y|=1$ with $\rho, \eta\in H^1(\partial B_1)$ with $|1-\rho|=O(\e)$ pointwise on  $\partial B_1$ and 
$\|\partial_\tau \eta\|^2_{L^2(\partial B_1)}=O(\delta)$ (due to the assumption $\beta_\e(v,1)\leq \delta)$.  
Then for $1\le r \le 2$,
we set $\tilde v(ry) :=  [1+ (2-r)(\rho(y)-1)]\exp i[\theta +\bar \eta  + (2-r)(\eta(y)-\bar \eta)]$,
where $\bar \eta$ is the mean of $\eta$ over $\partial B_1$.
Thus for small $\e>0$:
\[
I^0(\e,2) \le
 \int_{B_2}e_\e(\tilde v) \,dy
 \le 
 \int_{B_1}e_\e(v) \,dy +\pi \log 2 + C \delta.
 \]
It follows that $I^\delta(\e,1) \ge I^0(\e,2) - \pi \log 2 - C\delta$,
which implies \eqref{Idelta}. Then by combining \eqref{Idelta} and \eqref{teest},
we conclude Step 4. \\

\nd {\it Step 5. Lower bound \eqref{lb1}}. By Step 4 combined with \eqref{reduce}, we deduce that
\[
\int_{B_{t_\e}(a_{k, \e})}e_\e^{in}(u_\e) \vol_g \ge \pi \log \frac{t_\e}{\e} +\iota_F   + o(1) 
\qquad\mbox{ as $\e\to 0$, \ for every $1\leq k\leq n$.}
\]
As $t_\eps\in (r_\e, \sqrt{r_\e})$ and $\dist_S(a_{k,\e}, a_{\ell, \e})=O(1)\geq t_\e$ for $k\neq \ell$, we see that \eqref{farrr} holds true for $t_\e$ so that we can apply 
Lemma \ref{L.qgamma} for $t_\e$ yielding:
\begin{align*}
\int_{S_{t_\e}} e_\e^{in}(u_\e) \vol_g 
&\ge n\pi \log \frac 1{t_\e} + W(a^\e, d^\e, \Phi^\e)\\
&\qquad\qquad + \int_{S_{t_\e}}\frac 12
\left| \frac { j(u_\e)}{|u_\e|_g } - j^*_\e \right|_g^2 + e_\e^{in}(|u_\e|_g)\vol_g  \ -  o(1)
\end{align*}
as $\e \to 0$. By adding these inequalities and noting that for every fixed $\sigma>0$, $j^*_\e\to j^*=j(u^*(a,d,\Phi))$
uniformly on $S\setminus \cup_{k=1}^n B_\sigma(a_k)$ and $W(a^\e, d^\e, \Phi^\e)\to W(a,d,\Phi)$
as $\e\to 0$, we complete the proof of \eqref{lb1}.\\

\nd {\it Step 6. Conclusion}. 
We finally consider the general case, without the assumption $|u_\e|_g\le 1$.
Due to the cutting $u_\e$ by $\hat u_\e$ (see Step 1 of the proof of Proposition~\ref{P.comp}), as $e_\e^{in}(|\hat u_\e|_g)=0$ a.e. in $\{|u_\e|_g\geq 1\}$, it remains to check for a fixed $\sigma>0$: 
\begin{align*}
&E^{in}_\e(u_\e)-E^{in}_\e(\hat u_\e)\\
&\geq \int_{\{x\in S_\sigma\, :\, |u_\e|_g(x)>1\}} \frac 12
\left| \frac { j(u_\e)}{|u_\e|_g } - j^* \right|_g^2 + e_\e^{in}(|u_\e|_g)- \frac 12
\left| j(\hat u_\e) - j^* \right|_g^2 \, \vol_g+o(1), \, \textrm{ as } \e\to 0,
\end{align*}
where we denoted by $S_\sigma=S\setminus \cup_{k=1}^n B_\sigma(a_k)$ and $j^*=j(u^*(a,d,\Phi))$. Using \eqref{ilb.1} and $j(u_\e)=|u_\e|_g^2j(\hat u_\e)$ in $\{|u_\e|_g>1\}$, the above inequality will follow from
$$\int_{\{x\in S_\sigma\, :\, |u_\e|_g(x)>1\}} (|u_\e|_g-1)\, (j(\hat u_\e), j^*)_g\, \vol_g=o(1)  \, \textrm{ as } \e\to 0.$$
To prove this, we use $|j(\hat u_\e)|_g=|D \hat u_\e|_g\leq |D u_\e|_g$ in $\{|u_\e|_g>1\}$ so that we obtain by \eqref{F.growth} and the Cauchy-Schwarz inequality:
\begin{align*}
&\big|\int_{\{x\in S_\sigma\, :\, |u_\e|_g(x)>1\}} (|u_\e|_g-1)\, (j(\hat u_\e), j^*)_g\, \vol_g\big|\\
&\leq \|j^*\|_{L^\infty(S_\sigma)} (\int_S F(|u_\e|^2_g)\vol_g)^{1/2} (\int_S |Du_\e|^2_g \vol_g)^{1/2}
=O(\frac{\e \logeps}{\sigma})=o(1).
\end{align*}
\end{proof}

\begin{proof}[Proof of Theorem \ref{intrinsic.gammalim}] 
It is a direct consequence of Corollary \ref{Cor.cgc1} and Proposition \ref{intrinsic.gammalim2}.
\end{proof}

\bigskip

\section{$\Gamma$-limit in the extrinsic case}
\label{sec:ext_gamma}

In this section we prove the counterpart of 
Proposition \ref{intrinsic.gammalim2} for the extrinsic energy $E^{ex}_\e$ in Problem 2. Here, the surface $S$ is isometrically embedded in $\R^3$.

\begin{theorem}\label{ext.gamma}

The following $\Gamma$-convergence result holds.
\begin{itemize}
\item[1)] (Compactness) Let $(m_\eps)_{\eps\downarrow 0}$ be a family of sections of $\calX^{1,2}(S)$ satisfying
$
E^{ex}_\eps(m_\eps)  \le T \pi |\log \eps| + C
$
for some integer $T> 0$ and a constant $C>0$. Then
there exists a sequence $\eps \downarrow 0$ such that for every $p\in [1,2)$,
\begin{equation}
\omega(m_\eps)  \longrightarrow 2\pi \sum_{k=1}^{n} d_k\delta_{a_k} \quad \textrm{in } \, W^{-1,p}, \, \,  \textrm{ as } \, \eps \to 0,
\label{convergence_ext}\end{equation}
where $\{a_k\}_{k=1}^n$ are distinct points in $S$ and $\{d_k\}_{k=1}^n$ are nonzero integers satisfying \eqref{necessary} and $\sum_{k=1}^n |d_k|\leq T$. Moreover, if $\sum_{k=1}^n |d_k|= T$, then
$n = T$ and 
$|d_k|=1$ for every $k=1, \dots, n$; in this case, for a further subsequence, there exists $\Phi\in \calL(a;d)$ such that $\Phi(m_\eps)$ defined in \eqref{Phiu.def} converges to 
$\Phi$ as $\eps\to 0$.\\

\item[2)] ($\Gamma$-liminf inequality) Assume that the sections $m_\eps\in \calX^{1,2}(S)$ satisfy \eqref{convergence_ext} for $n$ distinct points $\{a_k\}_{k=1}^n\in S^n$ and $|d_k|=1$, $k=1, \dots n$ that satisfy \eqref{necessary} and $\Phi(m_\e)\to \Phi\in \calL(a;d)$. Then
\begin{equation}\label{lb1_ext}
\liminf_{\e\to 0}
\left[
E_\e^{ex}(m_\e)- n (\pi \logeps  + \iota_F) 
\right]  \ \ge \  
W(a,d,\Phi) + \tilde W(a,d,\Phi) 
\end{equation}
for $u^* = u^*(a,d,\Phi)$, $a=(a_1, \dots, a_n)$, $d=(d_1, \dots, d_n)$ and $\tilde W(a,d,\Phi)$ defined in  \eqref{renorm_ext}. 
\\ 

\item[3)] ($\Gamma$-limsup inequality) For every $n$ distinct points $a_1,\ldots, a_n\in S$ and $d_1,\ldots , d_n\in \{\pm 1\}$ satisfying
\eqref{necessary} and every $\Phi\in \calL(a;d)$
there exists a sequence of smooth sections $m_\eps:S\to TS$ such that $|m_\eps|_g\leq 1$ in $S$, \eqref{convergence_ext} holds,  $\Phi(m_\e)\to \Phi$ and
$$
E^{ex}_\eps(m_\eps)- n \pi |\log \eps| \longrightarrow  W(a,d, \Phi)+\tilde W(a,d,\Phi)+n \iota_F \quad \textrm{ as } \, \eps \to 0.
$$
\end{itemize}

\end{theorem}

\subsection{Compactness}
Let us start by computing the extrinsic Dirichlet energy of a section $m$:

\begin{lemma}
\label{lem:ext_Dir}
If $m:S\to TS$ is a section of $\calX^{1,2}$ then
$$|\dbar m|_g^2=|D m|_g^2+|\calS(m)|_g^2 \quad \textrm{a.e. in } S,$$
where $\calS:TS\to TS$ is the shape operator defined in \eqref{shape}.
\end{lemma}
\begin{proof} 
Let $\{\tau_1, \tau_2=i\tau_1\}$ be a local moving frame on $S$, i.e., 
$$\tau_\ell\cdot \tau_k:=(\tau_\ell, \tau_k)_g=\delta_{\ell k}, \quad \ell, k=1,2.$$
We write
$$m=\sum_{k=1}^2 m^k \tau_k, \quad D_\ell m:=D_{\tau_\ell} m, \quad \dbar_\ell m:=\tau_\ell\cdot\dbar m, \, \ell=1,2.$$
Denoting $N$ the Gauss map at $S$, we decompose the extrinsic differential as follows: 
$$\dbar m=Dm+(\dbar m\cdot N)\otimes N, \quad \textrm{i.e.,}  \quad \dbar_\ell m=D_\ell m+(\dbar_\ell m\cdot N) N, \quad \ell=1,2.$$
Therefore,
\beq
\label{e_vs_e1}
|\dbar m|_g^2=\sum_{\ell=1}^2 |\dbar_\ell m|_g^2=\sum_{\ell=1}^2 \bigg(|D_\ell m|_g^2+(\dbar_\ell m\cdot N)^2\bigg).
\eeq
Recall the definition of the shape operator \eqref{shape} (in particular, $\calS(\tau_\ell)=-\dbar_\ell N$ for $\ell=1,2$). It is a standard fact that $\calS$ is a symmetric operator corresponding to the
the second fundamental form $H$ of $S$; in other words, we have in the frame $\{\tau_1, \tau_2\}$ that
\footnote{The symmetry of $H$ follows from
$H_{\ell k}=-\tau_k \cdot \dbar_\ell N=N\cdot (\dbar_\ell \tau_\beta-D_\ell \tau_\beta)$ as $\tau_k\cdot N=0$ and $\dbar_\ell \tau_\beta-D_\ell \tau_\beta=\dbar_\beta \tau_\ell-D_\beta \tau_\ell+\underbrace{[\bar \tau_\ell, \bar \tau_\beta]-[\tau_\ell, \tau_\beta]}_{=0}$ where $[\cdot, \cdot]$ represents the commutator in $\R^3$ for the metric $g$.}
$$H_{\ell k}=\tau_k \cdot \calS(\tau_\ell)=\tau_\ell \cdot \calS(\tau_k)=H_{k \ell}.$$
Therefore, as $m\cdot N=0$ on $S$, we compute for every $\ell=1,2$:
$$
\dbar_\ell m\cdot N=-m\cdot \dbar_\ell N=\sum_{k=1}^2 m^k \tau_k\cdot \calS(\tau_\ell)=\sum_{k=1}^2 m^k \tau_\ell \cdot \calS(\tau_k)=\tau_\ell \cdot \calS(m) 
$$
so that
$$
\sum_{\ell=1}^2 (\dbar_\ell m\cdot N)^2=|\calS(m)|^2_g=\sum_{1\leq l, k\leq 2} m^l m^k H^2_{l k}, \quad \textrm{where } H^2_{lk}=\sum_{\ell=1}^2 H_{\ell l} H_{\ell k}.$$
\end{proof}

\begin{proof}[Proof of Theorem \ref{ext.gamma} point 1)] By Lemma \ref{lem:ext_Dir}, we see that $E^{in}_\e(m_\e)\leq E^{ex}_\e(m_\e)$, so that we can apply 
Corollary \ref{Cor.cgc1} and Theorem \ref{intrinsic.gammalim} point 1) to reach the conclusion. Note that the lower bound of $E^{in}_\e(m_\e)$ in Corollary \ref{Cor.cgc1} holds also true for $E^{in}_\e(m_\e)$.
\end{proof}

\subsection{Upper bound}

In the following, we adapt the construction from the proof of Proposition \ref{intrinsic.gammalim2}, point 2) to the case of Problem 2.

\begin{proof}[Proof of Theorem \ref{ext.gamma}, point 3)]
Let $u^*=u^*(a,d, \Phi)$ be a canonical harmonic map and $\Theta$ be a minimizer in \eqref{renorm_ext} (such a minimizer exists by the direct method in calculus of variations). Then
$\Theta$ satisfies the associated Euler-Lagrange equation to \eqref{renorm_ext}: 
\beq
\label{PDET}
-\Delta \Theta+\frac12\left(\cos(2\Theta) (\calS(u^*),\calS(iu^*))_g+\sin(2\Theta) (|\calS(iu^*)|_g^2-|\calS(u^*)|_g^2) \right)=0\quad \textrm{in} \quad S.
\eeq
Therefore, $\Delta \Theta\in L^\infty$ so 
$\Theta\in C^1(S)$. 
Let $U_\e:=U_\e(a,d, \Phi)$ 
be the vector field constructed for the upper bound in Proposition \ref{intrinsic.gammalim2}. We set 
\beq
\label{numarul}
m_\eps:=e^{i\Theta}U_\e \quad \textrm{in} \quad S.
\eeq
 By Lemma \ref{lem:ext_Dir}, 
we have that $|\dbar m_\e|_g^2=|D m_\e|_g^2+|\calS(m_\e)|_g^2$. We compute the intrinsic part as follows:
$$D_\ell m_\e=e^{i\Theta} D_\ell U_\e+ie^{i\Theta} \partial_\ell \Theta U_\e$$
yielding
\begin{align*}
|D m_\e|_g^2&=|D U_\e|^2_g+ |U_\e|_g^2 |d\Theta|^2_g+2\sum_\ell \partial_\ell \Theta 
\big(D_\ell U_\e, i U_\e\big)_g \\
\textrm{so that} \quad \quad |\dbar m_\e|_g^2&=|D U_\e|^2_g+ |U_\e|_g^2 |d\Theta|^2_g+|\calS(m_\e)|_g^2+2 \big(j(U_\e), d\Theta\big)_g.
\end{align*}
Recall that $|U_\e|_g\leq 1$ in $S$ and $U_\e=u^*$ in $S_{\sqrt{\e}}=S\setminus \cup_k B_{\sqrt{\e}}(a_k)$. Since $|m_\e|_g=|U_\e|_g$, we deduce by \eqref{renorm_ext}:
\begin{align*}
E^{ex}_\e(m_\e)  &\leq  \int_S e_\e^{in}(U_\e) \vol_g+\tilde W(a,d, \Phi)+
\int_S  \big(j(U_\e), d\Theta\big)_g \vol_g+\frac12 \int_{S\setminus S_{\sqrt{\e}}} |\calS(m_\e)|_g^2 \vol_g.
\end{align*}
The desired upper bound follows by the upper bound of $E^{in}_\e(U_\e)$ in Proposition \ref{intrinsic.gammalim2} point 2), by noting that as $\eps\to 0$: 
$$\int_{S\setminus S_{\sqrt{\e}}} |\calS(m_\e)|_g^2 \vol_g=o(1)$$
(as $|m_\e|_g\leq 1$) and
\begin{align*}
\bigg|\int_S \big(j(U_\e), d\Theta\big)_g \vol_g\bigg|&\leq
\bigg|\int_S \big(j(u^*), d\Theta\big)_g \vol_g\bigg|+\bigg|\int_{S\setminus S_{\sqrt{\e}}}
\big(j(U_\e)-j(u^*), d\Theta\big)_g \vol_g\bigg|\\
&\leq \bigg|\int_S \big(\underbrace{d^*j(u^*)}_{=0}, \Theta\big)_g \vol_g\bigg|+\|d\Theta\|_{L^\infty}\int_{S\setminus S_{\sqrt{\e}}}
|j(U_\e)-j(u^*)|_g \vol_g\\
&=o(1)
\end{align*}
because $j(U_\e)-j(u^*)\to 0$ strongly in $L^p(S)$ for every $p\in [1,2)$ (see \eqref{eqj}). It remains to prove the convergence of the vorticity $\omega(m_\eps)$ and of the flux integrals $\Phi(m_\e)$ as $\eps\to 0$. For that, we use $j(m_\e)=d\Theta |U_\e|^2_g+j(U_\e)$; since $d\Theta (1-|U_\e|_g^2)\to 0$ in  $L^p(S)$ for every $p\in [1,2)$ (as $d\Theta\in L^\infty$ and $|U_\e|_g=1$ in $S_{\sqrt{\e}}$), we deduce by \eqref{eqj}:
$$d j(m_\e)=d\bigg(j(U_\e)+d\Theta-d\Theta (1-|U_\e|_g^2) \bigg)\to dj(u^*)$$
in $W^{-1,p}(S)$ for $\eps\to 0$ yielding \eqref{convergence_ext}.
Also, for every harmonic $1$-form $\eta$, integration by parts and \eqref{eqj} yield
\bea
\int_S (j(m_\e), \eta)_g \vol_g&=\int_S (j(U_\e), \eta)_g +(d\Theta, \eta)_g \vol_g-\int_{S\setminus S_{\sqrt{\e}}} (1-|U_\e|_g^2)(d\Theta, \eta)_g  \vol_g\\
&=\int_S (j(u^*), \eta)_g \vol_g+o(1)
\end{align*}
because $d\Theta\in L^\infty$; in particular, $\Phi(m_\e)\to \Phi(u^*)$  for $\eps\to 0$. The smoothing argument follows as in Step 1 in the proof of Proposition \ref{P.comp}.
\end{proof}

\bigskip

\subsection{Lower bound}
\label{sec:lb_ex}
Following the proof in the intrinsic case in  Proposition \ref{intrinsic.gammalim2}, point 1), we start with a 
sharp lower bound away from the vortices, parallel
to Lemma \ref{L.qgamma} above. Let $u_\eps$ satisfy the assumption in Theorem \ref{ext.gamma} point 2). 
As $E_\e^{ex}(m_\eps)\geq E_\e^{in}(m_\eps)$, by the proof of Proposition \ref{intrinsic.gammalim2}, point 1), there exist 
$\beta = \beta(n)>0$ and
\begin{itemize}
\item 
 $n$ distinct points $a^\e=(a_{1, \eps}, \ldots, a_{n, \e})$ and  integers $d^\e=(d_{1, \eps}, \ldots, d_{n, \e})$ such that \eqref{farrr} is satisfied for the vorticity $\omega(m_\eps)$ for some 
$r_\e\in (\e^{\frac 1{2(n+1)}}, \e^\beta)$, and with $|d_{k, \e}|=1$ for all $k$;
\item
$\Phi^\e=(\Phi_{k,\eps})_{k=1}^{2\mathfrak g} \in \calL(a^\e, d^\e)$ such that $| \Phi(m_\e)  - \Phi^\e| \leq C\sqrt{r_\eps}$.
\end{itemize}
Moreover, $\dist_S(a_{k, \e}, a_{\ell, \e})\geq C>0$ for every $k\neq \ell$, $a_{k, \e}\to a_k$ as $\e\to 0$ and $|\Phi^\e|\le C$
(because $ \Phi(m_\e)\to \Phi$ by hypothesis).
In addition, $d_{k,\e}\to d_k$ as $\e\to 0$, and thus in fact $d_{k,\e} = d_k$
for all $k$, when $\e$ is small enough.
Finally, from Lemma \ref{L.near} we have 
\beq	
\int_{B_{{r_\e}}(a_{k, \e})} e_\e^{ex}(m_\e)
 \vol_g \ge
\int_{B_{{r_\e}}(a_{k, \e})} e_\e^{in}(m_\e)
 \vol_g \ge
(\pi \log \frac{{r_\e}}{\e}- C ), \quad k=1, \dots, n.
\label{eq:repair1}\eeq
Also, we may assume that 
\beq
\label{nu}
E_\e^{ex}(m_\e)\le n(\pi \logeps+C)\eeq 
for a constant $C=C(a,d,\Phi)>0$,
since otherwise \eqref{lb1_ext} is obvious.

\begin{lemma}\label{L.qgamma.ex}
Under the above hypotheses,
\begin{align}
\int_{S_{r_\e}} e_\e^{ex}(m_\e) \vol_g 
&\ge
 \pi n \log \frac 1 {r_\e} + (W+\tilde W)(a^\e,d^\e,\Phi^\e)
-o(1) \quad \textrm{as } \, \eps\to 0,
\label{qgamma.ex}
\end{align}
for $S_{r_\e} := S\setminus \cup_k B_{r_\e}(a_{k, \e})$.
\end{lemma}

\begin{proof}
For the proof it is useful to define\footnote{\label{foot12} Recall  from Theorem \ref{P1} that $u^*(a,d,\Phi)$ is unique only up to a rotation. For purposes of this definition, we assume that a representative $u^*$ has been (arbitrarily) fixed. In the end we are only interested in $\inf \mathcal I_\e$, and this is independent of the chosen rotation. We will therefore
feel free to adjust the rotations as needed.} a functional  $\mathcal I_\e[\, \cdot \,; a,d,\Phi]$  on $H^1(S;\C)$:
\begin{equation}\label{I.def}
\mathcal I_\e[w ; a,d, \Phi ] = \int_S \frac 12 |dw|_g^2 +\frac 12 |\mathcal S(w u^*(a,d,\Phi))|_g^2 + \frac 1{{4\e^2}} F(|w|_g^2) \ \vol_g . 
\end{equation}
We will also write $u^*_\e := u^*(a^\e,d^\e, \Phi^\e)$ and $j_\e^*=j(u_\e^*)$.

\medskip

\nd {\it Step 1}. It follows from \eqref{ilb.1} and \eqref{ilb.2} that 
\begin{equation}
\int_{S_{r_\e}} e^{in}_\e(m_\e) \vol_g
= \int_{S_{r_\e}} \frac 12|j^*_\e|_g^2  + \frac 12 \left| \frac {j(m_\e)}{|m_\e|_g} - j^*_\e\right|_g^2
+ e^{in}_\e(|m_\e|_g) \vol_g  + o(1) . 
\label{elb.1}\end{equation}
For every $x\in S\setminus\{a^\e\}$, 
since $\{ u_\e^*(x), iu_\e^*(x)\}$ 
is a basis for $T_xS$,
there is a $w_\e(x)\in \C$ such that
\[
m_\e = w_\e u^*_\e.
\]
If $m_\e\in \calX^{1,2}(S)$, 
then it is clear that  the function $w_\e:S\to \C$
defined in this way belongs to $H^1_{loc}$ away from $\{a^\e\}$,
and Section \ref{sec:calc} shows that
\[ 
d|w_\e| = d|m_\e|_g, \qquad 
\frac{(iw_\e, dw_\e)}{|w_\e|}  = \frac {j(m_\e)}{|m_\e|_g}  - |m_\e|_g j_\e^*,
\qquad 
|dw_\e|_g^2 = |d|w_\e||_{g}^2 + |\frac{(iw_\e, dw_\e)}{|w_\e|}|_{g}^2,
\] 
where here $( \, \cdot \, , \, \cdot \, )$ denotes the real inner product on $\C$, defined by  $(v,w) := \frac 12(v\bar w + w \bar v)$.
From Lemma \ref{lem:ext_Dir} we also know that
$e_\e^{ex} (m_\e) = e_\e^{in}(m_\e) + \frac 12 |\calS(m_\e)|_g^2$.
It follows that 
\begin{align*}
\int_{S_{r_\e}} e_\e^{ex}(m_\e) \vol_g
&=
\int_{S_{r_\e}} \frac 12 |j^*_\e|_g^2  \ \vol_g +
\int_{S_{r_\e}} \frac 12|dw_\e|_g^2
 +\frac 12 |\mathcal S(w_\e u^*_\e)|_g^2 + \frac 1{{4\e^2}} F(|w_\e|^2) \ \vol_g 
\\
&\qquad +
\frac 1 2 \int_{S_{r_\e}}\left| \frac {j(m_\e)}{|m_\e|_g} - j^*_\e\right|_g^2 -  \left| \frac {j(m_\e)}{|m_\e|_g} -|m_\e|_g j^*_\e\right|_g^2  \vol_g + o(1).
\end{align*}
Clearly
\[
\left| \frac {j(m_\e)}{|m_\e|_g} - j^*_\e\right|_g^2 -  \left| \frac {j(m_\e)}{|m_\e|_g} - |m_\e|_gj^*_\e\right|_g^2 \\
= 2 \frac {j(m_\e)}{|m_\e|_g} \cdot j^*_\e (|m_\e|_g-1)+ |j^*_\e|^2(1 - |m_\e|_g^2)  .
\]
Since $j^*_\e$ is smooth away from $\{a^\e\}$
and blows up like $\dist_S(\cdot, a_{k, \e})^{-1}$ near each
$a_{k, \e}$, see \eqref{jstar.log}, it is clear that $|j^*_\e|_g \le C r_\e^{-1} \le C \e^{-\frac 1{2(n+1)}}$
on $S_{r_\e}$.
Recalling that $ \frac {|j(m_\e)|_g}{|m_\e|_g} \le |Dm_\e|_g$,  
straightforward estimates then show that
\begin{equation}\label{wvsj}
\left|
 \int_{S_{r_\e}}\left| \frac {j(m_\e)}{|m_\e|_g} - j^*_\e\right|_g^2 -  \left| \frac {j(m_\e)}{|m_\e|_g} -|m_\e|_g j^*_\e\right|_g^2  \vol_g \right| \le C \e r_\e^{-1} E_\e^{in}(m_\e) + C r_\e^{-2} \e \sqrt{E_\e^{in}(m_\e)} = o(1)
\end{equation}
as $\e\to 0$.
We infer that
\begin{equation}
\int_{S_{r_\e}} e_\e^{ex}(m_\e) \vol_g
=
\int_{S_{r_\e}} \frac 12 |j^*_\e|_g^2  \ \vol_g +
\int_{S_{r_\e}} \frac 12|dw_\e|_g^2
 +\frac 12 |\mathcal S(w_\e u^*_\e)|_g^2 + \frac 1{{4\e^2}} F(|w_\e|^2) \ \vol_g 
 +  o(1) .
\label{elb.3}\end{equation}

\medskip

\nd {\it Step 2}. 
It follows from \eqref{eq:repair1} and \eqref{elb.1} that 
\begin{align*}
&\int_{S_{r_\e}} \frac 12|dw_\e|_g^2
 +\frac 12 |\mathcal S(w_\e u^*_\e)|_g^2 + \frac 1{{4\e^2}} F(|w_\e|^2) \ \vol_g 
 \\
&\qquad\qquad\qquad\qquad\le E_\e^{ex}(m_\e) - 
\int_{S_{r_\e}} \frac 12 |j^*_\e|_g^2  \ \vol_g 
- n\pi \log\frac{r_\e}{\e} + C.
\end{align*}
We recall also that, in view of \eqref{formula12}, the fact that $u^*_\e$ is a {\it unit} vector field implies that $|j^*_\e|_g = |Du^*_\e|$, and hence (from the definition of the intrinsic renormalized energy) that
\beq
\int_{S_{r_\e}} \frac 12 |j^*_\e|_g^2 \, \vol_g
= W(a^\e, d^\e,\Phi^\e) + n\pi \log \frac 1 {r_\e} + o(1).
\label{repair2}\eeq
Combining the above estimates with \eqref{nu}, we deduce that
\begin{equation}
\int_{S_{r_\e}} \frac 12|dw_\e|_g^2
 +\frac 12 |\mathcal S(w_\e u^*_\e)|_g^2 + \frac 1{{4\e^2}} F(|w_\e|^2) \ \vol_g \le C
\label{calI.bd}\end{equation}
for some constant independent of $\e$.

\nd {\it Step 3}. We next claim that exists $\tilde w_\e\in H^1(S;\C)$ such that $\tilde w_\e = w_\e$
on $S_{\sqrt{r_\e}}$ and 
\begin{equation}\label{extend.w}
\mathcal I_{\e}[ \tilde w_\e ; a^\e,d^\e, \Phi^\e] \le 
\int_{S_{r_\e}} \frac 12|dw_\e|_g^2
 +\frac 12 |\mathcal S(w_\e u^*_\e)|_g^2 + \frac 1{{4\e^2}} F(|w_\e|^2) \ \vol_g 
+ o(1)
\end{equation}
as $\e\to 0$, where $\mathcal I_\e$ was defined at \eqref{I.def}.

First, fix some $k\in \{1, \dots, n\}$ and consider exponential normal coordinates $\Psi$
at  $a_{k, \e}$, mapping a Euclidean ball $\{ y\in \R^2 :|y|< \sigma \}$
onto the geodesic ball $B_\sigma(a_{k, \e})$ in $S$,
see Section \ref{sect.coords}. For $y\in \R^2$ such that
$r_\e \le |y|\le \sqrt{r_\e}$, let
$v_\e(y) := w_\e(\Psi(y))$.
We may then rewrite the energy of $w_\e$ in this annulus
in terms of $v_\e$. Using \eqref{normalcs} to approximate the metric
$g$ by the Euclidean metric,  we find that 
\[
\int_{B_{\sqrt{r_\e}}(a_{k, \e})\setminus B_{r_\e}(a_{k, \e})}
\frac 12 |dw_\e|_g^2 + \frac 1{{4\e^2}}F(|w_\e|^2) \vol_g = (1+O(r_\e))
\int_{\{y\in \R^2 : r_\e < |y| < \sqrt{r_\e}\}} e_\e(v_\e)(y) dy
\]
where $e_\e(v_\e)$ denotes the Euclidean Ginzburg-Landau energy
and $dy$ is the Euclidean area element.
By arguing as in Step 3 of the proof of 
Proposition \ref{intrinsic.gammalim2}, see  \eqref{really},
we may find some ${\tilde t}_\e \in (r_\e, \sqrt{r_\e})$
such that 
\begin{equation}\label{tep.est}
{\tilde t}_\e  \int_{\{y\in \R^2 : |y|={\tilde t}_\e\}} 
e_\e(v_\e)(y) d\calH^1(y)
\le C\logeps^{-1} . 
\end{equation}
In particular, writing $\partial_\tau$ for the tangential derivative, it follows 
from Cauchy-Schwarz that
\[
\int_{\{y\in \R^2 : |y|={\tilde t}_\e\}} |\partial_\tau v_\e| d\calH^1 \ \le \ C \logeps^{-1/2}.
\]
Hence the total variation of $v_\e$ on $\{y\in \R^2 : |y| = {\tilde t}_\e\}$
is bounded by $C\logeps^{-1/2}$. Since ${\tilde t}_\e\geq \eps^{1/2(n+1)}$ and ${\tilde t}_\e \int_{ \{ |y|= {\tilde t}_\e\}} F(|v_\e|^2) \, d\h^1\le
C \e^2\logeps^{-1}$, it follows that there is some constant $v_0$ of unit modulus such
that $|v_\e(y) - v_0|\leq C\logeps^{-1/2}$ whenever $|y|={\tilde t}_\e$.
We may thus write
\[
v_\e(y) = \rho(y) e^{i\eta(y)} \qquad\mbox{ for }|y|= {\tilde t}_\e,
\]
where $\eta$ is  $H^1$ and real-valued. In particular, $\rho\leq 2$ on $\{|y|= {\tilde t}_\e\}$.
Let $\bar \eta$ denote the mean
of $\eta$ on $\{ y\in \R^2 : |y|={\tilde t}_\e\}$, and define a complex-valued function
$\tilde v_\e$  on $\{|y| < \tilde t_\e\}$ by 
\begin{align*}
\tilde v_\e(s y) &:= [1+ s(\rho(y)-1)] \exp[i (\bar \eta+s(\eta(y) - \bar \eta))] \qquad&\mbox{ for }|y|={\tilde t}_\e, 0\le s \le 1,\qquad
\\
\end{align*}
Then one can check from \eqref{tep.est} that 
\beq
\label{int_123}
\int_{\{y\in \R^2 : |y|  < {\tilde t}_\e\}} e_\e(\tilde v_\e)(y) dy \le C \logeps^{-1}.
\eeq
We next define $\tilde w_\e(x) = \tilde v_\e(\Psi^{-1}(x))$ in $B_{\tilde t_\e}(a_{k, \e})\subset S$. 
We remark that since the area of $B_{{\tilde t}_\e}(a_{k, \e})$
is bounded by $C {\tilde t}_\e^2$ and $|\tilde w_\e|_g\leq 2$ in $B_{{\tilde t}_\e}(a_{k, \e})$, it is clear that
\[
\int_{B_{{\tilde t}_\e}(a_{k, \e})} |\mathcal S(\tilde w_\e u^*_\e)|_g^2\vol_g 
\le 
C \int_{B_{{\tilde t}_\e}(a_{k, \e})} |\tilde w_\e|^2 \vol_g   = o(1) \qquad \mbox{ as }\e\to 0.
\]
Now we proceed in the same fashion for every $k\in \{1,\ldots, n\}$,
and we set  $\tilde w_\e = w_\e$ on $S_{\tilde t_\e}$ (in particular, $\tilde w_\e = w_\e$
on $S_{\sqrt{r_\e}}$).
This yields a function $\tilde w_\e\in H^1(S;\C)$.
Again using \eqref{normalcs}, which implies that the difference
between the metric $g$ (appearing in $e_\e^{ex}$) and
the Euclidean metric (appearing in $e_\e$) is negligible in small balls (in particular, inside $B_{{\tilde t}_\e}(a_{k, \e})$) for our choice of coordinates, we readily verify that  $\tilde w_\e$ 
satisfies \eqref{extend.w}, proving the claim.

\medskip

\nd 
{\it Step 4}.
We introduce\footnote{The footnote \ref{foot12} applies here as well.}  the functional $\mathcal I_0[\, \cdot \,; a,d,\Phi]$ for $v\in H^1(S;\SSS^1)$:
$$
\mathcal I_0[v ; a,d, \Phi ] = \int_S \frac 12 |dv|_g^2 +\frac 12 |\mathcal S(v u^*(a,d,\Phi))|_g^2 \ \vol_g .
$$
The definition \eqref{renorm_ext} then implies that $\tilde W(a,d,\Phi) = \inf_{\Theta\in H^1(S;\R)}\mathcal I_0[e^{i\Theta}; a, d, \Phi]$.
 In view of \eqref{elb.3}, \eqref{repair2}, and \eqref{extend.w}, to conclude to our desired estimate \eqref{qgamma.ex}, it suffices to prove
that
\beq
\liminf_{\e\to 0} 
\mathcal I_{\e}[ \tilde w_\e ; a^\e,d^\e, \Phi^\e]  
\ge
\tilde{W}(a,d,\Phi)
\label{remain1}\eeq
and 
\beq
\limsup_{\e\to 0} \tilde{W}(a^\e,d^\e,\Phi^\e)
\le
\tilde{W}(a,d,\Phi).
\label{remain2}\eeq

\medskip

\nd 
{\it Step 4'}. We prove the second assertion \eqref{remain2}. Toward this goal, we claim that after possible $\e$-dependent rotations of  $u^*_\e$,
we have
\beq
u^*_\e\to u^*
\quad \textrm{ a.e. in } S.  
\label{remain3}\eeq
First note that it is clear from the definition \eqref{psi.def} of $\psi(a;d)$ and \eqref{newpsi} that
\[
j^*_\e :=
 d^*\psi(a_\e, d_\e)+ \sum_{k=1}^{2\mathfrak g} \Phi_{k,\e}\eta_k
\to
 d^*\psi(a, d)+ \sum_{k=1}^{2\mathfrak g} \Phi_{k}\eta_k
=j^*
\]
in $C^1_{loc}$ away from $a_1,\ldots, a_n$ and globally in $L^q(S)$ for every $q\in [1,2)$.
Next, fix $x\in S\setminus \{a_k\}_k$ and a unit  vector $v\in T_xS$.
We may assume that $u_\e^*(x) = v$ for every $\e$.
Now consider $y\in S\setminus \{a_k\}_k$ and a
smooth curve $\gamma:[0,1]\to S\setminus\{a_k\}_k$ 
such that $\gamma(0)=x, \gamma(1)=y$.
For  $\e$ sufficiently small, the image of $\gamma$
is bounded away from $\{ a_{k,\e}\}_k$.
When this holds, we define
$U_\e(s) := u^*_\e(\gamma(s))$,
and similarly $U(s) = u^*(\gamma(s))$.
In \eqref{holonomy0}, we have derived an explicit formula
that gives $U_\e(s)$ in terms of $v\in T_xS$
and $j^*_\e$ (or $U(s)$ in 
terms of $v$ and $j^*$),
and with the convergence of $j^*_\e$ to $j^*$,
this formula immediately implies that
\[
u^*_\e(y) = U_\e(1) \to U(1) = u^*(y) \mbox{ as }\e\to 0.
\]
Since $y$ was an arbitrary point in $S\setminus\{a_k\}_k$, this proves the claim \eqref{remain3}.
\footnote{\label{footin}This argument proves the following continuity result (in addition to Theorem \ref{P1}): 
if $a_\e\to a$ and $\Phi_\e\in \calL(a_\e, d)\to
\Phi \in \calL(a,d)$, then up to rotations, 
$u_\e^*(a_\e, d, \Phi_\e)
\to u^*(a, d, \Phi)$ almost everywhere and in $L^p$ for all $p<\infty$. 
}

Now the direct method leads to the existence of 
$\Theta_0\in H^1(S;\R)$ minimizing $\mathcal I_0[ e^{i ( \cdot)} ; a,d, \Phi]$.
The continuity of the shape operator and the convergence $u^*_\e\to u^*$ {\it a.e.}
imply that $|\calS(e^{i\Theta_0}u^*_\e)|_g^2\to |\calS(e^{i\Theta_0}u^*)|_g^2$
almost everywhere and hence in $L^p$ for every $p<\infty$. It follows that
\[
\limsup_\e \tilde W(a^\e, d^\e, \Phi^\e) \le \lim_\e
\mathcal I_0[e^{i\Theta_0}; a^\e, d^\e, \Phi^\e] =
\mathcal I_0[e^{i\Theta_0}; a, d, \Phi] = \tilde W(a,d,\Phi),
\]
proving \eqref{remain2}.

\medskip

\nd
{\it Step 4''.}
We prove \eqref{remain1}. First note from \eqref{extend.w}, \eqref{calI.bd}
that $\|\tilde w_\e \|_{H^1}^2  \le 2 \mathcal I_\e[\tilde w_\e; a^\e,d^\e, \Phi^\e]\le C$.
We may thus assume, after
passing to a subsequence, that $\tilde w_\e \rightharpoonup w_0$ weakly in $H^1(S;\C)$
and thus a.e. in $S$ and strongly in $L^p$ for every $p<\infty$.
By Fatou's lemma,
$$\int_S F(|w_0|^2) \, \vol_g\leq \liminf_{\e\to 0} \int_S F(|\tilde w_\e|^2) \, \vol_g\leq \liminf_{\e\to 0} 4\eps^2 {\mathcal I_\e}[\tilde w_\e; a^\e, d^\e, \Phi^\e]=0,$$
so we deduce that $|w_0|=1$ a.e.
Standard weak lower semicontinuity arguments together with
\eqref{remain3} and
the continuity of the shape operator imply that
\[
\mathcal I_0[ w_0; a,d,\Phi] \le \liminf_{\e \to 0}{\mathcal I_\e}[\tilde w_\e; a^\e,d^\e, \Phi^\e]. 
\]
To complete the proof of \eqref{remain1}, it thus suffices to show that
$w_0$ admits a lifting, that is, that there exists some $\Theta_0\in H^1(S;\R)$ such that
$w_0 = e^{i\Theta_0}$. Note that this is a delicate issue as $S$ is not simply connected while standard results (see e.g., \cite{Bet_Zhe}) requires this topological contraint on $S$. We will show in Lemma \ref{lem:lifting} that $w_0$ has indeed an $H^1$ lifting provided that $w_0$ satisfies the constraint $\Phi(w_0 u^*) = \Phi(u^*)=\Phi$.
Toward this end, as in Step 1, we note that $j(w_0u^*)= (iw_0,dw_0) + j^*$.
Thus for $k\in \{1,\ldots, 2\mathfrak g \}$, we have
\begin{align*}
\Phi_k(w_0u^*)
=
\int_S ( (iw_0, dw_0) + j^*, \eta_k)_g \, \vol_g  
=
\int_S  ( (iw_0, dw_0) , \eta_k)_g \, \vol_g   + \Phi_k.
\end{align*}
Similarly, 
\begin{align*}
\Phi_k(\tilde w_\e u_\e^*)
=
\int_S  ( (i\tilde w_\e, d\tilde w_\e) , \eta_k)_g \, \vol_g   +\int_S  (|\tilde w_\e|^2-1) j^*_\e ,  \eta_k)_g \, \vol_g   + \Phi_{k, \e}.
\end{align*}
From these and the convergence $\tilde w_\e \rightharpoonup w_0$ weakly in $H^1(S;\C)$, $\tilde w_\e\to w_0$ strongly in $L^2(S)$ and $L^6(S)$ (in particular, $|\tilde w_\e|^2\to 1$ in $L^3(S)$), $j^*_\e\to j^*$ in $L^q(S)$ for $q=\frac32<2$,
and recalling that $\Phi^\e\to \Phi$, one can verify that 
$\Phi(w_0u^*) = \lim_{\e\to 0} \Phi(\tilde w_\e u_\e^*)$.

Next, recall that by construction in Step 3, $\tilde w_\e u_\e^* =  w_\e u_\e^* = m_\e$ in $S_{\tilde t_\e}$ and $|\tilde w_\e|\leq 2$ in $\calO:=\cup_{k=1}^N B_{\tilde t_\e} (a_{ k,\e})$. Therefore, $|j(\tilde w_\e u_\e^*)|_g\leq 4 |j^*_\e|_g+2|d\tilde w_\e |_g$ in $\calO$.
Thus
\begin{align*}
&|\Phi_k(\tilde w_\e u_\e^*) - \Phi_k(m_\e) |
\le C  \int_{\calO} | j(\tilde w_\e u_\e^*) |_g + |j(m_\e)|_g\, \vol_g \\
&\le C  \bigg(\int_{\calO} |j^*_\e|_g +|d\tilde w_\e |_g \vol_g+\int_{\calO\cap \{|m_\e|_g\leq 2\}}|Dm_\e |_g\vol_g+\int_{\calO\cap \{|m_\e|_g\geq 2\}}(|m_\e|_g-1)|Dm_\e |_g \vol_g\bigg)\\
&\leq C \bigg(\tilde t_\e^{1-1/q} \|j^*_\e\|_{L^q(S)}+\tilde t_\e^{1/2} \|d\tilde w_\e\|_{L^2(\mathcal O)}+ \tilde t_\e^{1/2} E^{in}_\e(m_\e)^{1/2}+\e E^{in}_\e(m_\e)\bigg) \to 0, 
\end{align*}
as $\e\to 0$, where we used H\"older's inequality, that $(j_\e^*)$ is uniformly bounded in $L^q(S)$ for $q=3/2$, \eqref{int_123}, \eqref{nu} and \eqref{F.growth}. 
Since $\Phi(m_\e)\to \Phi$ by assumption, 
we deduce that $\Phi(w_0u^*)=\Phi(u^*)$ as claimed. 

Now Lemma \ref{lem:lifting} below implies that $w$ admits a lifting,
completing the proof of \eqref{remain1} and hence of 
Lemma \ref{L.qgamma.ex}.
\end{proof}

\begin{lemma}\label{lem:lifting}
Assume that $w\in H^1(S; \SSS^1)$ and that $\Phi(w u^*) = \Phi(u^*)$
for some canonical harmonic unit vector field $u^*(a,d,\Phi)$.
Then there exists $\Theta \in H^1(S;\R)$ such that $w = e^{i\Theta}$.
\end{lemma}

\begin{proof}
It follows from  \cite{Bet_Zhe} that  any  $w\in H^1(S;\SSS^1)$ can {\em locally}
be written in the form $w = e^{i\theta}$. It follows that, again locally, $j(w)$ 
has the form $j(w) = d\theta$. Thus $d j(w)=0$. 
As a result, $\int_S ( j(w), d^*\beta)_g \vol_g = 0$ for all $2$-forms $\beta$ in $H^1(S)$.
This implies that the Hodge decomposition \eqref{Hodge} for
$j(w) $ takes the form
\[
j(w) = (iw, dw) = d\Theta +\eta \qquad\mbox{ where }\Theta\in H^1(S;\R) \mbox{ and $\eta$ is a harmonic $1$-form.}
\]
Next, by hypothesis, we have 
\bea
\Phi_k(u^*)=\Phi_k(w u^*) &= \int_S ( j(w u^*), \eta_k)_g \vol_g\\ 
&
= 
\int_S (j(w)+ j^* , \eta_k)_g \vol_g = \int_S (j(w) , \eta_k)_g \vol_g + \Phi_k(u^*) .
\end{align*}
yielding $\int_S (j(w), \eta_k)_g \vol_g= 0$ for every $k=1,\ldots, 2\mathfrak g$. Thus the harmonic part of $j(w)$ in the decomposition  \eqref{Hodge} vanishes, and  $j(w) = d\Theta$ for some $\Theta\in H^1(S;\R)$, or equivalently, $j(w e^{-i\Theta}) = 0$. 
Writing $v = w e^{-i\Theta}\in H^1(S;\SSS^1)$, we deduce that
\[
dv = (dv, \frac {iv}{|v|})\frac {iv}{|v|} + (dv, \frac {v}{|v|})\frac {v}{|v|} =
j(v) \frac {iv}{|v|} + d|v| \, \frac {v}{|v|}  = 0 . 
\]
It follows that $v$ is constant, from which we conclude that $w = e^{i(\Theta+\alpha)}$ for some $\alpha\in \R$.
\end{proof}

\begin{remark}
\label{lem:add}
Note that $$\tilde W(a^\e, d^\e, \Phi^\e)\to \tilde W(a, d, \Phi) \quad \textrm{if $a^\e\to a$, $d^\e\to d$ and $\Phi^\e\in \calL(a^\e, d^\e)\to \Phi\in \calL(a,d)$}.$$ In fact it is a consequence of \eqref{remain2} and $\liminf_{\e\to 0} \tilde W(a^\e, d^\e, \Phi^\e)\geq \tilde W(a, d, \Phi)$ which is follows the argument in Step 4'' above. Indeed, if we denote by $\Theta_\e$ a minimizer of $\tilde W(a^\e, d^\e, \Phi^\e)$ (such a minimizer exists as a consequence of the direct method in calculus of variations), we have that  
$\|d\Theta_\e\|^2_{L^2(S)}\leq 2 \tilde W(a^\e, d^\e, \Phi^\e)\leq \int_S |\calS(e^{i\pi}u^*(a^\e, d^\e, \Phi^\e))|^2_g \vol_g\leq C$ (because $\calS$ is bounded over the set of unit vector fields). Therefore, up to an additive constant, the Poincar\'e-Wirtinger inequality implies that $(\Theta_\e)$ is uniformly bounded in $H^1(S)$. Therefore, for a subsequence, there exists a limit $\Theta\in H^1(S)$ such that $\Theta_\e\rightharpoonup \Theta$ weakly in $H^1(S)$ and a.e. in $S$. As $u^*(a^\e, d^\e, \Phi^\e)\to u^*(a, d, \Phi)$ a.e. in $S$ (by footnote~\ref{footin}, recall that $d^\e=d$ for small $\e$), standard weak lower semicontinuity arguments, the continuity of the shape operator and Fatou's lemma imply $$\liminf_{\e\to 0} \tilde W(a^\e, d^\e, \Phi^\e)\geq \frac12 \int_S |d\Theta|^2_g+ 
|\calS(e^{i\Theta}u^*(a, d, \Phi))|^2_g \vol_g\geq \tilde W(a, d, \Phi).$$
\end{remark}

\begin{proof}[Proof of Theorem \ref{ext.gamma}, point 2)] We may assume the hypothesis on $(a^\e, d^\e, \Phi^\e)$ made at the beginning of Section \ref{sec:lb_ex}. Then we argue as in the proof of Proposition \ref{intrinsic.gammalim2} point 1). As $e^{ex}_\e(m_\eps)\geq e^{in}_\e(m_\eps)$ in $S$ (by Lemma \ref{lem:ext_Dir}, by Step 3 in the proof of Proposition \ref{intrinsic.gammalim2} point 1) ), there exists $t_\e\in (r_\e, \sqrt r_\e)$ such that
\[
\int_{B_{t_\e}(a_{k, \e})}e_\e^{ex}(m_\e) \vol_g \ge\int_{B_{t_\e}(a_{k, \e})}e_\e^{in}(m_\e) \vol_g \ge \pi \log \frac{t_\e}{\e} +\iota_F   + o(1)\] 
as $\e\to 0$, for every $1\leq k\leq n$. As $t_\eps\in (r_\e, \sqrt{r_\e})$ and $\dist_S(a_{k,\e}, a_{\ell, \e})=O(1)\geq t_\e$ for $k\neq \ell$, then  \eqref{farrr} holds true for $t_\e$ so that we can apply 
Lemma \ref{L.qgamma.ex} for $t_\e$ yielding:
$$
\int_{S_{t_\e}} e_\e^{ex}(m_\e) \vol_g 
\ge n\pi \log \frac 1{t_\e} + (W+\tilde W)(a^\e, d^\e, \Phi^\e)-  o(1)
$$
as $\e \to 0$. As $W(a^\e, d^\e, \Phi^\e)\to W(a, d, \Phi)$ (by Proposition \ref{prop.W}, as $\Phi_\e \to \Phi$ and  $d_\e=d$ for all small $\e$) and  $\tilde W(a^\e, d^\e, \Phi^\e)\to \tilde W(a, d, \Phi)$ (by Remark \ref{lem:add})  in the limit $\e\to 0$, we reach the desired lower bound for 
$E^{ex}_\e(m_\eps)$.
\end{proof}

\bigskip

\section{$\Gamma$-limit for the micromagnetic energy}\label{sec:11}

Before stating the main result for Problem 3, let us show that the quantity $\tilde \iota_F$ in \eqref{iota_til} is well defined, i.e., the limit in \eqref{iota_til} exists. For that, it is enough to prove the nondecreasing behaviour of $t\mapsto I^{mm}_F(t)+\pi \log t$ that follows as for $\iota_F$ (see \cite[Lemma III.1]{BBH}): for $0<t_1\leq t_2\leq 1$, we want $I^{mm}_F(t_1)\leq \pi \log \frac{t_2}{t_1}+I^{mm}_F(t_2)$. Indeed, if $v_2$ is the minimizer of $I^{mm}_F(\frac1{t_2},1)$, then setting $v_1=v_2$ in $B_{1/t_2}(0)$ and $v_1=\frac{x}{|x|}$ in 
$B_{1/t_1}(0)\setminus B_{1/t_2}(0)$ we have 
\beq
\label{aici}
I^{mm}_F(t_1)=I^{mm}_F(\frac1{t_1},1)\leq  \int_{B_{1/t_1}(0)} \tilde{e}_1(v_1)\, dy=I^{mm}_F(\frac1{t_2},1)+\pi \log\frac{t_2}{t_1}=I^{mm}_F(t_2)+\pi \log\frac{t_2}{t_1}.
\eeq

In this section we prove the counterpart of 
Theorem \ref{ext.gamma} for the micromagnetic energy $E^{mm}_\e$ in Problem 3 where the surface $S$ is isometrically embedded in $\R^3$ endowed with the Euclidian metric $g$. 
For $M:S\to \mathbb{S}^2$, we always use the decomposition  
\beq
\label{decompo}
M=m+\Mn N,\quad m=\Pi(M)\eeq
where $N$ is the Gauss map on $S$, $\Mn=M\cdot N$ is the normal component of $M$ and $m$ is the projection $\Pi$ of $M$ on the tangent plane $TS$. 

\begin{theorem}\label{mm.gamma}

The following $\Gamma$-convergence result holds.
\begin{itemize}
\item[1)] (Compactness) Let $(M_\eps)_{\eps\downarrow 0}$ be a family in $H^{1}(S; \SSS^2)$ satisfying
$$
E^{mm}_\eps(M_\eps)  \le T \pi |\log \eps| + C
$$
for some integer $T> 0$ and a constant $C>0$. Then
there exists a sequence $\eps \downarrow 0$ such that for every $p\in [1,2)$, the vorticity 
$\omega(m_\eps)$ of the projection $m_\eps=\Pi(M_\eps)$ satisfies  \eqref{convergence_ext}
for $n$ distinct points $\{a_k\}_{k=1}^n$ and nonzero integers $\{d_k\}_{k=1}^n$ satisfying \eqref{necessary} and $\sum_{k=1}^n |d_k|\leq T$. Moreover, if $\sum_{k=1}^n |d_k|= T$, then
$n = T$ and 
$|d_k|=1$ for every $k=1, \dots, n$; in this case, for a further subsequence, there exists $\Phi\in \calL(a;d)$ such that $\Phi(m_\eps)$ defined in \eqref{Phiu.def} converges to 
$\Phi$ as $\eps\to 0$.\\

\item[2)] ($\Gamma$-liminf inequality) Assume that the projections $m_\eps=\Pi(M_\e)\in \calX^{1,2}(S)$ of some family $M_\e:S\to \mathbb{S}^2$ satisfy \eqref{convergence_ext} for $n$ distinct points $\{a_k\}_{k=1}^n\in S^n$ and $|d_k|=1$, $k=1, \dots n$ that satisfy \eqref{necessary} and $\Phi(m_\e)\to \Phi\in \calL(a;d)$. Then
 for every $\sigma>0$, 
\bea
\liminf_{\e\to 0}
\left[
E_\e^{mm}(M_\e)- n \pi \logeps  \right]  & \ge \  
W(a,d,\Phi) + \tilde W(a,d,\Phi) +n \tilde \iota_F\\
& +  \liminf_{\e\to 0}\int_{S \setminus \cup_{k=1}^n B_\sigma(a_k)} |dM_{\perp, \e}|_g^2 \vol_g 
\end{align*}
for $u^* = u^*(a,d,\Phi)$, $a=(a_1, \dots, a_n)$, $d=(d_1, \dots, d_n)$, $\tilde W(a,d,\Phi)$ defined in  \eqref{renorm_ext}, $\tilde \iota_F$ is defined in \eqref{iota_til} and $M_{\perp, \e}$ is the normal component of $M_\e$.

\item[3)] ($\Gamma$-limsup inequality) For every $n$ distinct points $a_1,\ldots, a_n\in S$ and $d_1,\ldots , d_n\in \{\pm 1\}$ satisfying
\eqref{necessary} and every $\Phi\in \calL(a;d)$
there exists a sequence of smooth maps $M_\e:S\to \mathbb{S}^2$ such that  \eqref{convergence_ext} holds for the projections $m_\e=\Pi(M_\e)$ (see \eqref{decompo}),  $M_{\perp,\e}\to 0$ in $H^1_{loc}(S\setminus \{a_k\}_k)$, $\Phi(m_\e)\to \Phi$ and
$$
E^{mm}_\eps(M_\eps)- n \pi |\log \eps| \longrightarrow  W(a,d, \Phi)+\tilde W(a,d,\Phi)+n \tilde \iota_F \quad \textrm{ as } \, \eps \to 0.
$$
\end{itemize}

\end{theorem}

\begin{remark}
The term $\liminf_{\e\to 0}\int_{S \setminus \cup_{k=1}^n B_\sigma(a_k)} |dM_{\perp, \e}|_g^2 \vol_g $  in the above $\Gamma$-liminf inequality
will be used to show that if $M_\e$ minimizes $E_\e^{mm}$, then 
$M_{\perp,\e}\to 0$ in $H^1_{loc}(S\setminus \{a_k\}_k)$, where $\{a_k\}_k$
are limiting vortex locations.
\end{remark}

\begin{proof}We divide the proof in several steps. 

\medskip

\nd {\it Step 1. A basic computation}. Let $M:S\to \mathbb{S}^2$ such that $E^{mm}_\eps(M)  \le C |\log \eps|$. By \eqref{decompo},
we start by computing the extrinsic differential of $M$:
\bea
\dbar M&=\dbar m+ \dbar (\Mn N)\\
&=Dm+(\dbar m\cdot N)\otimes N+\Mn \dbar N+d\Mn\otimes N\\
&=[Dm+\Mn \dbar N]+[(\dbar m\cdot N)+d\Mn]\otimes N.
\end{align*}
In other words, in terms of partial derivatives, we have for $\ell=1,2$:
$$\dbar_\ell M=[D_\ell m+\Mn \dbar_\ell N]+[(\dbar_\ell m\cdot N)+\partial_\ell \Mn] N.$$
This entails the following extrinsic Dirichlet energy density:
$$
|\dbar M|^2_g=\underbrace{|Dm+\Mn \dbar N|^2_g}_{=:I}+\underbrace{|(\dbar m\cdot N)+d\Mn|^2_g}_{=:II}.
$$
Writing $$E^{mm}_\eps(M)=\int_S \frac12(I+II)+\frac1{4\e^2}F(|m|_g^2)\, \vol_g\leq C|\log \eps|,$$ we deduce by Young's inequality that
\bea
2C|\log \eps|\geq \int_S I\, \vol_g&=\int_S  |Dm|^2_g +\Mn^2 |\dbar N|^2_g+
2\Mn (Dm, \dbar N)_g \vol_g\\
&\geq \int_S  \frac 12|Dm|^2_g-3\|\dbar N \|^2_{L^\infty} \Mn^2\ \vol_g\\
& \geq \int_S \frac12 |Dm|^2_g \, \vol_g-O(\e^2 |\log \e|).
\end{align*}
Therefore,
$$\int_S I\, \vol_g=\int_S |Dm|^2_g +\underbrace{\Mn^2 |\dbar N|^2_g}_{=O(\e^2|\log \e|)}+
2\underbrace{\Mn (Dm, \dbar N)_g}_{=O(\e |\log \e|)} \, \vol_g= \int_S |Dm|^2_g \, \vol_g+O(\e |\log \e|).
$$
The second term $II$ is treated as follows:
\begin{align*}
\int_S II\, \vol_g&=\int_S |\dbar m\cdot N|^2_g+|d\Mn|^2_g +
2(\dbar m\cdot N, d\Mn)_g \, \vol_g\\
& =\int_S |\calS(m)|^2_g+|d\Mn|^2_g  \, \vol_g+O(\e |\log \e|)
\end{align*}
because $\dbar m\cdot N=-m\cdot \dbar N$ so $|\dbar m\cdot N|^2_g=|\calS(m)|^2_g$ and integration by parts yields
$$\int_S (\dbar m\cdot N, d\Mn)_g \, \vol_g=-\int_S \big(d^*(m\cdot \dbar  N), \Mn\big)_g \, \vol_g=
O(\e |\log \e|).$$
Therefore, we obtain:
\beq
\label{e_vs_e2}
\int_{S} |\dbar M|^2_g\, \vol_g=\int_{S} \big(|\dbar m|^2_g+|d\Mn|^2_g\big)\, \vol_g+O(\e |\log \e|)\geq \int_S |Dm|^2_g \, \vol_g+O(\e |\log \e|).
\eeq

\medskip

\nd {\it Step 2. Compactness}. Let $M_\e$ satisfy the assumptions at Theorem \ref{mm.gamma} point 1). 
By Step 1, we deduce that $E^{in}_\e(m_\eps)\leq T \pi |\log \eps| + C$ where $m_\e=\Pi(M_\e)$
is the projection of $M_\e$ on $TS$ (recall that the potential $F$ in $E^{mm}_\e$ satisfies \eqref{F.growth}), i.e., \eqref{F.growth} holds true.
Therefore, Theorem \ref{intrinsic.gammalim} point 1) leads to the desired conclusion.\\

\bigskip

\nd {\it Step 3. Upper bound}. The difference with respect Problem 2 is the following: within the notation in Step 1, as $|M|=1$, one has that $|\Mn|=\sqrt{1-|m|_g^2}$. By \eqref{e_vs_e2}, the only term that changes in the renormalized energy for Problem 2  comes from $|d\Mn|^2_g=|d\sqrt{1-|m|^2}|^2_g$ that influences the energy of the radial profile of a vortex by a constant (therefore, $\iota_F$ in Problem 2 will be replaced by $\tilde \iota_F$).  
Let $u^*=u^*(a,d, \Phi)$ be a canonical harmonic map and $\Theta$ be a minimizer in \eqref{renorm_ext}. 
As $\Theta$ satisfies the associated Euler-Lagrange equation \eqref{PDET}, it yields $\Delta \Theta\in L^\infty$ so 
$\Theta\in C^1(S)$. 
Let $U_\e:=U_\e(a,d, \Phi)$ 
be the vector field constructed in the proof of Proposition \ref{intrinsic.gammalim2} point 2). We have to modify $U_\e$ in the balls $B_{\sqrt\e/2}(a_k)$ according to the micromagnetic radial profile of a vortex given in 
$I^{mm}_F$. For that, we recall that $U_\e$ is denoted in exponential normal coordinates by $V_\e$ around every vortex $a_k$ of degree $d_k\in \{\pm 1\}$ and we have that  $V_\e=e^{id_k\theta}$ on $\partial B_{\sqrt\e/2}(0)$ (up to a rotation). We define $\tv_\e:B(0, \frac{\sqrt \e}2)\to \mathbb{S}^2$ as being a minimizer of $I^{mm}_F(\frac{\sqrt \e}2, \e)$ if $d_k=1$ (or its complex conjugate if $d_k=-1$) up to a rotation. We set $\tv_\e=V_\e$ outside these balls of radius $\frac{\sqrt \e}2$. Denoting by $\tu_\e$ the tangential component of the corresponding map to $\tv_\e$ on $S$, we set $m_\e=e^{i\Theta} \tu_\e$ and $M_\e=m_\e+\Mne N$ where $\Mne=0$ outside the balls $B_{\sqrt\e/2}(a_k)$ and $\Mne$ is the vertical component of $\tv_\e$ inside $B_{\sqrt\e/2}(a_k)$. By the proof of Proposition \ref{intrinsic.gammalim2} point 2) and the above choice of $M_\e$ inside the balls $B_{\sqrt\e/2}(a_k)$, 
we deduce that 
$$\int_S |Dm_\e|^2_g \, \vol_g\leq C|\log \eps|$$ so that \eqref{e_vs_e2} implies
\bea
\int_{S} |\dbar M_\e|^2_g\, \vol_g&=\int_{S} \big(|\dbar m_\e|^2_g+|d\Mne|^2_g\big)\, \vol_g+O(\e |\log \e|)\\
&=\int_{S} (|D \tu_\e|^2_g+ \underbrace{|\tu_\e|_g^2 |d\Theta|^2_g+|\calS(m_\e)|_g^2+2 \big(j(\tu_\e), d\Theta\big)_g}_{III}+|d\Mne|^2_g\big)\, \vol_g+O(\e |\log \e|).
\end{align*}
{\it Estimating $III$}. Recall that $|\tu_\e|_g^2\leq 1$ in $S$ and $\tu_\e=u^*$ in $S_{\sqrt{\e}}$ so that
\begin{align*}
\int_S III \ \vol_g&=\int_S |\tu_\e|_g^2 |d\Theta|^2_g+|\calS(m_\e)|_g^2+2 \big(j(\tu_\e), d\Theta\big)_g \, \vol_g\\
&\leq \int_S |d\Theta|^2_g+|\calS(e^{i\Theta} u^*)|_g^2+2 \big(j(\tu_\e), d\Theta\big)_g \, \vol_g+\int_{S\setminus S_{\sqrt{\e}}} |\calS(m_\e)|_g^2 \vol_g\\
&\leq 2\tilde W(a,d, \Phi)+o(1)
\end{align*}
because 
$$\int_{S\setminus S_{\sqrt{\e}}} |\calS(m_\e)|_g^2 \vol_g=o(1)$$
(as $|m_\e|_g\leq 1$) and
\begin{align*}
\bigg|\int_S \big(j(\tu_\e), d\Theta\big)_g \vol_g\bigg|&\leq
\bigg|\int_S \big(j(u^*), d\Theta\big)_g \vol_g\bigg|+\bigg|\int_{S\setminus S_{\sqrt{\e}}}
\big(j(\tu_\e)-j(u^*), d\Theta\big)_g \vol_g\bigg|\\
&\leq \bigg|\int_S \big(\underbrace{d^*j(u^*)}_{=0}, \Theta\big)_g \vol_g\bigg|+\|d\Theta\|_{L^\infty}\int_{S\setminus S_{\sqrt{\e}}}
|j(\tu_\e)-j(u^*)|_g \vol_g\\
&=o(1)
\end{align*}
because  $\|j(\tu_\e)-j(u^*)\|_{L^1(S\setminus S_{\sqrt{\e}})} \le  \|j(\tu_\e))\|_{L^1(S\setminus S_{\sqrt{\e}})}+\|j(u^*)\|_{L^1(S\setminus S_{\sqrt{\e}})}\to 0$
(from H\"older's inequality in the small balls of radius $\sqrt\e$,  using
control over $\| D\tu_\e\|_{L^2}$ coming from the energy, and estimating $\| Du^*\|_{L^p}$ for $p<2$, as in   Steps 1 and 3 of the proof of  Proposition~\ref{intrinsic.gammalim2}).

\medskip
\nd {\it Estimating the integral of $\frac12(|D \tu_\e|^2_g+|d\Mne|^2_g)+\frac1{4\e^2} F(|m_\e|_g^2)$ on $S$}. First, by definition of $\tv_\e$ inside the ball $B(0, \frac{\sqrt \e}2)$ and $\tilde \iota_F$, we obtain by \eqref{uuu}
$$ \int_{B(0,\frac{\sqrt \e}2)} \tilde e_\e(\tv_\e) \, dy=\pi \log \frac {\sqrt \e}{2\e} + \tilde \iota_F + o(1).
$$
Recall that inside the annulus $B(0, \sqrt \e)\setminus B(0, \frac{\sqrt \e}2)$, we have by \eqref{123}
\begin{align*}
\int_{B(0, \sqrt \e)\setminus B(0, \frac{\sqrt \e}2)} \frac12|\nabla \tv_\e|^2\, dy&=\pi \log 2+o(1)
\end{align*}Thus, by \eqref{normalcs} and \eqref{uvsv} 
as $\Mne=0$ outside $B(a_k,\sqrt \e/2)$:

\begin{align*}
 \int_{B(a_k,\sqrt \e)} \frac12  (|D \tu_\e|^2_g+|d\Mne|^2_g)&+\frac1{4\e^2} F(|m_\e|_g^2)\, \vol_g\\
& = 
\int_{\{ y\in \R^2 : |y|<\sqrt \e\} } \left[(1+O(\e))\tilde e_\e(\tv_\e)+ O(1)\right] \sqrt{g(y)}dy 
\\
 &= \pi \log \frac {\sqrt \e}{\e} + \tilde \iota_F + o(1) \quad \textrm{as } \e\to 0.
\end{align*}
Finally, by definition of $W(a,d,\Phi)$, we have
$$\int_{S_{\sqrt \e} } \frac 12 |D \tu_\e|_g^2 \, \vol_g=\int_{S_{\sqrt \e} } \frac 12 |D u^*|_g^2 \, \vol_g=
W(a,d,\Phi) + n \pi \log \frac 1{\sqrt \e}+o(1).$$
Summing up, the desired upper bound follows.

The convergence $j(m_\eps)\to j(u^*)$ in $L^p(S)$ for every $p\in [1,2)$ follows as in the proof of Theorem \ref{ext.gamma} point 3) because the change made above for $\tilde U_\e$ (instead of $U_\e$) in the small balls $B(a_k,\sqrt \e)$ does not affect the convergence of the current due to
$\|j(m_\eps)\|_{L^p(B(a_k,\sqrt \e))}\to 0$ for every $p\in [1,2)$ (coming from the blow up of $\|j(m_\eps)\|_{L^2(B(a_k,\sqrt \e))}$  as $|\log \e|$ and the H\"older inequality in the ball $B(a_k,\sqrt \e)$). This entails also the convergence of the vorticities $\omega(m_\e)$ in \eqref{convergence_ext}, as well as $\Phi(m_\e)\to \Phi=\Phi(u^*)$.  

\medskip

\nd {\it Step 4. Lower bound}. Let $M_\e$ satisfy the assumptions at Theorem \ref{mm.gamma} point 2). Furthermore, we may assume that $E^{mm}_\e(M_\e)\leq n\pi|\log \e|+c$ for some $c>0$ (otherwise the lower bound is trivial).
By \eqref{e_vs_e2}, we deduce that 
\beq
\label{444}
E^{mm}_\e(M_\e)=E^{ex}_\e(m_\e)+\int_S |dM_{\perp, \e}|^2\, \vol_g+O(\e|\log \e|)\eeq
 where $m_\e=\Pi(M_\e)$
is the projection of $M_\e$ on $TS$ and $F$ satisfies \eqref{F.growth}.
The rest of the proof follows the argument in the proof of Theorem \ref{ext.gamma}, point 2). With the same notation, the only change here concerns the estimate inside the small balls $B_{t_\e}(a_{k, \e})$. As in Step 4 in the proof of Proposition \ref{intrinsic.gammalim2} point 1), one uses the entire micromagnetic energy density and \eqref{uuu} to conclude using \eqref{normalcs} and \eqref{uvsv} 
\bea
\int_{B_{t_\e}(a_{k, \e})}e_\e^{ex}(M_\e)+ |dM_{\perp, \e}|^2 \vol_g  &\ge 
\int_{\{ y\in \R^2 : |y|<t_\e\} } \left[(1+O(\e))\tilde e_\e(v_\e)+ O(1)\right] \sqrt{g(y)}dy \\
&\ge \pi \log \frac{t_\e}{\e} +\tilde \iota_F  + o(1)
\end{align*}
as $\e\to 0$, for every $1\leq k\leq n$ where $v_\e:B(0, t_\e)\to \SSS^2$ is the representation in normal coordinates of $M_\e$, with its in-plane component
corresponding to $m_\e$ and its vertical component to $M_{\perp, \e}$ (see \eqref{decompo}).
Also note the extra term in $|dM_{\perp, \e}|^2$ in \eqref{444} that is used only inside the small balls, therefore, it leads to the extra term in the desired lower bound outside the fixed balls $B_{\sigma}(a_{k})$ around the limit vortices.
\end{proof}

\bigskip

\section{Minimizers of the considered functionals}
\label{sec:min}

In this section, we study the asymptotic behaviour of minimizers of our three functionals as $\e\to 0$.

\medskip

\nd {\it The intrinsic case}.

\begin{theorem}
\label{thm:min_in}
For $\eps>0$, let $u_\eps$ be a minimizer of $E^{in}_\eps$ over the set $\calX^{1,2}(S)$. Then
there exists a sequence $\eps \downarrow 0$ such that for every $p\in [1,2)$,
$$
\begin{cases}
& \omega(u_\eps)  \longrightarrow 2\pi \sum_{k=1}^{n} d^*_k\delta_{a^*_k} \quad \textrm{in } \, W^{-1,p}(S), \\ 
&u_\eps\rightharpoonup u^* \quad \textrm{weakly in } \, \calX^{1,p}(S),\\
&\Phi(u_\eps)\to \Phi^*
\end{cases}
\textrm{ as } \, \eps \to 0,
$$
where $n=|\chi(S)|$, $\{a^*_k\}_{k=1}^n$ are distinct points in $S$, $d^*_k=\sign(\chi(S))$, $\Phi^* \in {\mathcal L}(a^*,d^*)$ such that $(a^*,d^*, \phi^*)$ is a minimizer of the renormalized energy (for the above $d^*$)
$$\{W(a, d^*, \Phi)\, :\, a=(a_1, \dots, a_n)\in S^n \textrm{ distinct points}, \, \Phi\in {\mathcal L}(a,d^*) \}$$
and $u^*$ is a canonical harmonic vector field associated to $(a^*,d^*, \phi^*)$. Moreover, we have the following second order energy expansion:
$$E^{in}_\eps(u_\eps)=n\pi \log \frac 1 \eps+W(a^*, d^*, \Phi^*) +n\iota_F+o(1), \textrm{ as } \, \eps\to 0.$$ 
\end{theorem}

\begin{remark}
\label{rem:min_in}

\smallskip

1) 
We will also prove that $j(u_\e)\to j(u^*)$ in 
$L^p(S)$ for every $p<2$, and that
\begin{equation}\
\int_{S_\sigma}  e_\e^{in}(u_\e) \vol_g\leq C_\sigma, \quad 
\int_{S_\sigma} \bigg[\frac 12 \big|\frac{j(u_\e)}{|u_\e|_g}-j(u^*)\big|_g^2 + 
e_\e^{in}(|u_\e|_g)\bigg] \vol_g \to 0,
\label{r12.2.1}\end{equation}
for every $\sigma>0$, where
$S_\sigma: = S \setminus \cup_{k=1}^n B_\sigma(a^*_k)$.
It follows that
 $|u_\e|_g\to 1$ in $H^1(S_\sigma)$ and $u_\e\rightharpoonup u^*$ weakly in $\calX^{1,2}(S_\sigma)$.

\medskip

2) In the case $\chi(S)=0$, as smooth unit length vector fields do exist over $S$, the minimal energy $E^{in}_\e$ is uniformly bounded as $\e\to 0$. Therefore, any sequence of minimizers $u_\eps$ of $E^{in}_\eps$ has a subsequence strongly convergent in $\calX^{1,2}(S)$ to a smooth canonical harmonic vector field $u^*\in \calX(S)$.
 As $|u_\e|_g\leq 1$ in $S$ (by the standard cutting off argument at $1$ for a minimizer $u_\e$), it entails $j(u_\e)\to j(u^*)$ in $L^2$ and therefore,  
$\omega(u_\eps)\to 0$ in $H^{-1}\cap L^1(S)$ (by \eqref{egal_dj}).
This case was treated in \cite{SSV}. 

\medskip

3) When $\mathfrak g = 0$ (that is, when $S$ is a topological sphere) then $\Phi$
is not present, and the renormalized energy and vector field $u^*$ both simplify significantly.

\end{remark}

\bigskip

\begin{proof}[Proof of Theorem \ref{thm:min_in}] Let $n=|\chi(S)|$. We assume $n>0$ (the case $n=0$ is treated in  Remark \ref{rem:min_in} point 1)). Fix $d_k=\sign(\chi(S))$ for $1\leq k\leq n$. Denoting $d=(d_1, \dots, d_n)$, by Proposition \ref{prop.W}, the direct method in the calculus of variation implies the existence of a minimizer of
$$\{W(a, d, \Phi)\, :\, a=(a_1, \dots, a_n)\in S^n \textrm{ distinct points}, \Phi\in {\mathcal L}(a,d) \}.$$
Fix $(a, d, \Phi)$ such a minimizer.
First, by the upper bound in Theorem \ref{intrinsic.gammalim} point 3) applied to the triple $(a, d, \Phi)$, we deduce that every minimizer $u_\e$ of $E^{in}_\e$ has to satisfy 
\beq
\label{need_up}
E^{in}_\e(u_\e)\leq \pi n|\log \e|+W(a,d, \Phi)+n\iota_F+o(1).  
\eeq
By the compactness result  in Theorem \ref{intrinsic.gammalim} point 1), we have for a subsequence that  $\omega(u_\eps)  \longrightarrow 2\pi \sum_{k=1}^{K} d^*_k\delta_{a^*_k}$ in $W^{-1,p}(S)$ for some distinct points $a^*_k$ and nonzero integers $d^*_k$ satisfying 
$\sum_k d^*_k=\chi(S)$ and $\sum_k |d^*_k|\leq n=|\chi(S)|$. It entails that $K=n$ and  $d^*_k=d_k=\sign(\chi(S))$. In this case, Theorem \ref{intrinsic.gammalim} point 1) gives us the existence of $\Phi^*\in {\mathcal L}(a^*, d^*)$ such that for a subsequence $\Phi(u_\e)\to \Phi^*$. Applying Theorem \ref{intrinsic.gammalim} point 2), we deduce that $E^{in}_\e(u_\e)\geq \pi n|\log \e|+W(a^*,d^*, \Phi^*)+n\iota_F+o(1)$. Then \eqref{need_up} leads as $\eps\to 0$ to $W(a^*,d^*, \Phi^*)\leq W(a,d, \Phi)$. In other words, $(a^*,d^*, \Phi^*)$ is a minimizer of the intrinsic renormalized energy $W(\cdot, d^*, \cdot)$. Moreover, using the stronger lower bound in \eqref{lb1}, we obtain the second estimate in \eqref{r12.2.1}.

It remains to prove the convergence of $u_\e$ to a canonical harmonic vector field\footnote{which 
we recall is unique only up to a global rotation; in fact,  $j(u^*) = j^*(a^*,d^*, \Phi^*)$ is genuinely unique as defined in \eqref{psi.def}, \eqref{form_jstar}. } associated to $(a^*,d^*, \Phi^*)$.  Let $u^*$ denote one such vector field.
By Lemma \ref{L.sep} and Step 2 in the proof of Proposition \ref{intrinsic.gammalim2}, there exist $n$ distinct points  $a^\eps=(a_{k, \e})_{1\leq k\leq n}$ such that $\dist_S(a_{k, \e}, a_{\ell, \e})\geq C_0>0$ for every $k\neq \ell$, $a_{k, \e}\to a^*_k$ and  
$$\|\omega(u_\eps)-2\pi \sign(\chi(S))\sum_{k=1}^{n} \delta_{a_{k, \e}}\|_{W^{-1,1}}=o(1) \, \textrm{  as }\,  \e\to 0.$$ Then, by Proposition \ref{P.vballs} and Lemma \ref{L.near}, we deduce that for every small $\eps>0$ and for every $r\in(\eps^\beta, r_0(C_0))$ (with $\beta=\beta(n)>0$), 
\beq
\label{inter_equ}
\int_{B_r(a_{k, \e})} e_\e^{in}(u_\e)
 \vol_g \ge
 \pi \log \frac{r}{\e}- C, \quad k=1, \dots, n.
\eeq
By the upper bound \eqref{need_up}, it yields 
$$\int_{S\setminus \cup_k B_r(a_{k, \e})} e_\e^{in}(u_\e)
 \vol_g \leq n\pi |\log r|+C$$
 for every small $\eps>0$ and for every $r\in(\eps^\beta, r_0(C_0))$. It follows by Lemma \ref{lem:Struwe} below that $(u_\e)_{\e \to 0}$ is bounded in $\calX^{1,p}(S)$, therefore, for a subsequence, $u_\e\rightharpoonup u_*$ in $\calX^{1,p}(S)$ for every $p\in [1,2)$ for a unit length vector field $u_*$. As $a_{k, \e}\to a^*_k$, by \eqref{inter_equ}, we also deduce the first estimate in
 \eqref{r12.2.1}.
  
Now we aim to prove that $u_*$ is a canonical harmonic unit vector field,
{\it i.e.} that $|u_*|=1$ (which is obvious) and
  $j(u_*) = j(u^*)$.
By the Sobolev embedding, we have that $u_\e\to u_*$ strongly in $L^q(S)$ for every $q<\infty$, and we know that $Du_\e\rightharpoonup Du_*$ in $L^p$ for
every $p\in [1,2)$. 
Together these imply that  $j(u_\e) \rightharpoonup j(u_*)$ in $L^p(S)$ for every $p\in [1,2)$. On the other hand, we claim that $j(u_\e)\to j(u^*)$ in $L^p(S)$
for every $p\in [1,2)$. 
To prove this, fix $p\in [1,2)$ and
note that
 \begin{align*}
\|\frac{j(u_\e)}{|u_\e|_g} - j(u_\e) \|_{L^p(S)} 
&\le \| |Du_\e|_g (1 - |u_\e|_g)\|_{L^p(S)} \\
& \le 
\| Du_\e\|_{L^2(S)} 
\| 1-|u_\e|_g\|_{L^2(S)}^{\frac{2-p}{p}}
\| D( 1-|u_\e|_g)\|_{L^2(S)}^{2-\frac2p}
\end{align*}
by H\"older and a Gagliardo-Nirenberg inequality. 
Since $(1 - |u_\e|_g)^2\leq \frac1C F(|u_\e|_g^2)$ (by \eqref{F.growth}) and $|D(1-|u_\e|_g)|_g \le |Du_\e|_g$,
all terms on the right-hand side can be controlled by $E^{in}_\e(u_\e)$,
leading to
\[
\|\frac{j(u_\e)}{|u_\e|_g} - j(u_\e) \|_{L^p(S)} 
\le C|\log\e|\, \e^{\frac{2-p} p}.
\]
Also, \eqref{r12.2.1} and H\"older's inequality readily imply that
$\|\frac{j(u_\e)}{|u_\e|_g} - j(u^*) \|_{L^p(S_\sigma)}  \to 0$ for every $\sigma>0$,
and hence that $\| j(u_\e) - j(u^*) \|_{L^p(S_\sigma)}  \to 0$. Finally, for 
$q\in (p,2)$, since $j(u^*)\in L^q(S)$ and $\{j(u_\e)\}$ is uniformly bounded in
$L^q(S)$, H\"older's inequality implies that 
\[
\| j(u_\e) - j(u^*)\|_{L^p(S\setminus S_\sigma)} \le
 \|1 \|_{L^{\frac{qp}{q-p}}(S\setminus S_\sigma)} 
 \| j(u_\e) - j(u^*)\|_{L^q(S\setminus S_\sigma)} 
 \le
C \sigma^{\frac {2(q-p)}{pq}}.
\]
As $\sigma$ can be chosen arbitrarily small, summing up, we obtain that $j(u_\e)\to j(u^*)$ in $L^p(S)$ as claimed.
(This completes the proof of Remark \ref{rem:min_in}).
Since $j(u_*)$ and $j(u^*)$ are both limits of $j(u_\e)$, we conclude that
$j(u_*) = j(u^*)$ and hence that $u_*$ is a canonical harmonic unit vector field.
\end{proof}

We use the following estimate reminiscent from the work of Struwe \cite{Struwe:1994a} in a ball $B_r\subset S$ of radius $r\in (0,r_0)$ (thus,
$r_0$ is at most the injectivity radius of $S$) and we have
\begin{equation}
\left |\mbox{Vol}_g(B_r) - \pi r^2 \right| \ \le \ c_1 r^4, 
\qquad\quad
\left| \calH^1(\partial B_r)  - 2\pi r \right| \le c_1 r^3 
\label{fix.c0}\end{equation}
which is a consequence of Bertrand-Diguet-Puisseux 
Theorem (see Section \ref{subsec:connect})  and the compactness of $S$.

\begin{lemma}\label{lem:Struwe}
Let $r_0$ be the injectivity radius of $S$, $\beta, c>0$ and $\rho_0, R\in (0, r_0)$ with $\rho_0<R$.
For every $\e\in (0, \frac12)$, let $f_\e :B_R\to \R$ be a function on a ball $B_R\subset S$
such that 
$$\|f_\e\|^2_{L^2(B_R)}\leq c(1+|\log \e|)$$ and 
for every $\e^\beta \leq \rho \leq \rho_0$,
\[
\|f_\e\|^2_{L^2(B_R\setminus B_\rho)} \leq c (1+|\log \rho|).
\]
Then for $1\leq p<2$ we have
\[
\|f_\e\|_{L^p(B_R)} \leq C,
\]
where $C>0$ is independent of $\e$.
\end{lemma}

\begin{proof} Let $1\le p<2$ and $\rho_j=2^{-j}\rho_0$ for $0\leq j\leq j_\beta$ with $j_\beta=\lfloor \log_2 \frac{\rho_0}{\e^\beta}\rfloor$ where $\lfloor s \rfloor$ is the integer part of a real $s$.
Using H\"older's inequality, we have
\begin{align*}
\int_{B_{\rho_0}\setminus B_{\rho_{j_\beta}}} |f_\e|^p \, \vol_g
& = \sum_{j=0}^{j_\beta-1} \int_{B_{\rho_j}\setminus B_{\rho_{j+1}}} (|f_\e|^2)^{\frac{p}{2}}  \, \vol_g
\\
&\leq\sum_{j=0}^{j_\beta-1} \left(\int_{B_{\rho_j}\setminus B_{\rho_{j+1}}} |f_\e|^2\vol_g\right)^{\frac p2}
\mbox{Vol}_g(B_{\rho_j}\setminus B_{\rho_{j+1}})^{1-\frac p2}\\
&\le C(p,S)\sum_{j=0}^{j_\beta-1}
\left(\int_{B_R\setminus B_{\rho_{j+1}}} |f_\e|^2 \vol_g\right)^{\frac p2} (2^{-j}\rho_0)^{2- p}\\
&\le C(p,S, \rho_0)\sum_{j=0}^\infty (1+j\log 2-\log \rho_0)^{\frac p2} 2^{-(2-p)j}.
\end{align*}
Since 
\[
\lim_{j\to \infty} \frac{(1+(j+1)\log 2-{\log \rho_0})^{\frac p2} 2^{-(2-p)(j+1)}}{(1+j\log 2-{ \log \rho_0})^{\frac p2} 2^{-(2-p)j}}=2^{-(2-p)}<1,
\]
the above sum converges so $ \|f_\e\|_{L^p(B_{\rho_0}\setminus B_{\rho_{j_\beta}})}\le C$. Also, the hypothesis combined with H\"older's inequality yield $\|f_\e\|_{L^p(B_R\setminus B_{\rho_0})} \le C$ as well as 
\[
 \|f_\e\|_{L^p(B_{\rho_{j_\beta}})}\le \|f_\e\|_{L^2(B_R)} \mbox{Vol}_g (B_{\rho_{j_\beta}})^{\frac1p-\frac12}=O(|\log \eps| \e^{\beta(\frac2p-1 )})=o(1).
 \]
\end{proof}

\bigskip

\nd {\it The extrinsic case}.

If $S$ is a surface isometrically embedded in $\R^3$, then Theorem \ref{thm:min_in} holds true also in the extrinsic case for minimizing sections $m_\e$ of $E^{ex}_\e$ with the natural change of having the limit triplet
$(a^*,d^*, \phi^*)$ to be a minimizer of the extrinsic (instead of the intrinsic) renormalized energy 
$$
\{(W+\tilde W)(a, d^*, \Phi)\, :\, a=(a_1, \dots, a_n)\in S^n \textrm{ distinct points},  \Phi\in {\mathcal L}(a,d^*) \}.
$$
Moreover, we have the following formula for the minimal extrinsic energy:
\beq
\label{aici2}
E^{ex}_\eps(m_\eps)=n\pi \log \frac 1 \eps+(W+\tilde W)(a^*, d^*, \Phi^*) +n\iota_F+o(1),
\eeq
as $\e\to 0$. Finally, after passing to a subsequence if necessary, 
$$m_\e\rightharpoonup  e^{i\Theta^*}u^* \,  \textrm{ in } \, \calX^{1,p}(S),$$ where
$u^*$ is some fixed canonical harmonic map $u^*(a^*,d^*,\phi^*)$ and $\Theta^*$
is a minimizer of 
$\Theta \mapsto \frac 12\int_S|d\Theta|_g^2 + |\mathcal S(e^{i\Theta}u^*)|_g^2\vol_g$.

\medskip

We  sketch the proof of the above claim.
The energy expansion \eqref{aici2} and the convergence 
\[
\omega(m_\e)\to 2\pi \sum_{k=1}^n d_k^* \delta_{a_k^*}, \qquad
\Phi(m_\e) \to \Phi^*
\]
are proved exactly as in the intrinsic case, using Theorem \ref{ext.gamma}
in place of Theorem \ref{intrinsic.gammalim}.
Similarly, exactly the same arguments as for the intrinsic case prove that there
exists some $m_*\in \calX^{1,p}(S)$ for $m_\e\rightharpoonup m_*$
in $\calX^{1,p}(S)$, for all $p\in [1,2)$.

It remains to prove that  $m_* = e^{i\Theta^*}u^*(a^*,d^*, \Phi^*)$.
We will deduce this conclusion from estimates carried out during the proof
of Theorem \ref{ext.gamma}. We start by recalling from the
proof of  point 2)  Theorem \ref{ext.gamma} that for 
every
sufficiently small $\e>0$, there exists 
$t_\e\in (\e^{\frac 1{2(n+1)}}, \e^{\beta/2})$, $n$ distinct 
points $a^\e=(a_{k, \e})_{1\leq k \leq n} \in S^n$, $d^\e\in \{ \pm 1\}^n$, and $\{ \Phi_{k,\e} \}_{k=1}^{2\mathfrak g} \in \calL(a^\e, d^\e)$
such that 
\[
a^\e\to a^*, 
\qquad d^\e\to d^* \ \  \mbox{(so $d^\e = d^*$ for all small $\e$)}, \qquad 
\Phi^\e\to \Phi^*
\]
and
\begin{align*}
\int_{B_{t_\e}(a_{k, \e})}e_\e^{ex}(m_\e) \vol_g
&\ge \pi \log \frac{t_\e}{\e} +\iota_F   + o(1)\quad\mbox{ for every $1\leq k\leq n$,}
\\
\int_{S_{t_\e}} e_\e^{ex}(m_\e) \vol_g 
&\ge n\pi \log \frac 1{t_\e} + (W+\tilde W)(a^\e, d^\e, \Phi^\e)-  o(1)
\end{align*}
as $\e \to 0$. Let $u_\e^*$ be a canonical harmonic map $u^*(a^\e, d^\e, \Phi^\e)$ so that $u_\e^*$
(up to a rotation as in  \eqref{remain3}) satisfies
 $u_\e^* \to u^*$ in $L^p(S)$, $p<\infty$, as $\e\to 0$.
We may then define $w_\e:S_{t_\e}\to \C$ by requiring that
\[
m_\e = w_\e u^*_\e \qquad\mbox{ in }S_{t_\e}.
\]
Then \eqref{elb.3}, 
established 
during the proof of
Lemma \ref{L.qgamma.ex}, can be applied in $S_{t_\e}$ (instead of $S_{r_\e}$) implies that
\[
\int_{S_{t_\e}} e_\e^{ex}(m_\e) \vol_g 
= W(a^\e, d^\e, \Phi^\e) +n\pi\log \frac 1{t_\e} 
+
\int_{S_{t_\e}} \frac 12|dw_\e|_g^2
 +\frac 12 |\mathcal S(w_\e u^*_\e)|_g^2 + \frac 1{{4\e^2}} F(|w_\e|^2) \ \vol_g 
 +  o(1) .
\]
Combining these and \eqref{aici2} and recalling that $W(a^\e, d^\e, \Phi^\e)\to W(a^*, d^*, \Phi^*)$ (by Proposition \ref{prop.W}) and  that  $\tilde W(a^\e, d^\e, \Phi^\e)\to \tilde W(a^*, d^*, \Phi^*)$ (by Remark \ref{lem:add}),
we obtain
\[
\int_{S_{t_\e}} \frac 12|dw_\e|_g^2
 +\frac 12 |\mathcal S(w_\e u^*_\e)|_g^2 + \frac 1{{4\e^2}} F(|w_\e|^2) \ \vol_g 
\to \tilde W(a^*, d^*, \Phi^*)
\]
as $\e\to 0$. 
Next, following the arguments in Step 3 of the proof of Lemma
\ref{L.qgamma.ex}, we may construct a function $\tilde w_\e:S\to \C$ such 
that $\tilde w_\e = w_\e$ in $S_{\sqrt{r_\e}}$ and
\[
\limsup_{\e\to 0}\int_{S} \frac 12|d \tilde w_\e|_g^2
 +\frac 12 |\mathcal S(\tilde w_\e u^*_\e)|_g^2 + \frac 1{{4\e^2}} F(|\tilde w_\e|^2) \ \vol_g 
\leq \tilde W(a^*, d^*, \Phi^*).
\]
In particular $\{ \tilde w_\e\}$ is uniformly bounded in $H^1(S)$. We pass to a 
subsequence that converges weakly in $H^1$, and hence strongly in $L^p$
for every $p<\infty$, to a limit $w_*$. As in Step 4'' of the proof of Lemma \ref{L.qgamma.ex}, 
one checks that $w_*\in H^1(S;\SSS^1)$ satisfying the compatibility condition $\Phi(w_* u^*)=\Phi(u^*)$ so that Lemma \ref{lem:lifting} yields $w_* = e^{i\Theta_*}$. Standard 
lower semicontinuity arguments imply
\begin{align*}
\tilde W(a^*, d^*, \Phi^*)&\leq \int_{S} \frac 12|d \Theta_*|_g^2
 +\frac 12 |\mathcal S(e^{i\Theta_*} u^*)|_g^2 \ \vol_g\\& = \int_{S} \frac 12|d w_*|_g^2
 +\frac 12 |\mathcal S(w_* u^*)|_g^2 \ \vol_g \\
&\leq \liminf_{\e\to 0} \int_{S} \frac 12|d \tilde w_\e|_g^2
 +\frac 12 |\mathcal S(\tilde w_\e u^*_\e)|_g^2 + \frac 1{{4\e^2}} F(|\tilde w_\e|^2) \ \vol_g \\
& \leq \tilde W(a^*, d^*, \Phi^*).
\end{align*}
That is, $\Theta_*$ attains the minimum in the definition
of $\tilde W$.
Finally, the construction implies that for {\it a.e.} $x\in S$,
$
m_* = \lim_\e m_\e = \lim_\e \tilde w_\e u_\e^* = e^{i\Theta_*}u^*
$,
completing the proof.

\bigskip

\nd {\it The micromagnetic case}.

Let $S$ is a surface isometrically embedded in $\R^3$ endowed with the Euclidian metric. If $M_\e:S\to \SSS^2$
is a minimizer of $E^{mm}_\e$, then the projections $m_\eps=\Pi(M_\e)$ on $TS$ satisfy (for a subsequence) the convergences in Theorem \ref{thm:min_in} as $\eps\to 0$ where the limit triplet $(a^*,d^*, \phi^*)$ is a minimizer of the extrinsic renormalized energy $W+\tilde W$. The second order expansion of the minimal micromagnetic energy has the form \eqref{aici2} with the natural change of $\iota_F$ by $\tilde \iota_F$. Moreover, the normal components $M_{\perp, \eps}$ of $M_\e$ satisfy $M_{\perp, \eps}\to 0$ in $H^1(S \setminus \cup_{k=1}^n B_\sigma(a_k))$ as $\e\to 0$.

\bigskip

\appendix
\section{Ball construction. Proof of Proposition \ref{P.vballs}}\label{App.A}

In this section we present the vortex ball construction leading to 
Proposition \ref{P.vballs}.
We start with several lemmas in
which we verify, largely by adapting classical proofs to our setting, that
basic ingredients needed for the vortex ball argument on the Euclidean plane remain available in the present setting.
Once these ingredients are available, 
we follow classical arguments.
Since the arguments are rather standard,
our exposition is terse in places.

We first fix  positive constants $c_1(S), r_0(S)$ such that 
$\partial B_r(x)$ is a homeomorphic to a circle
for every $x\in S$,
and $0<r<r_0$ (thus,
$r_0$ is at most the injectivity radius of $S$) and
we recall \eqref{fix.c0}.
In several places later in our argument,
we will impose additional smallness conditions on
$r_0$.

In view of Lemma \ref{L.density}, in proving lower energy bounds we may 
restrict our attention to smooth vector fields.

\begin{lemma}\label{L.ingredients}Given $u\in \calX(S)$, let $\rho:= |u|_g$.
Assume that $\e<r<r_0(S)$. Then 
there exist positive constants $c_2,c_3, c_4$ such that the following hold.
First, 
\begin{equation}
\frac 12 \int_{\partial B_r(x)} | d \rho|_g^2 + \frac 1 {2\e^2}F(\rho^2) \, d\calH^1
\ge
\frac
{c_2}{\e} \| 1-\rho\|_{L^\infty}^2
\label{lb.first}\end{equation}
where $|d\rho|_g^2(x) := (d\rho(\tau_1))^2 + (d\rho(\tau_2))^2$ for any orthonormal
basis $\{ \tau_1,\tau_2 \}$ of $T_xS$.
Second, 
\begin{equation}
\int_{\partial B_r(x)}  e_\e^{in}(u) d\calH^1
\ge \lambda_\e(\frac r{|d|} )\quad\mbox{ for } d = \begin{cases}
\deg(u; \partial B_r(x))&\mbox{ if $\rho\ge \frac 12 $ on $\partial B_r(x)$ }\\
\mbox{any positive integer} &\mbox{ if not }
\end{cases}
\label{lb.third}\end{equation}
where
\begin{equation}\label{lambdaep.def}
\lambda_\e(r) :=
\min_{0<s\le 1} \left[ \frac {c_2}{4 \e}(1-s)^2 + s^2  \frac {\pi }{r}(1- c_3 r^2)\,\right] \ge 
\frac{\pi(1-c_3 r^2)} {r+c_4 \e} 
\end{equation}
is a nonincreasing function, and we use the convention that  for $r>0$,  $\lambda_\e(r/0) =  \lambda_\e(+\infty) = 0 $.
\end{lemma}

\begin{proof}
At $y\in \partial B_r(x)$, 
let $\tau\in T_yS$ denote the unit tangent to $\partial B_r(x)$,
oriented in the standard way, and let ${}'$ denote differentiation 
with respect to $\tau$.
Further define $\zeta := (1 - \rho)^2$. Then by \eqref{F.growth}
\[
|\zeta'| + \frac 1 \e|\zeta|      = 2|\rho-1| |\rho'|  + \frac 1 \e \zeta \le
\e |\rho'|^2 + \frac 2\e (1-\rho)^2 \le 
C \e \left(|d \rho|_g^2 + \frac 1 {2\e^2}F(\rho^2)\right).
\]
Then \eqref{lb.first} follows from a (suitably scaled) Sobolev embedding $W^{1,1}\hookrightarrow L^\infty$
on $\partial B_r(x)$,
taking into account the fact that $\calH^1(\partial B_r(x)) \ge \e$.

Next, if $\rho\ge \frac 12$ on $\partial B_r(x)$, then
we can define $v = u/\rho$ and $d:= \deg(u; \partial B_r(x))$ . Since $|v|_g=1$, we have $|Dv|_g = |j(v)|_g$, and thus
\begin{align*}
\int_{\partial B_r(x)}|Dv|_g^2 \, d\h^1
=
\int_{\partial B_r(x)}|j(v)|_g^2 \, d\h^1
&
\ \ge\  \frac 1{\calH^1(\partial B_r(x))} \big|\int_{\partial B_r(x)}j(v) \big|^2_g
\\
&\overset{\eqref{deg.def}}= \frac 1{\calH^1(\partial B_r(x))}
\left( 2\pi d - \int_{B_r(x)}\kappa \ \vol_g \right)^2.
\end{align*}
Since $S$ is compact and smooth,  
it follows from  this and \eqref{fix.c0} that
\begin{equation}
\int_{\partial B_r(x)} \frac 12 |Dv|_g^2 d\calH^1 \ge \frac {\pi d^2}{r}(1- c_3 r^2)\, .
\label{lb.second}\end{equation}

Finally, if we write $s :=\min_{\partial B_r(x)} (\rho \wedge 1) >0$, then 
$|Du|_g^2 = |d\rho|{_g}^2 + \rho^2|Dv|{_g}^2 \ge |d\rho|_g^2 + s^2 |Dv|_g^2$.
Then 
one may deduce  \eqref{lb.third}
from \eqref{lb.first} and \eqref{lb.second}, after first taking $r_0$  small enough so that $c_3r_0^2 \le 1/2$, which yields $|d|(1-c_3r^2)\geq 1-c_3\frac{r^2}{d^2}$ for every $|d|\geq 1$.
Then \eqref{lb.third} follows directly in the case $\rho\ge \frac 12$ on $\partial B_r(x)$, if $d\ne 0$, whereas if $d=0$ it is immediate.
If $\min_{\partial B_r(x)}\rho < \frac 12$, then $ \| 1-\rho\|_{L^\infty} > \frac 12$, and thus
\begin{align*}
\int_{\partial B_r(x)}  e_\e^{in}(u) d\calH^1
&\ge
\frac 12 \int_{\partial B_r(x)} | d \rho|_g^2 + \frac 1 {2\e^2}F(\rho^2) \, d\calH^1
\overset{\eqref{lb.first}}\ge
\frac
{c_2}{4\e} \\
&\ge 
\min_{0<s\le 1} \left[ \frac {c_2}{4 \e}(1-s)^2 + s^2  \frac {\pi }{r}(1- c_3 r^2)\,\right] 
= \lambda_\e(r) \ge \lambda_\e(\frac r {|d|})
\end{align*}
for any $d$. (If $\rho=0$ somewhere on $\partial B_r(x)$, then the definition $v = u/\rho$ may not make sense, but the proof of \eqref{lb.third} relies only on \eqref{lb.first} and makes no mention of $v$.)
\end{proof}

We also need:

\begin{lemma}\label{L.degree.energy}
Assume that $u$ is a smooth vector field on $S$ and that
for some $0<r<r_0(S)$ and $x\in S$, 
\[
\rho := |u|_g \ge \frac 12 \mbox{ on }\partial B_r(x), \qquad
\deg(u;\partial B_r(x)) = d\ne 0.
\]
Then if $r_0$ is sufficiently small, 
\[
\int_{B_r(x)} |Du|_g^2 \vol_g \ge  \frac \pi 4 |d| \, . 
\]
\end{lemma}

\begin{proof}
First, let $O := \{y\in B_r(x) : \rho(y)<t \}$, where $t$
is  a regular value  of $\rho(\cdot)$ such that $\frac 18<t<\frac 14$. 
Then $O$ is an open set with smooth
boundary, compactly contained in $B_r(x)$,  and 
\begin{align*}
d
= \deg(u;\partial B_r(x))  
&= 
\frac 1{2\pi} 
\left( \int_{\partial B_r(x)} j(\frac u{|u|_g}) +
\int_{B_r(x)} \kappa\vol_g \right) \\
&= 
\frac 1{2\pi} \left( \int_{\partial O} \frac{ j(u)}{t^2}  + \int_O \kappa \vol_g
\right)
\end{align*}
where the final equality follows from Lemma \ref{L.omega0} and Stokes' Theorem,
as well as the fact that $|u|_g=t$ on $\partial O$. Thus if $r_0$ is sufficiently small
(depending on $\|\kappa\|_\infty$) then
\[
2\pi t^2 \left( |d| - \frac 12  \right)\le \big|\int_{\partial O} j(u)
\big|_g = 
 \big|\int_{O} dj(u) \big|_g \le  \int_O|d j(u)|_g \, \le \int_O |Du|_g^2,
\]
where the last inequality follows from
Lemma \ref{L.last}, see \eqref{egal_dj}.
\end{proof}

Finally we recall 

\begin{lemma}  \label{Zballs}
Assume that $u_\e$ is a smooth vector field satisfying \eqref{Escaling}. Then 
there exists $\e_0>0$ such that whenever $0<\e <\e_0$,
there exists a collection $\tilde {\mathcal B}^0  = \{ \tilde B^0_j \}$ of closed 
pairwise disjoint
balls that cover the set where $|u_\e|_g\le \frac 12$, and such that 
\[
\sum_j \tilde r^0_j \le C \e \int_S e_\e^{in}(|u_\e|_g)\, \vol_g, \qquad \mbox{ where $\tilde r^0_j$ denotes 
the radius of $\tilde B^0_j$}.
\]
\end{lemma}

\begin{proof}
Let $\rho := |u_\e|_g$.
Then $\frac 12 |d\rho|_g^2 + \frac 1 {4\e^2}F(\rho^2) \ge \frac 1{\e\sqrt 2} |d\rho|_g \sqrt {F(\rho^2)} \ge \frac c\e |1-\rho| \, |d\rho|_g$, by \eqref{F.growth}. Thus the coarea formula, which remains valid on a smooth manifold, implies that
\[
\e \int_S e_\e^{in}(\rho) \vol_g \ge 
c \int_0^\infty |1-s| \calH^1(\rho^{-1}(s))\, ds
\ge
c \int_{1/2}^{3/4}\calH^1(\rho^{-1}(s))\, ds.
\]
In particular we may find some $\alpha\in [\frac 12, \frac 34]$, a regular value of $\rho$, such that
$\calH^1(\rho^{-1}(\alpha)) \le C \e \int_S e_\e^{in}(\rho) \vol_g$.
Following standard arguments, we may start with an efficient
finite cover of $\rho^{-1}(\alpha)$ and then merge
balls to find a collection
of closed {\em pairwise disjoint} balls that cover $\rho^{-1}(\alpha)$
and whose radii sum to at most
$2\calH^1(\rho^{-1}(\alpha))$. This is $\tilde {\mathcal B}^0$.
The complement of the union
of these balls is connected as long as $\e$ is small enough, so on the complement,
either $\rho > \alpha$ or $\rho<\alpha$ everywhere.
The latter case is impossible by \eqref{Escaling} and \eqref{lb.first}, 
if  $\e$ is small enough, and this proves the lemma.
\end{proof}

A few more definitions are needed before we prove Proposition \ref{P.vballs}.
W.l.o.g., we may assume that $\frac12$ is a regular value of $\rho=|u|_g$. First, we set
\begin{align*}
Z &:= \{ x\in S : |u(x)|_g \le \frac 12\}, \\ 
Z_E& :=  \cup \{ \mbox{connected components $Z_l$ of $Z$ }  \ : \ \deg(u; \partial Z_l) \ne0 \}.
\end{align*}

Next, for any set $V\subset S$ such that $\partial V \cap Z_E = \emptyset$ we define
the generalized degree
\[
\mbox{dg}(u;\partial V) := \sum \{ \deg(u; \partial Z_l) \ : \ \mbox{ components $Z_l$ of $Z_E$
such that }Z_l\subset\subset V\}.
\]
Note that $\mbox{dg}(u;\partial V) = \deg(u,\partial V)$ if $\partial V$ is $C^1$, say,
and $|u|_g > \frac 12$ on $\partial V$.
Finally we define
\[
\Lambda_\e(\sigma)  := \int_0^\sigma \lambda_\e(r)\, dr.
\]
It is straightforward to check that 
\begin{equation}\label{Lambda.log}
\Lambda_\e(\sigma) \ge \pi \log(1+ \frac \sigma{c_4\e}) -  C \sigma^2  \ge \pi (\log\frac \sigma \e - C)\qquad\mbox{ for }0\le \sigma \le r_0(S).
\end{equation}
We record several other properties. First, since $\Lambda_\e(\cdot)$ is the
integral over $[0,\sigma]$ of a positive nonincreasing function, it is easy to see that
\[
\Lambda_\e(\sigma_1+\sigma_2) \le \Lambda_\e(\sigma_1) + \Lambda_\e(\sigma_2),
\qquad \sigma\mapsto \frac 1 \sigma\Lambda_\e(\sigma)\mbox{ is nonincreasing}. 
\]
Finally, consider two radii $r_1<r_2$ such that $\e\le r_j \le r_0$ for $j=1,2$,
and assume that $x\in S$ is a point such that 
$Z_E$ does not intersect the annulus $B_{r_2}\setminus B_{r_1}(x)$.
Then $\mbox{dg}(u;\partial B_r(x)) = \mbox{dg}(u; \partial B_{r_1}(x))$ for
all $r\in (r_1,r_2)$, so one may 
use the coarea formula and integrate \eqref{lb.third} from $r_1$ to $r_2$ 
to find that 
\begin{equation}
\int_{B_{r_2}\setminus B_{r_1}(x)}e_\e^{in}(u) \, \vol_g \ge {|d|}\left[ \Lambda_\e(\frac{r_2}{{|d|}}) - \Lambda_\e(\frac {r_1}{{|d|}})\right],\qquad d:= \mbox{dg}(u; \partial B_{r_1}(x)).
\label{best2}\end{equation}

\begin{proof}[Proof of Proposition \ref{P.vballs}]
We divide the proof in several steps:

\medskip

\nd {\it Step 1. An initial covering of $Z_E$.}
We claim first that there exists a collection $\mathcal B^0 = \{ B_{l,0}\}_{k=1}^K$ of 
closed, pairwise disjoint balls with  centers $a_{l,0}$ and radii $r_{l,0}\ge \e$ for all $l$, 
such that 
$Z_E\subset \cup B^0_k$, 
and (after possibly decreasing the constant $c_2$ in the definition \eqref{lambdaep.def} of $\lambda_\e$, in a way that depends only on the geometry of $S$)
\begin{equation}
\int_{B_{l,0}} e^{in}_\e(u) \vol_g \ \ge \frac {c_2}{4\e} r_{l,0} \ge \ \Lambda_\e(r_{l,0}) \qquad\mbox{ for every }l.
\label{best1}\end{equation}
We first cover  $Z_E$ with balls that satisfy
\eqref{best1}. 
Indeed, for every $x\in Z_E$, this estimate holds for  $B_r(x)$, for the smallest $r\ge \e$ such that $\min_{\partial B_r(x)}\rho \ge 1/2$.  This is  a result of  Lemma \ref{L.degree.energy},
if its $r\le 2\e$, and otherwise it follows from  \eqref{lb.first} and the coarea formula. One can then choose a finite subcover.
The balls obtained in this fashion may overlap. If so, they may be combined
into pairwise disjoint balls that still satisfy \eqref{best1},
by exactly the arguments in   \cite{Je},  proof of Proposition 3.3,
where the same result is proved in the Euclidean setting. This argument 
involves a slightly more
careful choice of balls (so that no center is contained in any other ball) and use of the Besicovitch covering lemma. For our present purposes, we may appeal to Federer \cite{Federer}, sections 2.8.9 - 2.8.14,
for a difficult  but doubtless correct version of the
covering lemma that is valid on a smooth compact Riemannian manifold,
and indeed in much greater generality.
Adjustments to the constant $c_2$ depend on constants appearing in this
covering lemma, which are explicitly described in the above reference.

\medskip

\nd {\it  Step 2. Growing and merging balls.}
Now let $d_{l,0} := \mbox{dg}(u;\partial B_{l,0})$. We will assume for this
discussion that $d_{l,0}\ne 0$ for some  $l$, 
as the other case is 
both easier and less relevant for our main results.
Using Lemma \ref{L.ingredients} and associated properties of
$\Lambda_\e$, such as  those in \eqref{best1},
we may now follow the algorithm 
from \cite{Je}, proof of Proposition 4.1, to which one
may refer for the details omitted here.
We describe it briefly.
First, define
\[
\sigma_0 := \min_{\mathcal B^0} r_{l,0}/ |d_{l,0}| \overset{\eqref{Escaling}, \eqref{best1}}\le C \ \e \logeps \ .
\]
 We claim that for 
$\sigma\in (\sigma_0, r_0(S))$,
there exists a finite collection of pairwise disjoint closed balls $\mathcal B^\sigma = \{ B_{l,\sigma}\}_{l=1}^{K_\sigma}$
with centers $a_{l,\sigma}$ and radii $r_{l,\sigma}$, such that 
\begin{equation}
Z_E \subset \cup B_{l,\sigma}\, , \qquad
\int_{B_{l,\sigma}} e_\e^{in}(u)\vol_g \ge \frac {r_{l,\sigma}}\sigma \Lambda_\e(\sigma) \,,\ 
\mbox{ and } \ r_{l,\sigma}\ge  \sigma |d_{l,\sigma}|\quad\mbox{ for all }l, 
\label{best3}\end{equation}
where $d_{l,\sigma} = \mbox{dg}(u;B_{l,\sigma})$. 
We take $\mathcal B^{\sigma_0}$ to be the collection found in Step 1
above.
Given any $\sigma_1 \ge \sigma_0$ for which such a collection exists, 
we say that the {\em minimizing balls} are those for which $r_{l,\sigma_1} = \sigma_1 |d_{l,\sigma_1}|$. Since the balls are closed and pairwise disjoint, there is some
$\delta>0$ such that for $\sigma_1 \le \sigma \le \sigma+\delta$, 
we can expand the minimizing balls, while leaving the centers fixed,
by enclosing them in pairwise disjoint annuli chosen so  that, 
for every $\sigma$, the equality
$r_{l,\sigma} = {\sigma} |d_{l,\sigma}|$ holds for all minimizing balls. 
We add balls to the collection of minimizing balls as $\sigma$ increases, when necessary.
This preserves \eqref{best3} due to properties of $\Lambda_\e$ summarized above, such as \eqref{best2}. At certain values of $\sigma$, for example $\sigma = \sigma_1+\delta$, the expansion process will lead to two or more balls colliding. When this occurs, one can regroup them into larger, pairwise disjoint balls
in a way that preserves the properties \eqref{best3}. (Details of all
these assertions can be found in \cite{Je}.) This process can
be continued as long as  every minimizing ball has radius at most
$r_0(S)$, which happens as long as $\sigma < r_0(S)$.

\medskip

\nd {\it Step 3. Stopping the process, and covering all of $Z$.}
Recalling that $n>T-1$ by hypothesis, we fix $q \in (0,  1-\frac T{n+1})$,
which implies that $\frac T{1-q}< n+1$.
It then follows from \eqref{best3}, \eqref{Lambda.log}, and \eqref{Escaling} that
if $\e^q\le \sigma < r_0(S)$, then
\begin{equation}\label{radii.bd}
\sigma \sum |d_{l,\sigma}| \le \sum r_{l,\sigma} \le \sigma \frac{T\pi  \logeps + C}{\pi |\log (\sigma/\e)|-C} \le \sigma\Big(\frac T{1-q} + \frac C \logeps
\Big).
\end{equation}
Thus there exists $\e_0>0$
(depending on $S$, $q$, and the constant in \eqref{Escaling}) 
 such that if $0 <\e <\e_0$, then
$\sum |d_{l,\sigma}| < n+1$, and thus $\sum |d_{l,\sigma}|\le n$.

These balls have all the desired properties (the bound \eqref{balls3} on the sum of the
radii follows from \eqref{radii.bd}) except that they cover $Z_E$ rather than all of $Z$.
To rectify this, recall from Lemma \ref{Zballs} that $Z\setminus Z_E$ can be covered by a finite collection
of balls whose radii sum to at most $C \e \logeps$. We can add these balls
to $\mathcal B^\sigma$, merging as necessary to obtain a pairwise disjoint collection
(still denoted $\mathcal B^\sigma$) that covers all of $Z$, and
with the sum of the radii increased by at most $C\e\logeps$. Since 
$\frac T{1-q}<n+1$, it follows from \eqref{radii.bd} that these still satisfy
$\sum r_{l,\sigma} < \sigma(n+1)$ for $0<\e<\e_0$. The bound on the total
degrees \eqref{balls2} and the energy lower bound
\eqref{balls4} for this modified collection of balls are directly inherited from the previous
collection.

\end{proof}

\paragraph{\bf Acknowledgment.}  R.I. acknowledges partial support by the ANR project ANR-14-CE25-0009-01. The work of R.J. was partially supported by the Natural Sciences and Engineering Research Council of Canada under operating Grant 261955.

\begin{center}

\bigskip

\bigskip

{\scshape Radu Ignat}\\
{\footnotesize
Institut de Math\'ematiques de Toulouse \& Institut Universitaire de France \\
UMR 5219, Universit\'e de Toulouse, CNRS, UPS, IMT \\
F-31062 Toulouse Cedex 9, France\\
\email{radu.ignat@math.univ-toulouse.fr}
}
\bigskip

\bigskip

{\scshape Robert L. Jerrard}\\
{\footnotesize
 Department of Mathematics\\
University of Toronto \\
Toronto, Ontario, M5S 2E4\\
\email{rjerrard@math.toronto.edu}
}

\end{center}

\end{document}